\documentclass[12pt]{amsart}
\pdfoutput=1
\usepackage[utf8]{inputenc}
\usepackage[usenames,dvipsnames]{xcolor}
\usepackage{tikz}
\usepackage{subfigure}
\usepackage{fancyhdr}
\usepackage{txfonts}
\usepackage{graphicx}
\usepackage{epsfig}
\usepackage{mathrsfs}
\usepackage{amssymb}
\usepackage{latexsym}
\usepackage{amsmath}
\usepackage{amsfonts}
\usepackage{amsbsy}
\usepackage{amscd}
\usepackage{amsthm}
\usepackage{verbatim}
\usepackage{color,tikz}
\usetikzlibrary{calc}
\usetikzlibrary{decorations.pathreplacing}

\usepackage{amsfonts}
\usepackage{amsbsy,calc}
\usepackage{graphicx,ifthen,cite}

\newcommand{\cF}{{\mathcal F}}
\newcommand{\cG}{{\mathcal G}}
\newcommand{\cH}{{\mathcal H}}

\newcommand{\cL}{{\mathcal L}}

\newcommand{\cP}{{\mathcal P}}

\newcommand{\Z}{{\mathbb Z}}

\newcommand{\N}{{\mathbb N}}

\newcommand{\bomega}{\boldsymbol{\omega}}
\newcommand{\bbeta}{{\boldsymbol{\beta}}}
\newcommand{\balpha}{{\boldsymbol{\alpha}}}
\newcommand{\btau}{\boldsymbol{\tau}}
\newcommand{\bnu}{\boldsymbol{\nu}}

\def\be{{\bf e}}

\def\bu{{\bf u}}

\newtheorem{prop}{Proposition}[section]
\newtheorem{lem}[prop]{Lemma}

\newtheorem{defi}[prop]{Definition}

\newtheorem{theo}{Theorem}[section]

\newtheorem{exam}{Example}[section]
\newtheorem{rema}[prop]{Remark}

\newcommand{\rmnum}[1]{\romannumeral #1}
\newcommand{\bx}{{\pmb{x}}}
\newcommand{\by}{{\pmb{y}}}
\newcommand{\bz}{{\pmb{z}}}

 \def\bu{{\mathbf u}}

 \def\diag{{\rm diag}}

\def\bi{\mathbf{i}}

\def\be{\mathbf e}

\def\bu{\mathbf{u}}

\newif\ifdraft
\drafttrue
\ifdraft
\else
\fi
\numberwithin{equation}{section}

\begin{document}
\title[Conformal dimension of post-critically finite fractals]{Connected p.c.f self-similar fractals satisfying
 the angle separation condition
have conformal dimension $1$}
\author{Hui Rao} \address{Department of Mathematics and Statistics, Central China Normal University, Wuhan, 430079, China
} \email{hrao@mail.ccnu.edu.cn
 }

\author{Zhi-Ying Wen} \address{Department of Mathematical Sciences, Tsinghua University,  Beijing, 100084, China
} \email{wenzy@tsinghua.edu.cn
 }

\author{Qihan Yuan$^*$} \address{Department of Mathematics and Statistics, Yangtze Normal University, Chongqing, 408100, China
} \email{20240005@yznu.edu.cn
 }


\date{\today}


\thanks{* The correspondence author.}
\begin{abstract} Let $\cG$ be a graph-directed iterated function system (GIFS) satisfying the finite post-critical points
property (p.c.f), and assume that the attractors $(M_\ell)_{\ell\in\cL}$ are connected.
In this paper,  we use valuation pairs to construct new metrics on the attractor of $\cG$, which seems to be very natural but
yet is not seen in literature.  Surprisingly,  under these new metrics, the  GIFS  becomes  nearly self-similar.
Moreover, we show that, if the original  GIFS is  self-similar and satisfies an angle separation condition,  then
for every $\ell\in \cL$, $(M_\ell,d_\ell)$ is symmetrically equivalent to $(M_\ell,D_\ell)$, where
$d_\ell$ is the original metric and $D_\ell$ is the induced metric.
Consequently,  the conformal dimension of $M_\ell$ is $1$ for every $\ell\in \cL$.
\end{abstract}

\maketitle

\section{Introduction}\label{SECTION1}
Quasisymmetric mapping was introduced by Beurling and Ahlfors \cite{Beurling_1956} in 1956.
Let $(X,d_X)$ and $(Y,d_Y)$ be two metric spaces.
A homeomorphism $f:(X,d_X)\rightarrow(Y,d_Y)$ is said to be \emph{quasisymmetric}
if there exists an increasing homeomorphism $\eta$ of $[0,\infty)$ to itself such that
\begin{equation}\label{def_quasisymmetric}
 \frac{d_Y(f(x),f(y))}{d_Y(f(x),f(z))}\leq\eta\left (\frac{d_X(x,y)}{d_X(x,z)} \right )
\end{equation}
for all $x,y,z\in X$; in this case we say that $(X,d_X)$ and $(Y,d_Y)$ are \emph{quasisymmetrically equivalent}.
In other words, a homeomorphism is quasisymmetric if
it distorts relative distances in a uniform and scale invariant fashion, see \cite{Heinonen_2001}.

The conformal dimension, as one of the most important quasisymmetry invariants, was introduced by Pansu \cite{Pansu_1989} in 1989.
 The \emph{conformal dimension} of $(X,d_X)$, denoted by  $\dim_C(X,d_X)$,
is defined as
$$\dim_C(X,d_X)=\inf\left\{\dim_H(Y,d_Y); (Y,d_Y) \ \text{is quasisymmetrically equivalent to} \ (X,d_X) \right\}.$$

 Kovalev \cite{Kovalev_2006} proved that    conformal dimension does not take values in $(0,1)$. When $\dim_H X=1$,
 both $\dim_C X=0$ and $\dim_C X=1$ can occur according to $X$, see Staples and Ward \cite{{Staples_1998}} and Tukia \cite{Tukia_1989}.
 H. Hakobyan \cite{Hakobyan_2010} proved that a large class of Cantor sets on the real line have
 conformal dimension $1$, see also Wang and Wen \cite{Wang_2014}.
  M. Carrasco Piaggio \cite{Carrasco_2013} studied the Ahlfors regular conformal gauge and dimension of a compact metric space.

Recently, there are many works  devoted to the study of the conformal dimension of self-similar sets and self-affine sets. \cite{TW_2006,Bishop_2001,Kigami_2014,Dang_2021,Kovalev_2006,Hakobyan_2010,Wang_2014, Carrasco_2013}.
If a self-similar set is totally disconnected and satisfies the open set condition, then its conformal dimension is zero.
  Bishop and Tyson \cite{Bishop_2001}, Tyson and Wu \cite{TW_2006}, Dang and  Wen \cite{Dang_2021} proved that several classes of p.c.f. self-similar sets, including the Sierpi$\acute{\text{n}}$ski gasket, have conformal dimension  $1$.
The  conformal dimension of the Sierpi$\acute{\text{n}}$ski carpet is still unknown, some estimates can be found in   J. Kigami \cite{Kigami_2014}  and  Kwapisz \cite{Kwapisz_2020}.
  Binder, Hakobyan,
and Li \cite{Binder_2025} proved the minimality of a class of
Bedford--McMullen carpets with uniform fibers.
Y. F. Zhang and S. Q. Zhang \cite{Zhang_2025} characterized when a self-affine sponge of Gatzouras-Lalley type has conformal dimension $0$.


\begin{figure}[h]
\centering
\begin{minipage}{.5\textwidth}
\centering
    \begin{tikzpicture}
    \draw[xscale=4,yscale=4](-1, 0)--(0,0.5)--(1,0)--(0,-0.5)--cycle;
    \draw[xscale=2,yscale=2][shift ={(-1,0)}][fill=red!50](-1, 0)--(0,0.5)--(1,0)--(0,-0.5)--cycle;
    \draw[xscale=2,yscale=2][shift ={(1,0)}][fill=red!50](-1, 0)--(0,0.5)--(1,0)--(0,-0.5)--cycle;
    \draw[xscale=1,yscale=1][rotate =90][shift ={(1,0)}][fill=red!50](-1, 0)--(0,0.5)--(1,0)--(0,-0.5)--cycle;
    \draw[xscale=1,yscale=1][rotate =-90][shift ={(1,0)}][fill=red!50](-1, 0)--(0,0.5)--(1,0)--(0,-0.5)--cycle;
    \draw(-4.2,0)node{\footnotesize $a_1$};
    \draw(4.2,0)node{\footnotesize $a_3$};
    \draw(0,-2.2)node{\footnotesize $a_2$};
    \draw(0,2.2)node{\footnotesize $a_4$};
    \end{tikzpicture}
\end{minipage}
\caption{Dang and Wen \cite{Dang_2021} shows that the above self-similar set  has conformal dimension $1$. 
The antenna set in Bishop and Tyson \cite{Bishop_2001} is the upper half  of $K_\alpha$.}
\label{fig1}
\end{figure}

In this paper, we show that  a large class of  p.c.f. self-similar fractals have conformal dimension $1$.

\begin{theo}\label{main-4}
Let $\left((g_e)_{e\in \Sigma}, (M_\ell)_{\ell\in \cL}\right)$ be a  nearly self-similar p.c.f. GIFS such
that all $M_\ell$ are connected. If  the GIFS  satisfies the angle separation condition,
then $\dim_C M_\ell=1$.
\end{theo}

In the following, we explain the terminologies and techniques we used to prove the above theorem.

\subsection{GIFS  and nearly self-similarity}

An \emph{iterated function system (IFS)} is a family of contractions $F=\{f_i\}_{i=1}^N$ on a complete metric space $(X,d)$,
and the \emph{attractor} of it is the unique nonempty compact set $K$ satisfying $K=\bigcup_{i=1}^Nf_i(K)$. See Hutchinson \cite{Hutchinson_1981}. For simplicity, we will write an IFS as $\left( F=\{f_i\}_{i=1}^N,K\right)$ to indicate the attractor.

In this paper, we will deal with the graph-directed IFS, which  is a generalization of IFS.
 Let $({\mathcal L},\Sigma)$ be a directed graph with vertex set $\cL=\{1,\dots,L\}$ and directed edge set $\Sigma$. The set of edges from
 vertex $\ell$ to vertex $\delta$ is denoted by $\Sigma_{\ell,\delta}$,
 and we set
 $$\Sigma_\ell=\bigcup_{\delta\in\cL}\Sigma_{\ell, \delta}$$
 to be the collection of edges aneminating from $\ell$. From now on, we shall call $\ell\in \cL$ a \emph{state} instead of vertex to avoid confusion.

\begin{rema}\rm We require  that for any $\ell\in \cL$, there is at least one edge starting from vertex $\ell$; otherwise, this state will not be involved in the iteration process. We emphasize here that we do not put any other restriction on the graph!
\end{rema}

 Let $\left(X_\ell,d_\ell\right)_{\ell\in \cL}$ be a sequence of  compact metric spaces.
 For each edge $e\in\Sigma_{\ell, \delta}$, let $G(e):~X_\delta\to X_\ell$ be a contraction map.
 For simplicity, we write $G(e)=g_e$. We call
$$G_\Sigma=(g_e)_{e\in \Sigma}$$
a \emph{graph-drected IFS (GIFS)}.   See \cite{Mauldin_1988}.
It is well known that there exists a  unique vector of non-empty compact sets
$(M_\ell)_{\ell\in \cL}$ such that  $M_\ell\subset X_\ell$ and
\begin{equation}\label{def:attractor}
M_\ell=\bigcup\limits_{\delta\in\cL} \bigcup\limits_{e\in\Sigma_{\ell,\delta}}g_e(M_\delta),\quad   \ell  \in \cL.
\end{equation}
We call $(M_\ell)_{\ell\in \cL}$ the  \emph{attractors} of the GIFS $(g_e)_{e\in\Sigma}$.

In the following, we will denote a GIFS by
$$
\left((g_e)_{e\in \Sigma},(M_\ell)_{\ell\in \cL}\right) \text { or } (G_\Sigma, M_{\cL}).
$$
Sometimes, we also use the notation $((g_e)_{e\in \Sigma},(M_\ell, d_\ell)_{\ell\in \cL})$ if several metrics are involved in
the sets $(M_\ell)_{\ell\in \cL}$.

\paragraph{\textbf{Symbolic spaces related to a GIFS}} Let $(G_\Sigma,M_\cL)$ be a GIFS.
For $k\geq 1$, we use $\Sigma_\ell^k$ to denote the collection of directed paths in the graph emanating from the state $\ell$ and with length $k$.
 We denote $\Sigma^k=\bigcup_{\ell\in \cL} \Sigma_\ell^k$, $\Sigma^*_{\ell}=\bigcup_{k=1}^\infty \Sigma^{k}_{\ell}$ and $\Sigma^*=\bigcup_{k=1}^\infty \Sigma^{k}$.
 For $e_1\dots e_k\in \Sigma^k$, we define $t(e_1\dots e_k)=t(e_k)$ to be the terminus state of $e_k$,   set
\begin{equation}
M_{e_1\dots e_k}=g_{e_1\dots e_k}\left(M_{t(e_k)}\right)
\end{equation}
and we call it a \emph{$k$-th cylinder}.


Let $\Sigma_\ell^\infty$ denote the collection of infinite directed paths emanating from the state $\ell$.
Let $(e_j)_{j\geq 1}\in\Sigma_\ell^\infty$ and let $x$ be the point obtained by
$
\{x\}=\bigcap_{k=1}^\infty M_{e_1\dots e_k}.
$
We define $\pi_\ell((e_j)_{j\geq 1})=x$, and we say $(e_j)_{j\geq 1}$ is a \emph{coding of $x$ in $M_\ell$}.
Denote  $\Sigma^\infty=\bigcup_{\ell\in \cL} \Sigma_\ell^\infty$.

 For $\pmb{i},\pmb{j}\in\Sigma^\ast$, we denote by $|\pmb{i}|$ the length of $\pmb{i}$,
and use $\pmb{i}\wedge\pmb{j}$ to denote the maximal common prefix of $\pmb{i}$ and $\pmb{j}$. For a sequence $\bx=(x_j)_{j\geq 1}$, we denote $\bx|_m=x_1\dots x_m$.

A function $r: \Sigma\to (0,1)$ will be called a \emph{ratio function}. We define
  $r_{e_1\dots e_k}=r_{e_1}\cdots r_{e_k}$.


A mapping $f:(X,d)\rightarrow(X,d)$ is said to be  a \emph{similitude}, if there exists $0<c<1$ such that
 $d(f(x),f(y))=cd(x,y)$
holds   for all $x,y\in X$.
A  GIFS is called   \emph{self-similar}  if all mappings in the GIFS are similitudes.

To calculate the conformal dimension of a p.c.f. fractal $M_\ell$, our strategy
is to construct new metrics on $M_\ell$ by valuation pairs. However, under these new metrics, although the GIFS is usually
 not self-similar, it is  `nearly self-similar'.

 \begin{defi}[Nearly self-similar]\label{def:nearly-self-similar}\textnormal{We say a GIFS $((g_e)_{e\in \Sigma},(M_\ell)_{\ell\in \cL})$ is   \emph{nearly self-similar}, if there exists
a ratio function $r:\Sigma\to (0,1)$ and a positive constant $C_0$ such that
$$C_0^{-1}r_\be d_{t(\be)}(x,y)\leq { d}_\ell(g_{\be}(x),g_{\be}(y))\leq C_0r_\be d_{t(\be)}(x,y), \quad \forall   \be\in\Sigma^\ast_{\ell}, \ \forall x,y\in M_{t(\be)}.$$
We call $C_0$ a \emph{deformation bound}, and call $r_e$ the \emph{contraction ratio} of $g_e$.
}
\end{defi}

  \subsection{P.c.f. GIFS and angle-separation condition}

Let $\left (G_\Sigma,M_\cL \right )$ be a GIFS.
As in the IFS case, we can define   the  critical  point sets $(C_\ell)_{\ell\in \cL}$, the post-critical point
sets $(P_\ell)_{\ell\in \cL}$ and the post-critical coding sets  $(\cP_\ell)_{\ell\in \cL}$, see Section \ref{sec:def-p.c.f.}.

 A GIFS $\left (G_\Sigma,M_\cL \right )$ is said to be \emph{post-critically finite (p.c.f.)}  if
$\bigcup_{\ell\in \cL}\cP_\ell$ is a finite set.

For a p.c.f. fractal, we define  the angle separation condition, which is stronger than the  well-known open set condition.

 \begin{defi}[Angle separation condition]
\emph{A  p.c.f. GIFS $(G_\Sigma, M_{\cL})$ is said to satisfy the \emph{angle separation condition}, if there exists a constant $\theta_0>0$, such that if $\bomega, \bnu\in\Sigma^*_{\ell}$ satisfy $M_{\bomega}\cap M_{\bnu}=\{z\}$,  then
$$d_\ell(x,y)\geq \theta_0\max\{d_\ell(x,z),d_\ell(y,z)\}, \quad \text{for any} \ x\in M_{\bomega}, y\in M_{\bnu}.$$
}
\end{defi}

  A stronger version of the angle separation condition  has been used by  Zhu and Yang \cite{Zhu_2018} in the Euclidean space.


\subsection{Criterion for quasisymmetric equivalence}
The set $P_\ell$ can be regarded as the vertex set of $M_\ell$. Let $P_{\be}=g_{\be}(P_{t(\be)})$
and we call it the \emph{vertex set} of $M_\be$. It is seen that if two cylinders intersect, then they intersect
at their common vertices.  We define
\begin{equation}
\partial\Sigma_\ell =\left\{e\in\Sigma_{\ell};~M_{e}\cap P_{\ell}\neq\emptyset\right\}
\end{equation}
 and set $\partial \Sigma=\bigcup_{\ell\in \cL}\partial\Sigma_{\ell}.$
Our first main result is

\begin{theo}\label{main-1-GIFS} Let   $((g_e)_{e\in \Sigma},(M_\ell, d_\ell)_{\ell\in \cL})$ and $((g_e)_{e\in \Sigma},(M_\ell, d'_\ell)_{\ell\in \cL})$  be two p.c.f. GIFS.
Suppose    both systems  are  nearly self-similar  and satisfy the angle separation condition.
Let $r_e$ and $r'_e$ be the  contraction ratio of $g_e$ in
$d$ and $d'$,  respectively.
 If there exists $s>0$ such that
$$r'_e=(r_e)^s \quad \text{for every} \ e\in \partial \Sigma,$$
then $(M_\ell,d_\ell)$ and $(M_\ell,d_\ell')$ are quasisymmetrically equivalent.
\end{theo}

\begin{rema}
 \emph{  Theorem \ref{main-1-GIFS} considerably improves the result of Tyson and Wu \cite{TW_2006}, where they require
$r_i=r$ for all $i\in \partial \Sigma$,  and $r'_i=r'$ for all $i\in \partial \Sigma$.
}
\end{rema}

  Let $(G_\Sigma, M_\cL)$ be a GIFS. We say a state $\ell\in \cL$ is   \emph{isolated}  if there is no edge in $\Sigma$ that terminates at $\ell$.
Clearly a state $\ell$ is isolated if and only if $P_\ell=\emptyset$.
We emphasize that in Theorem \ref{main-1-GIFS}, we allow the existence of isolated states, and this makes Theorem \ref{main-1-GIFS}
a very powerful tool.

\subsection{Metrics induced by valuation pair} We shall use a valuation pair to construct new metrics on attractors of a p.c.f. GIFS, which seems to be a very natural idea.
We use IFS to explain our idea. Let $\left(\{f_j\}_{j=1}^N,K\right)$ be a p.c.f. IFS such that $K$ is connected.

First,  we assign each pair $(p,p')\in P^2$
a distance, which we denote by $\tau(p,p')$; namely,  $\tau$ is a distance on $P$. We may regard the metric space $(P,\tau)$
as a complete graph with vertex set $P$, and regard $\tau(p,p')$ as the weight of the edge $[p,p']$. Let us denote this graph
by $(P, \Lambda)$.

Next, we assign each cylinder a ratio, which we denote by $R:\Sigma\to (0,1)$.   We call $(R,\tau)$ a \emph{valuation pair}.

\begin{figure}[h]
  \centering
  \subfigure[]{
  \includegraphics[width=6 cm]{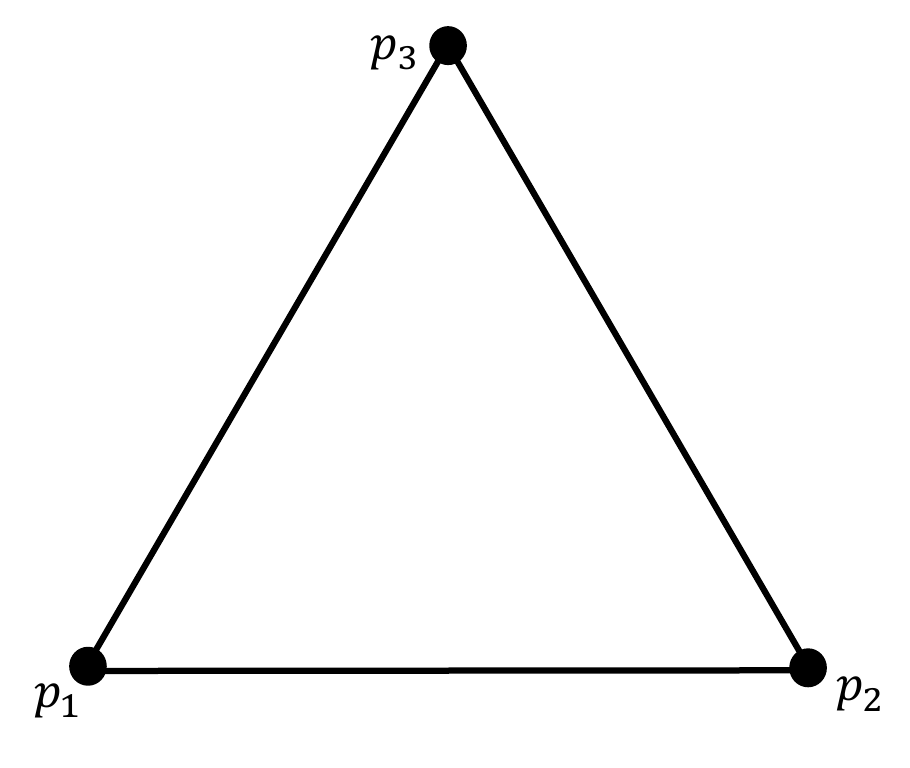}
  }
  \subfigure[]{
  \includegraphics[width=6 cm]{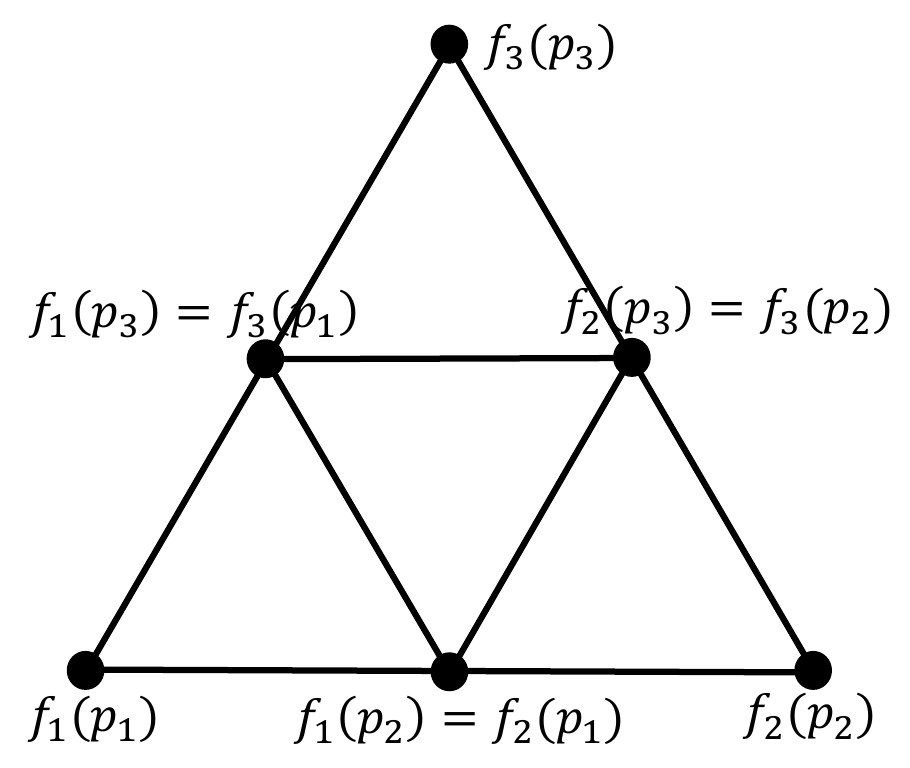}
  }\\
  \caption{(a) complete graph $(P, \Lambda)$. (b) complete graph $(P^*, \Lambda^*)$.}
  \label{complete graph}
\end{figure}

Now we define $f_j(P, \Lambda)=(f_j(P), f_j(\Lambda))$ to be the complete graph with vertex set $f_j(P)$. Set
$$
(P^*, \Lambda^*)=\left (\bigcup_{j=1}^N f_j(P), \bigcup_{j=1}^N f_j(\Lambda)\right )
$$
to be the union of the small graphs. See Figure \ref{complete graph}.
 We assign a weight $R(j) \tau(p,p')$ to the edge  $[f_j(p), f_j(p'])$, and denote it by
$$\varphi\left([f_j(p), f_j(p')]\right)=R(j)\tau(p,p').$$
  Let $L$ be a path on $(P^*, \Lambda^*)$   joining $p$ and $p'$, we define the weight of $L$ to be the sum
of weights of  edges in it. If $\tau(p,p')$ coincides with the minimal weight of paths joining $p$ and $p'$,
then we say $(R,\tau)$ satisfies the \emph{consistency condition} (see Definition \ref{def:consistency}).

\begin{theo}\label{thm:R-tau} Let   $(G_\Sigma, M_{\cL})$
  be a p.c.f. GIFS such that $M_\ell$ are connected.
If $(R,\tau)$ is a valuation pair  satisfying the consistency condition, then there exist  metrics   $(D_\ell)_{\ell\in \cL}$  on $(M_\ell)_{\ell\in \cL}$ such that

(i) The GIFS  $(G_\Sigma, M_{\cL}, D_{\cL})$ is nearly self-similar when we set $R(e)$ to be the  ratio of $g_e$.

(ii) The GIFS $(G_\Sigma, M_{\cL}, D_{\cL})$ satisfies the  angle separation condition.
\end{theo}


\begin{exam}[Conformal dimension of Sierpi$\acute{\text{n}}$ski gasket]\label{exam:gasket}
 \rm
 The Sierpi$\acute{\text{n}}$ski gasket $S$
is  the attractor of the IFS
 $F=\left\{f_1(z)={z}/{2}, \   f_2(z)={z+1}/{2}, \  f_3(z)={z+\frac{1}{2}+\frac{\sqrt{3}}{2}i}/{2}\right\}.$
 Tyson and Wu \cite{TW_2006} showed the conformal dimension of the Sierpi$\acute{\text{n}}$ski gasket  is $1$.
 Using Theorem \ref{main-1-GIFS}  and Theorem \ref{thm:R-tau}, we can give a quick proof of this fact.

 Tyson and Wu define a sequence of IFS $F_m$ as follows.
 Let $F_0=F$. Suppose $F_{m-1}$ is defined.
 Then $F_m$ is obtained by replacing every cylinder $K_I$ in a corner by the three sub-cylinders $K_{I1}, K_{I2}$ and $K_{I3}$.
 We call $F_m$  the \emph{m-th corner evolution} of $F$. Figure \ref{fig:corner evolution} illustrates $F_1$ and $F_2$.

For $F_2$, we set $R(i)\equiv 1/6$ and $\tau(p,p')\equiv 1$. Clearly, to join
two corner points, we need at least $6$ edges.
This proves that  $(R,\tau)$ is consistent.   (See Figure \ref{fig:corner evolution}(b).)
By Theorem \ref{thm:R-tau}, there is a metric $D_2$ on $K$ such that $F_2$ is nearly self-similar and the contraction ratios are $1/6$.
 Then
 by Theorem \ref{main-1-GIFS},  $(K,|\cdot|)$ is quasisymmetrically equivalent to
 $(K,D_2)$. It follows that $\dim_CK\leq \dim_H(K, D_2)=\log(15)/\log(6)$.
 Consider the $m$-th corner evolution $F_m$, one can obtain that
 $$\dim_CK\leq \dim_H(K, D_m)=\log(3+6m)/\log(2+2m).$$
  Therefore,  $\dim_C K=1$.

Next, we explain the proof in \cite{TW_2006}.
In   \cite{TW_2006}, they construct a sequence of new  IFS ${G}_m$, $m\geq 1$, see  Figure \ref{fig3} for ${G}_1$ and ${G}_2$.
 Clearly $G_2$ induces a metric on $K$ by the coding map, which we denote by $\rho_2$.
 Since $(K, \rho_2)$ and $(G_2, |\cdot|)$ are isometric, we conclude that $(K,\rho_2)$ is self-similar and satisfies
 the angle separation condition. Notice that $\rho_2$ is defined by a Euclidean geometric consideration, and in general, it does not work.
 Again, by Theorem \ref{main-1-GIFS}, $(K,|\cdot|)$ and $(K,\rho_2)$ are quasisymetrically
 equivalent.  So   $\dim_C K\leq \dim_H (K, \rho_2)=\dim_H G_2$.
   $\Box$
\end{exam}

\begin{figure}[h]
\centering
\begin{minipage}[t]{.49\textwidth}
\centering
\begin{tikzpicture}[xscale=4.5,yscale=4.5]
    \draw[shift ={(0,0)}, scale =0.3333][fill=red!50](0,0)--(0.5,-0.866)--(-0.5,-0.866)--cycle;
    \draw[shift ={(-0.1666,-0.2886)}, scale =0.19245][fill=red!50](0,0)--(0,-1)--(0.866,-0.5)--cycle;
    \draw[shift ={(0.1666,-0.2886)}, scale =0.19245][fill=red!50](0,0)--(0,-1)--(-0.866,-0.5)--cycle;

    \draw[shift ={(-0.3334,-0.5774)}, scale =0.19245][fill=red!50](0,0)--(0.866,0.5)--(0.866,-0.5)--cycle;
    \draw[shift ={(-0.1667,-0.866)}, scale =0.19245][fill=red!50](0,0)--(0,1)--(0.866,0.5)--cycle;
    \draw[shift ={(-0.3334,-0.5774)}, scale =0.3333][fill=red!50](0,0)--(0.5,-0.866)--(-0.5,-0.866)--cycle;

    \draw[shift ={(0.3334,-0.5774)}, scale =0.19245][fill=red!50](0,0)--(-0.866,0.5)--(-0.866,-0.5)--cycle;
    \draw[shift ={(0.1667,-0.866)}, scale =0.19245][fill=red!50](0,0)--(0,1)--(-0.866,0.5)--cycle;
    \draw[shift ={(0.3334,-0.5774)}, scale =0.3333][fill=red!50](0,0)--(0.5,-0.866)--(-0.5,-0.866)--cycle;
    \end{tikzpicture}%
\end{minipage}
\hfill
\begin{minipage}[t]{.5\textwidth}
\centering
\begin{tikzpicture}[xscale=4.5,yscale=4.5]
    \draw[shift ={(0,0)}, scale =0.25][fill=red!50](0,0)--(0.5,-0.866)--(-0.5,-0.866)--cycle;
    \draw[shift ={(-0.125,-0.2165)}, scale =0.1443][fill=red!50](0,0)--(0,-1)--(0.866,-0.5)--cycle;
    \draw[shift ={(-0.125,-0.3608)}, scale =0.1443][fill=red!50](0,0)--(0,-1)--(0.866,-0.5)--cycle;
    \draw[shift ={(0.125,-0.2165)}, scale =0.1443][fill=red!50](0,0)--(0,-1)--(-0.866,-0.5)--cycle;
    \draw[shift ={(0.125,-0.3608)}, scale =0.1443][fill=red!50](0,0)--(0,-1)--(-0.866,-0.5)--cycle;

    \draw[shift ={(-0.375,-0.6495)}, scale =0.25][fill=red!50](0,0)--(0.5,-0.866)--(-0.5,-0.866)--cycle;
    \draw[shift ={(-0.375,-0.6495)}, scale =0.1443][fill=red!50](0,0)--(0.866,0.5)--(0.866,-0.5)--cycle;
    \draw[shift ={(-0.25,-0.866)}, scale =0.1443][fill=red!50](0,0)--(0,1)--(0.866,0.5)--cycle;
    \draw[shift ={(-0.25,-0.57735)}, scale =0.1443][fill=red!50](0,0)--(0.866,0.5)--(0.866,-0.5)--cycle;
    \draw[shift ={(-0.125,-0.79385)}, scale =0.1443][fill=red!50](0,0)--(0,1)--(0.866,0.5)--cycle;

    \draw[shift ={(0.375,-0.6495)}, scale =0.25][fill=red!50](0,0)--(-0.5,-0.866)--(0.5,-0.866)--cycle;
    \draw[shift ={(0.375,-0.6495)}, scale =0.1443][fill=red!50](0,0)--(-0.866,0.5)--(-0.866,-0.5)--cycle;
    \draw[shift ={(0.25,-0.866)}, scale =0.1443][fill=red!50](0,0)--(0,1)--(-0.866,0.5)--cycle;
    \draw[shift ={(0.25,-0.57735)}, scale =0.1443][fill=red!50](0,0)--(-0.866,0.5)--(-0.866,-0.5)--cycle;
    \draw[shift ={(0.125,-0.79385)}, scale =0.1443][fill=red!50](0,0)--(0,1)--(-0.866,0.5)--cycle;
    \end{tikzpicture}
\end{minipage}%
\caption{The IFS $G_1$ and $G_2$ in  Tyson and Wu \cite{TW_2006}. }
\label{fig3}
\end{figure}

\subsection{Spanning IFS}
As we have seen, it is an easy matter to construct consistent valuation pair   for the Sierpinski gasket.
But to do it for general p.c.f. GIFS, it is far from trivial.
To this end,  we introduce a notion of spanning IFS.

Again we use IFS to explain our idea.  Let $\left(\{f_j\}_{j=1}^N,K\right)$ be a p.c.f. IFS such that $K$ is connected.
Let $\Gamma\subset K$ be a dendrite. We call $a\in \Gamma$   an \emph{end point} of $\Gamma$ if $\Gamma\setminus\{a\}$
 is connected; we denote the set of end points of $\Gamma$ by ${\mathcal E}(\Gamma)$.
 We denote the ramification points of $\Gamma$ by ${\mathcal R}_0(\Gamma)$.

 We call $\Gamma$ a \emph{spanning dendrite} if   $P_F\subset {\mathcal E}(\Gamma)$; in this case, we  call the triple $(F, K, \Gamma)$ a \emph{spanning IFS}.  It is easy to show that  $\left(\{f_j\}_{j=1}^N,K\right)$ admits spanning dendrites
if and only if $K\setminus P_F$ is connected (Theorem \ref{theo:span}).
Figure \ref{fig:Sier1}(a) illustrates  a spanning dendrite of the Sierpi$\acute{\text{n}}$ski gasket.

To construct a consistent valuation pair $(R,\tau)$, we   start with  a spanning IFS $(F_\Sigma, K, \Gamma)$.
First, we iterate the IFS properly so that the new system satisfies several separation conditions related to
 the dendrite $\Gamma$.
Secondly, we select all the cylinders intersecting the dendrite $\Gamma$, and denote their union by $H_0$.
Then we remove some of the cylinders to obtain a minimal set $H$ in the sense that $H$ is connected and contains
$P$ as a subset.
  Thirdly,  we classify the cylinders in $H$ into three classes:

  $\bullet$ the village class is the collection of cylinders intersecting ${\mathcal R}_0(\Gamma)$;

  $\bullet$ the bright corner class is the collection of cylinders intersecting $P$;

  $\bullet$ the ess-covering class: otherwise.

The cylinders that are not in $H$ are called \emph{dark cylinders}. We intend to set $\tau(p,p')$ to be bigger than $2$ and less than $3$.
We set the ratios in dark cylinders large, so that the minimal paths will always stay in $H$.
We show that the bright area $H$ has a lot of cut points, so the candidates of the minimal paths are very limited.
Finally, by solving linear equations, we obtain $\tau$ and $R$, and that the paths we prepared are minimal
in this valuation pair.  This is done in
Section \ref{SECTION6} to \ref{SECTION11}.

We show that provided that $M_\ell\setminus P_\ell$ are connected, we can find a pair $(R,\tau)$ which is consistent,
and under the induced metrics $(D_\ell)_{\ell\in \cL}$, the GIFS satisfies the conditions in Theorem \ref{main-1-GIFS}.
To decrease the Hausdorff dimension of $(M_\ell, D_\ell)$, we borrow a technique
of corner evolution from  Tyson and Wu in \cite{TW_2006}.

\subsection{The case $M_\ell\setminus P_\ell$ is disconnected for some $\ell\in \cL$}
There do exist p.c.f. GIFS such that $M_\ell\setminus P_\ell$ is disconnected for some $\ell\in \cL$, see
Figure \ref{fig:christmas}. To overcome this difficulty,  we show that the connected components, after taking a closure,
 are the attractors of a new GIFS,  and the new system always admits spanning dendrites. (See Theorem  \ref{thm:transfer-GIFS}).

\begin{figure}[h]
  \centering
  \includegraphics[width=4.6 cm]{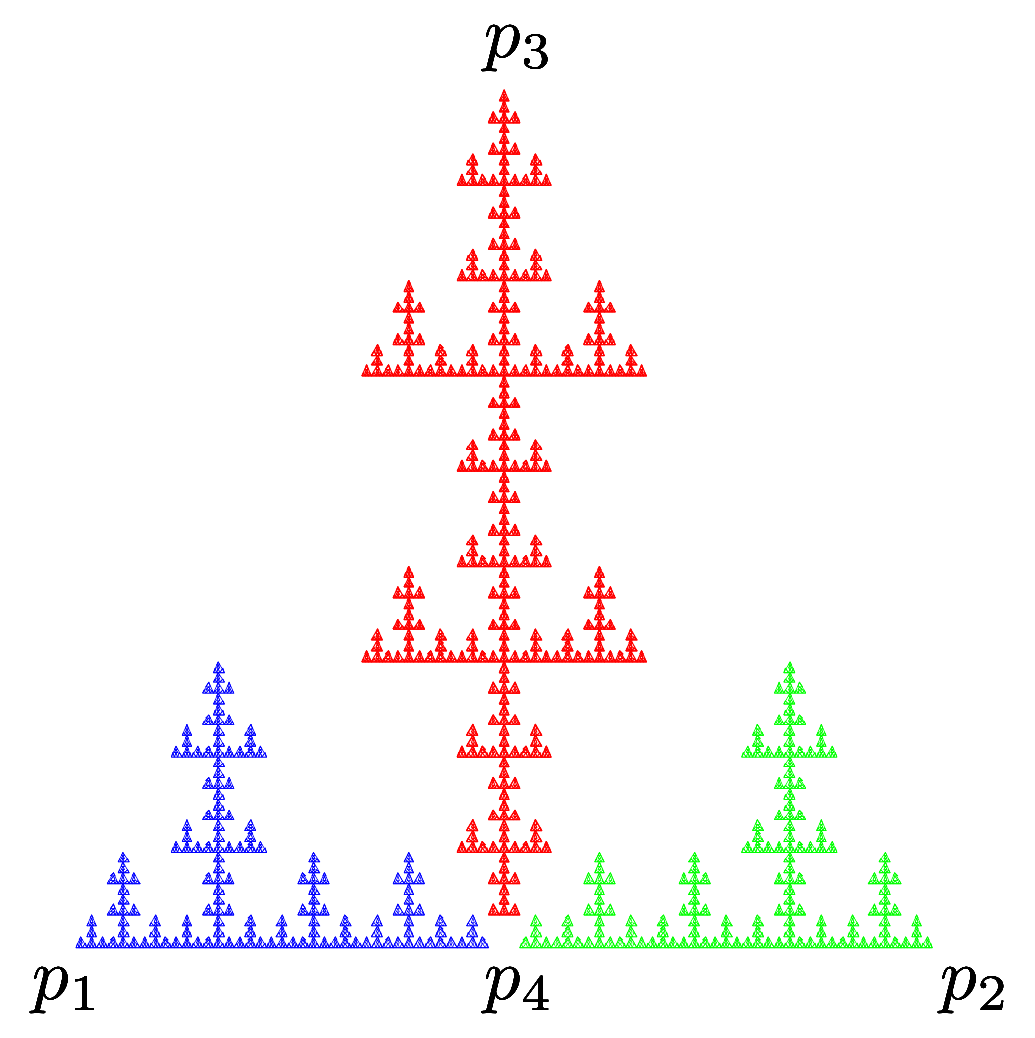}\\
   \caption{The Christmas tree has four post-critical points, and $K\setminus P$ has three components. }
   \label{fig:christmas}
\end{figure}


Wen and Xi \cite{Wen_2003} constructed a self-similar arc which does not   satisfy the angle separation condition.


%

\begin{figure}[h]
\centering
\begin{minipage}{\textwidth}
\centering
    \begin{tikzpicture}
    \draw[xscale=3,yscale=3](0,0)--(4,0)--(2,1)--cycle;

    \draw[xscale=1.3,yscale=1.3][fill=red!50](0,0)--(4,0)--(2,1)--cycle;
    \draw[xscale=1.1,yscale=1.1][shift ={(6.915,0)}][fill=red!50](0,0)--(4,0)--(2,1)--cycle;

    \draw[xscale=0.276,yscale=0.276][shift ={(18.837,0)},rotate =75][fill=red!50](0,0)--(4,0)--(2,1)--cycle;
    \draw[xscale=0.5,yscale=0.5][shift ={(10.95,2.135)},rotate =75][fill=red!50](0,0)--(4,0)--(2,1)--cycle;

    \draw[xscale=0.776,yscale=0.776][shift ={(7.732,3.866)},rotate =285][fill=red!50](0,0)--(4,0)--(2,1)--cycle;

    \draw[xscale=0.3,yscale=0.3][shift ={(22.67,0)},rotate =52][fill=red!50](0,0)--(4,0)--(2,1)--cycle;
    \draw[xscale=0.238,yscale=0.238][shift ={(31.65,4)},rotate =273][fill=red!50](0,0)--(4,0)--(2,1)--cycle;

    \draw(0,0)[fill=black]circle (0.02);\draw(0,-0.2)node{\tiny $A_0$};
    \draw(5.199,0)[fill=black]circle (0.02);\draw(5.199,-0.2)node{\tiny $A_1$};
    \draw(5.475,1.0675)[fill=black]circle (0.02);\draw(5.7,1.0675)node{\tiny $A_2$};
    \draw(6,3)[fill=black]circle (0.02);\draw(6,3.2)node{\tiny $A_3$};
    \draw(6.8,0)[fill=black]circle (0.02);\draw(6.8,-0.2)node{\tiny $A_4$};
    \draw(7.535,0.95)[fill=black]circle (0.02);\draw(7.578,1.0675)node{\tiny $A_5$};
    \draw(7.578,0)[fill=black]circle (0.02);\draw(7.578,-0.2)node{\tiny $A_6$};
    \draw(12,0)[fill=black]circle (0.02);\draw(12,-0.2)node{\tiny $A_7$};
    \draw(7.15,1.685)[fill=black]circle (0.02);\draw(7.3,1.7)node{\tiny $B$};
    \draw(6.93,0.65)[fill=black]circle (0.02);\draw(7.1,0.87)node{\tiny $C$};

    \draw(2.5,0.6)node{\tiny $K_0$}; \draw(5.3,0.5)node{\tiny $K_1$};
    \draw(5.5,2.1)node{\tiny $K_2$}; \draw(6.7,1.5)node{\tiny $K_3$};
    \draw(7.15,0.5)node{\tiny $K_4$}; \draw(7.8,0.5)node{\tiny $K_5$}; \draw(10,0.5)node{\tiny $K_6$};

    \draw (0.5 ,0) arc (0:26:0.5); \draw(0.7,0.15)node{\tiny $\theta$};
    \end{tikzpicture}
\end{minipage}
\caption{ A self-similar arc   not   satisfying the angle separation condition (\cite{Wen_2003}).}
\label{figWen}
\end{figure}


{\color{red}The paper is organized as follows. In Section \ref{SECTION2}, we recall the notations and definitions related to p.c.f. GIFS.
We prove Theorem \ref{main-1-GIFS} in Section \ref{SECTION3}.
Section \ref{SECTION4} deals with the valuation pair and Theorem \ref{thm:R-tau} is proved there.
In Section \ref{SECTION5}, we introduce the notion of spanning dendrites. 
Section \ref{SECTION6} is devoted to the classification of cylinders of a spanning system and the construction of villages and primary arcs.
In Section \ref{SECTION7}, we study the bright area, induced chains, and the corner evolution of a spanning system. 
Section \ref{SECTION8} is devoted to $\tau$-functions on $P_{\cL}$. 
In Section \ref{SECTION9}, we introduce action arcs and their evolutions, in order to produce primary chains. 
Section \ref{SECTION10} construct $(R,\tau)$ for a spanning system. 
Section \ref{SECTION11} verifies the consistency condition. 
In Section \ref{SECTION12}, we prove the main theorem under the assumption that $M_\ell\setminus P_\ell$ are connected. 
Finally, Section \ref{SECTION13} treats the general GIFS case and completes the proof of Theorem \ref{main-4}.}

\section{{Graph-directed} IFS with p.c.f. property}\label{SECTION2}

 In this section, we recall the definitions and notations related to p.c.f. GIFS.

\subsection{Graph-directed IFS}

 Let $G=({\mathcal L},\Sigma)$ be a directed graph with vertex set $\cL=\{1,\dots,L\}$ and directed edge set $\Sigma$. The set of edges from
 vertex $\ell$ to vertex $\delta$ is denoted by $\Sigma_{\ell,\delta}$,
 and $\Sigma_\ell=\bigcup_{\delta\in\cL}\Sigma_{\ell, \delta}$.
 We assume that for any $\ell\in \cL$, there is at least one edge starting from vertex $\ell$.

 Let $(X_\ell,d_\ell)_{\ell\in \cL}$ be a sequence of  compact metric spaces.
 For each edge $e\in\Sigma_{\ell, \delta}$, let $G(e):~X_\delta\to X_\ell$ be a corresponding contraction.
 For simplicity, we write $G(e)=g_e$. We call
$$G_\Sigma=(g_e)_{e\in \Sigma}$$
a \emph{graph-drected IFS (GIFS)}.   See \cite{Mauldin_1988}.

From now on, we shall call $\ell\in \cL$ a \emph{state} instead of vertex to avoid confusion.


\begin{theo}\label{thm:attractor}
  There exist unique non-empty compact sets
$(M_\ell)_{\ell\in \cL}$ such that  $M_\ell\subset X_\ell$ and
\begin{equation}\label{def:attractor}
M_\ell=\bigcup\limits_{\delta\in\cL} \bigcup\limits_{e\in\Sigma_{\ell,\delta}}g_e(M_\delta),\quad   \ell  \in \cL.
\end{equation}
We call $(M_\ell)_{\ell\in \cL}$ the  \emph{attractors} of the GIFS $(g_e)_{e\in\Sigma}$,
and we also call them \emph{graph-directed sets}.
\end{theo}

 \begin{rema}\rm
  Theorem \ref{thm:attractor} is actually a direct consequence of the much more general Theorem 3.4 in \cite{Miculescu_2026}. Indeed, by setting in their Definition 3.1 the parameters $J=\cL$, $I=\Sigma$(finite), $m=1$, and taking the comparison function $\phi(t)=\lambda t$ with $\lambda=\max_{e\in \Sigma}r_{e}<1$, their MPIGIFS reduces exactly to our GIFS. Consequently, the fixed point equation for their fractal operator $F_{\mathcal{S}}$ collapses to our system $M_\ell=\bigcup\limits_{\delta\in\cL} \bigcup\limits_{e\in\Sigma_{\ell,\delta}}g_e(M_\delta)$, and Theorem 3.4 guarantees the unique attractor family $(M_\ell)_{\ell\in \cL}$. Thus, our Theorem 2.1 can be seen as a special case of their result.
\end{rema}



\subsection{P.c.f. GIFS}\label{sec:def-p.c.f.}
Let $\left (G_\Sigma,M_\cL \right )$ be a GIFS.
Let $\sigma:\Sigma^\infty\to \Sigma^\infty$ be
the operator deleting the first letter of a word.
We  define the critical and post-critical points of the GIFS  as follows. For $\ell\in \cL$, we define
$$
C_\ell=\bigcup_{e_1\neq e_2\in \Sigma_\ell} M_{e_1}\cap M_{e_2},
$$
$$\cP(G_\Sigma)=\bigcup_{\ell\in \cL}\left\{\sigma^k(\be); ~ \be\in \pi_\ell^{-1}(C_\ell), k\geq 1 \right\},
\text{ and } \cP_\ell (G_\Sigma)=\cP(G_\Sigma)\cap \Sigma_\ell^\infty,
$$
$$
P_\ell(G_\Sigma)=\pi_\ell({\mathcal P}_\ell(G_\Sigma)).
$$

(We shall use $\cP$, $\cP_\ell $ and $P_\ell$   instead of $\cP(G_\Sigma)$, $\cP_\ell (G_\Sigma)$ and $P_\ell(G_\Sigma)$   whenever it   causes no confusion.)

\begin{defi}\label{defi-pcf}
{A GIFS $\left (G_\Sigma,M_\cL \right )$ is said to be \emph{post-critically finite (p.c.f.)}  if
$\bigcup_{\ell\in \cL}\cP_\ell$ is a finite set.}
\end{defi}

\begin{defi}[Vertices of a cylinder]
For $\bomega\in \Sigma^*$, denote
$$P_{\bomega}=g_{\bomega}(P_{t(\bomega)}),$$
 where $t(\bomega)$ is the terminate state of the word $\bomega$. We call the elements of $P_{\bomega}$ the vertices
 of $M_{\bomega}$.
 \end{defi}

\subsection{Evolution of a GIFS}
We call $\Omega\subset \Sigma^*$ a \emph{blocking set}   if
$\bigcup_{I\in \Omega}[I]$ is a partition of $\Sigma^\infty$, where $[I]$ denotes the set of all words in $\Sigma^\infty$ whose prefix is $I$.

Let $\Omega\subset \Sigma^*$ be a blocking set.
We call
$$
{ G}_{\Omega}:=\left (g_I\right )_{I\in \Omega}
$$
an \emph{evolution} of $G_\Sigma$. Here we regard $\Omega$ as a directed graph, while each
$I\in \Omega$ is regarded as a directed edge from the first letter of $I$ to the last letter of $I$.

Let $I\in \Sigma^*$ and denote $\delta=t(I)$. We  define
\begin{equation}\label{eq:otimes}
\{I\}\otimes \Sigma=\{IJ;~J\in \Sigma_\delta\}.
\end{equation}

Kigami \cite{Kigami_2001} proved that a p.c.f. IFS $F_\Sigma$ and its n-th iteration $F_{\Sigma^n}$ share the same post-critical point set. But for general evolution, this no longer holds.  The following lemma is obvious.

\begin{prop}\label{prop:Kigami-GIFS}
Let $G_\Sigma$ be a p.c.f. GIFS, let $G_\Omega$ be an evolution of $G_\Sigma$.
Then   the post-critical point set of $G_\Omega$ is a subset of    that of $G_\Sigma$.
\end{prop}

\section{{Proof of Theorem \ref{main-1-GIFS}}}\label{SECTION3}

A p.c.f. GIFS $\cG=(G_\Sigma, M_{\cL})$ is said to have \emph{sparse post-critical point set},
 or simply to be \emph{sparse}, if each  $1$-st cylinder located in $M_\ell$ contains at most one point of $P_\ell$.
 The following lemma is obvious.

\begin{lem} Let $(G_\Sigma, M_{\cL})$ be a p.c.f. GIFS. Then

(i) $(G_e)_{e\in \Sigma^n}$, the $n$-th iteration of $\cG$, is sparse for $n$ large enough.

(ii) If $G_\Sigma$ satisfies the angle separation condition, then any evolution of $G_\Sigma$ also satisfies
the angle separation condition.

\end{lem}


In this section, from now on,  we always assume that $(G_\Sigma, M_{\cL})$   is a nearly self-similar  p.c.f. GIFS, and that it has a sparse
post-critical point set.

\begin{defi}[Separation pair]\label{separationprefixGIFS}
 \emph{Let $\bx=(x_k)_{k=1}^\infty,\by=(y_k)_{k=1}^\infty\in {\Sigma_{\ell}^\infty}$ with $\pi_{\ell}(\bx)\neq\pi_{\ell}(\by)$. Denote $k=|\bx\wedge\by|$, $\pmb{j}=\bx|_k$.
We define the \emph{separation pair} of $(\bx,\by)$, denoted by $S_{\ell}(\bx,\by)$,
to be $$S_{\ell}(\bx,\by)=(\bx|_m,\by|_n),$$ where $m, n\in {\mathbb Z}_{\geq 1}\cup\{+\infty\}$ are defined as follows:}

 \emph{$(\textrm{\rmnum{1}})$ \   If   $M_{\pmb{j}x_{k+1}x_{k+2}}\cap M_{\pmb{j}y_{k+1}}=\emptyset$ and
$ M_{\pmb{j}x_{k+1}}\cap M_{\pmb{j}y_{k+1}y_{k+2}}=\emptyset$, then
$m=n=k+1$. }

 \emph{$(\textrm{\rmnum{2}})$   If   $\left(M_{\pmb{j}x_{k+1}x_{k+2}}\cap M_{\pmb{j}y_{k+1}}\right)
 \cup \left(M_{\pmb{j}x_{k+1}}\cap M_{\pmb{j}y_{k+1}y_{k+2}}\right)\neq \emptyset$, pick a point $a$ in the above union, and define  $m$ to be  the smallest integer such that
  $a\not\in M_{x_1\dots x_m}$ if $a\neq \pi_{\ell}(\bx)$, and set $m=+\infty$ otherwise; similarly,  $n$ is the smallest integer such that
  $a\not\in M_{y_1\dots y_n}$ if $a\neq \pi_{\ell}(\by)$, and $n=+\infty$ otherwise.}
\end{defi}

\begin{rema}\label{rem:sparse}
\rm (i) Since the GIFS is sparse, we have that $ M_{\pmb{j}x_{k+1}x_{k+2}}\cap M_{\pmb{j}y_{k+1}} $
is either the empty set or a singleton. So is $ M_{\pmb{j}x_{k+1}}\cap M_{\pmb{j}y_{k+1}y_{k+2}}$.
Therefore, to make the above definition unique, if both intersections are not empty, we may choose $a$ to be the point
in the first intersection if $x_{k+1}\prec y_{k+1}$ and choose $a$ to be the point in the second intersection otherwise.

(ii) Note that $m=n=+\infty$ is impossible since $\pi_\ell(\bx)\neq \pi_\ell(\by)$.
\end{rema}

\begin{figure}[h]
  \centering
  \includegraphics[width=7 cm]{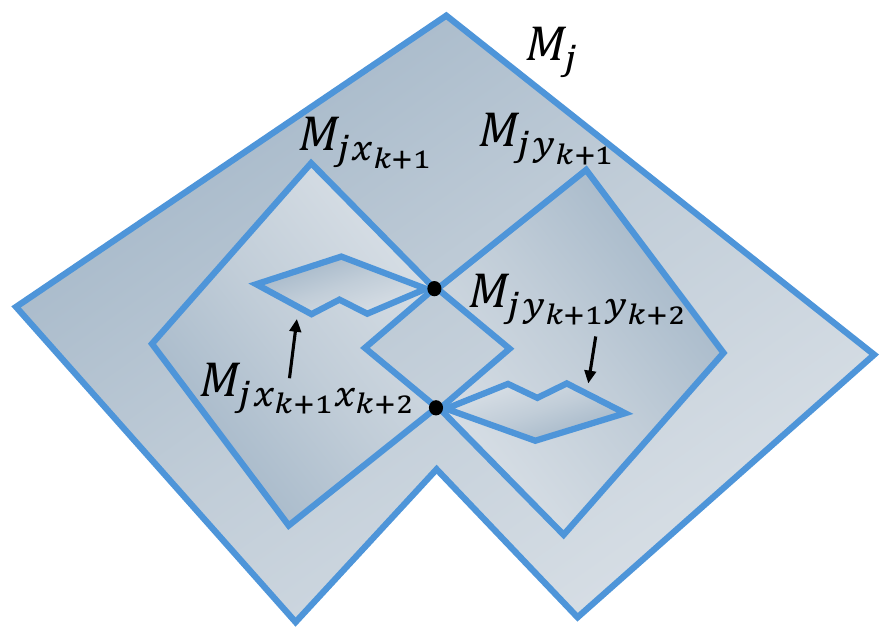}\\
   \caption{$ M_{\pmb{j}x_{k+1}x_{k+2}}\cap M_{\pmb{j}y_{k+1}} $ and $ M_{\pmb{j}x_{k+1}}\cap M_{\pmb{j}y_{k+1}y_{k+2}}$ are not empty.}\label{fig:similar}
\end{figure}

 Using separation pair, we define a metric-like function $\rho_{\ell}$ on $ {\Sigma_{\ell}^\infty}\times {\Sigma_{\ell}^\infty}$ as
\begin{equation}
\rho_{\ell}(\bx,\by)=\left \{
\begin{array}{ll}
\max\left\{r_{\boldsymbol{\mu}}, r_{\boldsymbol{\nu}}\right\}, & \text{ if $\pi_{\ell}(\bx)\neq \pi_{\ell}(\by)$; }\\
0, &\text{ otherwise, }
\end{array}
\right .
\end{equation}
  where $(\boldsymbol{\mu}, \boldsymbol{\nu})$ is the separation pair of $(\bx,\by)$.

 Let $X=(X,d)$ be a complete metric space.  For $A,B\subset X$, we define the \emph{distance} between $A$ and $B$ by
$\text{dist~}(A,B)=\inf\{d(x,y): x\in A , y\in B\}.$
 We denote the \emph{diameter} of set $A$ by $|A|_d$, or simply by $|A|$.

  For a p.c.f. GIFS, we define two constants as follows:

  Let $\xi$ be the minimum distance between a  $1$-st cylinder of $M_\ell$  and a post-critical point that does not belong to it, \textit{i.e.},
  \begin{equation}\label{xiK2}
  \xi=\min \left\{\text{dist}(M_e, p);~e\in \Sigma_\ell, p\in P_\ell\setminus M_e\right\}.
  \end{equation}
   Let $\xi'$ be the minimum distance between two $2$-nd cylinders of $M_\ell$  which are disjoint, \textit{i.e.},
   \begin{equation}\label{xiK3}
\xi'=\min\left\{\text{dist}\left(M_{\bomega},M_{\bnu}\right); \bomega, \bnu\in\Sigma_\ell^2 \ \text{and} \ M_{\bomega}\cap M_{\bnu}=\emptyset\right\}.
\end{equation}
Denote
$$
r^{\ast}=\max\{r_e;~e\in \Sigma\}, \quad r_*=\min\{r_e;~e\in \Sigma\},\quad m^*=\max\{|M_\ell|;~\ell\in {\mathcal L}\}.
$$

\begin{lem}\label{LipschitzGIFS} If $(G_{\Sigma}, M_{\cL})$ is a nearly self-similar p.c.f. GIFS satisfying the angle separation condition, then
 there is a constant $c_1>0$ such that
\begin{equation}
c_1^{-1}\rho_{\ell}(\bx,\by)\leq d_{\ell}(\pi_{\ell}(\bx),\pi_{\ell}(\by))\leq c_1\rho_{\ell}(\bx,\by), \quad \text{for any} \ \bx,\by\in {\Sigma_{\ell}^\infty}.
\end{equation}
\end{lem}

\begin{proof}
 Let  $C_0$ be a  deformation bound for the nearly self-similarity.
Denote
$$x=\pi_\ell(\bx), \ y=\pi_\ell(\by), \ k=|\bx\wedge\by|, \  \pmb{j}=\bx|_k.$$
According to the two scenarios in  Definition \ref{separationprefixGIFS}, we divide the proof into two cases.

 \textit{Case 1.} Scenario (i): $M_{\pmb{j}x_{k+1}}\cap M_{\pmb{j}y_{k+1}y_{k+2}}=\emptyset$ and $M_{\pmb{j}x_{k+1}x_{k+2}}\cap M_{\pmb{j}y_{k+1}}=\emptyset$.

 In this case,   $S_\ell(\bx,\by)=(\bx|_{k+1},\by|_{k+1})$, so $\rho_\ell(\bx,\by)=r_{\pmb{j}}\max\left\{r_{x_{k+1}}, r_{y_{k+1}}\right\}$.
By $x,y\in M_{\pmb{j}}$, we have
\begin{equation}\label{C41GIFS}
d_\ell(x,y)\leq |M_{\pmb{j}}|_{d_\ell}\leq C_0|M_{t(\pmb{j})}|_{d_{t(\pmb{j})}}r_{\pmb{j}}  \leq \frac{C_0 m^*}{r_\ast}\rho_\ell(\bx,\by).
\end{equation}
That $M_{\pmb{j}x_{k+1}}\cap M_{\pmb{j}y_{k+1}y_{k+2}}=\emptyset$ implies that
 \begin{equation}\label{C42GIFS}
d_\ell(x,y)\geq \text{dist~}_{\ell}\left(M_{\pmb{j}x_{k+1}}, M_{\pmb{j}y_{k+1}y_{k+2}}\right)\geq C_0^{-1}r_{\pmb{j}}\xi'\geq C_0^{-1}\xi'\rho_\ell(\bx,\by).
\end{equation}

 \textit{Case 2.} Scenario (ii): $M_{\pmb{j}x_{k+1}}\cap M_{\pmb{j}y_{k+1}y_{k+2}}=\{a\}$ or $M_{\pmb{j}x_{k+1}x_{k+2}}\cap M_{\pmb{j}y_{k+1}}=\{a\}$. (See Remark \ref{rem:sparse}.)

 Without loss of generality, we assume the first equality holds.

 \textit{Case 2.1.} $a\notin\{x, y\}$.

 In this case, we have $S_\ell(\bx,\by)=(\bx|_m,\by|_n)$, where $m,n$ are the smallest integers such that
$a\notin M_{x_1\cdots x_m}$ and $a\notin M_{y_1\cdots y_n}$ respectively.
So $\rho_\ell(\bx,\by)=\max\left\{r_{x_1\cdots x_m}, r_{y_1\cdots y_n}\right\}$.

 On one hand,  by $M_{x_1\cdots x_{m-1}}\cap M_{y_1\cdots y_{n-1}}=\{a\}$,  we have
\begin{equation}\label{C43GIFS}
\begin{split}
d_\ell(x,y)&\leq |M_{x_1\cdots x_{m-1}}|+|M_{y_1\cdots y_{n-1}}| \\ &\leq
2C_0|M_\ell|\frac{\max\left\{r_{x_1\cdots x_m}, r_{y_1\cdots y_n}\right\}}{r_\ast}\leq \frac{2C_0m^*}{r_\ast}\rho_\ell(\bx,\by).
\end{split}
\end{equation}

 On the other hand,  by the angle separation condition,  we have
$$d_\ell(x,y)\geq \theta_0\max\left\{d_\ell(x,a),d_\ell(y,a)\right\}$$
where $0<\theta_0<1$ is the constant in the angle separation condition.
Since $a\notin M_{x_1\cdots x_m}$ and $a\notin M_{y_1\cdots y_n}$, we have $d_\ell(x,a)\geq C_0^{-1}\xi r_{x_1\cdots x_{m-1}}$ and $d_\ell(y,a)\geq C_0^{-1}\xi r_{y_1\cdots y_{n-1}}$.
Hence,
\begin{equation}\label{C44GIFS}
d_\ell(x,y)\geq \theta_0C_0^{-1}\xi\max\left\{r_{x_1\cdots x_{m-1}},r_{y_1\cdots y_{n-1}}\right\}\geq\frac{\theta_0\xi}
{C_0}\rho_\ell(\bx,\by).
\end{equation}

 \textit{Case 2.2.} $a\in\{x,y\}$.

 We only prove the case $a=x$. (The case of $a=y$ can be handled in the same way.)

 In this case, we have $S_\ell(\bx,\by)=(\bx|_{+\infty},\by|_n)$, where $n$ is the smallest integer such that $a\notin M_{y_1\cdots y_n}$.
 So $\rho_\ell(\bx,\by)=r_{y_1\cdots y_n}$.
Since $x\in M_{y_1\cdots y_{n-1}}$, we have
\begin{equation}\label{C45GIFS}
d_\ell(x,y)\leq |M_{y_1\cdots y_{n-1}}|\leq\frac{C_0m^{\ast}}{r_\ast}\rho_\ell(\bx,\by).
\end{equation}
Since $a\notin M_{y_1\cdots y_n}$, we have
\begin{equation}\label{C46GIFS}
d_\ell(x,y)=d_\ell(a,y)\geq C_0^{-1}\xi r_{y_1\cdots y_{n-1}}\geq  C_0^{-1}\xi \rho_\ell(\bx,\by).
\end{equation}
 Summing up  \eqref{C41GIFS}-\eqref{C46GIFS},  the lemma holds by setting
 $$c_1=\max\left\{\frac{2C_0m^*}{r_\ast},  \ \frac{C_0}{ \xi'},    \ \frac{C_0}{\theta_0 \xi}\right\}.$$
\end{proof}

Recall that
\begin{equation}
\partial \Sigma=\bigcup_{\ell\in \cL} \{e\in\Sigma_\ell;~ M_e\cap P_{\ell}\neq \emptyset\}.
\end{equation}
For convenience, we make the convention that the empty word $\varepsilon\in \partial \Sigma$.


 \begin{lem}\label{partialSigmaK1GIFS}
Let $\bx=(x_k)_{k=1}^\infty,\by=(y_k)_{k=1}^\infty\in {\Sigma_{\ell}^\infty}$ with $\pi_{\ell}(\bx)\neq \pi_{\ell}(\by)$.
Denote $h=|\bx\wedge\by|$ and
let $S_{\ell}(\bx, \by)=(\bx|_m, \by|_n)$ be the separation pair. Then
\begin{equation}\label{eq:boundary-sequence}
\{x_k\}_{k=h+2}^{m-1}, \ \{y_k\}_{k=h+2}^{n-1}\subset \partial \Sigma.
\end{equation}
\end{lem}

\begin{proof} If scenario  (i) of Definition \ref{separationprefixGIFS} occurs, then $m=n=h+1$, and both sequences in \eqref{eq:boundary-sequence} are empty, so the lemma holds in this case.

  Now suppose scenario  (ii) of Definition \ref{separationprefixGIFS} occurs. Denote $\pmb{j}=\bx|_h$ and
   let $a$ be the point in item (ii) of Definition \ref{separationprefixGIFS}. Since $x_{h+1}\neq y_{h+1}$,
   we infer that $a$ is a common vertex of both $M_{\pmb{j}x_{k+1}}$ and $M_{\pmb{j}y_{k+1}}$.
Since  $m$, $n$ are the smallest integers such that
$a\notin M_{x_1\cdots x_m}$, $a\notin M_{y_1\cdots y_n}$ respectively,
we deduce that  $x_{h+1}x_{h+2}\cdots x_{m-1}$ and $y_{h+1}y_{h+2}\cdots y_{n-1}$
are words over $\partial \Sigma$.
 The lemma is proved.
\end{proof}

\begin{proof}[\textbf{Proof of Theorem \ref{main-1-GIFS}}]
 If the given system $((g_e)_{e\in \Sigma},(M_\ell, d_\ell)_{\ell\in \cL})$ is not sparse, then
 its $m$-th iteration  $((g_e)_{e\in \Sigma^{m}},(M_\ell, d_\ell)_{\ell\in \cL})$ has sparse post-critical point set for $m$ large;
  moreover,
$((g_e)_{e\in \Sigma^{m}},(M_\ell, d_\ell)_{\ell\in \cL})$ and $((g_e)_{e\in \Sigma^{m}},(M_\ell, d'_\ell)_{\ell\in \cL})$ satisfy all the conditions in Theorem \ref{main-1-GIFS}.
Therefore, without loss of generality, we assume that $((g_e)_{e\in \Sigma},(M_\ell, d_\ell)_{\ell\in \cL})$ has sparse post-critical point set.

 Pick $\ell\in \cL$.  Pick $\bx,\by\in \Sigma_\ell^\infty$.
Since the two GIFS in consideration differ only at metrics, the separation pair of $(\bx,\by)$ in $(M_\ell, d_\ell)$  and
 that in $(M_\ell, d'_\ell)$ coincide, and we will denote them by
 $S_\ell(\bx,\by)$.

 We are going to show that the identity map $id:M_\ell\to M_\ell$ is   quasisymmetric  from $(M_\ell, d_\ell)$ to $(M_\ell, d'_\ell)$.

 Let $\bx=(x_k)_{k=1}^\infty, \by=(y_k)_{k=1}^\infty, \bz=(z_k)_{k=1}^\infty\in {\Sigma_\ell^\infty}$ such that $\pi_\ell(\bx),\pi_\ell(\by),\pi_\ell(\bz)$ are distinct. Let $t\in[0,\infty)$.
According to Lemma \ref{LipschitzGIFS}, to prove $id:M_\ell\to M_\ell$ is quasisymmetric, we only need to prove that there exists a homeomorphism $\eta$ of $[0,\infty)$ to itself such that   $\rho_\ell(\bx,\by)\leq t\rho_\ell(\bx,\bz)$
implies
\begin{equation}\label{quasisymmetricGIFS}
\rho'_\ell(\bx,\by)\leq \eta(t)\rho'_\ell(\bx,\bz).
\end{equation}
 In the following, we show that $$\eta(t)=
\frac{1}{(r'_\ast)^{3}(r_\ast)^{2s}}\max\left(t^s, t^{s+\ {\log r'_\ast}/{\log r^*}}\right)$$
is the desired homeomorphism,   where $s$ is the real number such that
\begin{equation}\label{eq:es}
r_e=(r_e')^s, \quad  \forall e\in \partial \Sigma.
\end{equation}

Let $\bomega$ be a word, we make the convention that
  $r_{\bomega}=r'_{\bomega}=0$  if $|\bomega|=\infty$, and  $r_{\varepsilon}=r'_{\varepsilon}=1$ where $\varepsilon$ denotes the empty word.

Denote $h=|\bx\wedge\by|$ and $p=|\bx\wedge\bz|$.  Denote
\begin{equation}\label{SGFxyzGIFS}
 S_\ell(\bx, \bz)=(\bx|_m, \bz|_n) \quad \text{and} \quad
 S_\ell(\bx, \by)=(\bx|_{m'}, \by|_{n'}),
\end{equation}
where $m, n, m', n' \in \Z^+\cup\{+\infty\}$.
  Then
\begin{equation}\label{rhoxz1GIFS}
\rho_\ell(\bx,\bz)=r_{x_1\dots x_p}\max\left\{r_{x_{p+1}\cdots x_m}, r_{z_{p+1}\cdots z_n}\right\}, \quad
\rho'_\ell(\bx,\bz)=r'_{x_1\dots x_p}\max\left\{r'_{x_{p+1}\cdots x_m}, r'_{z_{p+1}\cdots z_n}\right\}
\end{equation}
and
\begin{equation}\label{rhoxy1GIFS}
\rho_\ell(\bx,\by)=r_{x_1\dots x_h}\max\left\{r_{x_{h+1}\cdots x_{m'}}, r_{y_{h+1}\cdots y_{n'}}\right\}, \quad
\rho'_\ell(\bx,\by)=r'_{x_1\dots x_h}\max\left\{r'_{x_{h+1}\cdots x_{m'}}, r'_{y_{h+1}\cdots y_{n'}}\right\}.
\end{equation}

We claim that
\begin{equation}\label{eq:power-law-mn}
 r'_{x_{p+2}\cdots x_{m-1}}= \left(r_{x_{p+2}\cdots x_{m-1}}\right)^s, \quad
   r'_{z_{p+2}\cdots z_{n-1}}= \left(r_{z_{p+2}\cdots z_{n-1}}\right)^s.
 \end{equation}
 Let us first prove the first equality.
   If $m=p+1$ or $p+2$, then $x_{p+2}\cdots x_{m'-1}$ is the empty word, then both sides of the equation equal $1$.
 If $m=\infty$, then both sides of the equation equals $0$. So \eqref{eq:power-law-mn} holds in these cases.
 If $p+3\leq m<\infty$,    by Lemma \ref{partialSigmaK1GIFS}, $x_{p+2}\cdots x_{m-1}$
is a word over the alphabet $\partial \Sigma$, so \eqref{eq:power-law-mn}   holds by \eqref{eq:es}.
By the same argument, one can prove the second equality in  \eqref{eq:power-law-mn}. Our claim is proved.

Similarly, we have
\begin{equation}\label{eq:power-law-m'n'}
 r'_{x_{h+2}\cdots x_{m'-1}}= \left(r_{x_{h+2}\cdots x_{m'-1}}\right)^s,\quad
 r'_{y_{h+2}\cdots y_{n'-1}}= \left(r_{y_{h+2}\cdots y_{n'-1}}\right)^s.
 \end{equation}

Without loss of generality, we assume that $p\geq h$.
Using $h=|\bx\wedge\by\wedge\bz|$,  we obtain
\begin{equation}\label{case2aGIFS}
\frac{\rho'_\ell(\bx,\by)}{\rho'_\ell(\bx,\bz)}=
\frac{1}{r'_{x_{h+1}\cdots x_p}}\cdot
\frac{\max\left\{r'_{x_{h+1}\cdots x_{m'}}, r'_{y_{h+1}\cdots y_{n'}}\right\}}
{\max\left\{r'_{x_{p+1}\cdots x_m}, r'_{z_{p+1}\cdots z_n}\right\}}
\leq
\frac{1}{(r'_*)^3} \cdot \frac{1}{r'_{x_{h+2}\cdots x_p}} \cdot B',
\end{equation}
where
$$
B'=\frac{\max\left\{r'_{x_{h+2}\cdots x_{m'-1}}, r'_{y_{h+2}\cdots y_{n'-1}}\right\}}
{\max\left\{r'_{x_{p+2}\cdots x_{m-1}}, r'_{z_{p+2}\cdots z_{n-1}}\right\}}.
$$
Similarly, we have
\begin{equation}\label{case2bGIFS}
\frac{\rho_\ell(\bx,\by)}{\rho_\ell(\bx,\bz)}=
\frac{1}{r_{x_{h+1}\cdots x_p}}\cdot
\frac{\max\left\{r_{x_{h+1}\cdots x_{m'}}, r_{y_{h+1}\cdots y_{n'}}\right\}}
{\max\left\{r_{x_{p+1}\cdots x_m}, r_{z_{p+1}\cdots z_n}\right\}}
\geq
(r_*)^2 \cdot \frac{1}{r_{x_{h+2}\cdots x_p}}\cdot B,
\end{equation}
where
$$
B=\frac{\max\left\{r_{x_{h+2}\cdots x_{m'-1}},r_{y_{h+2}\cdots y_{n'-1}}\right\}}
{\max\left\{r_{x_{p+2}\cdots x_{m-1}},r_{z_{p+2}\cdots z_{n-1}}\right\}}.$$
By  \eqref{eq:power-law-mn}  and \eqref{eq:power-law-m'n'}, we obtain  $B'=B^s$, which together with
 \eqref{case2aGIFS} and \eqref{case2bGIFS} imply that
\begin{equation}\label{case2cGIFS}
\frac{\rho'_\ell(\bx,\by)}{\rho'_\ell(\bx,\bz)}
\leq \frac{t^s}{(r'_\ast)^3(r_\ast)^{2s}} \cdot
\frac{\left(r_{x_{h+2}\cdots x_p}\right)^s}{r'_{x_{h+2}\cdots x_p}}:=
\frac{t^s}{(r'_\ast)^3(r_\ast)^{2s}}\cdot Y,
\end{equation}
where
$$
Y= \frac{\left(r_{x_{h+2}\cdots x_p}\right)^s}{r'_{x_{h+2}\cdots x_p}}.
$$

 If $p<m'$, then $\{x_k\}_{k=h+2}^p$ is a word over $\partial \Sigma$, so $Y=1$ and
\begin{equation}\label{t2GIFS}
\rho'_\ell(\bx,\by)\leq \frac{t^s}{(r'_\ast)^3(r_\ast)^{2s}}\rho'_\ell(\bx,\bz).
\end{equation}

 If $p\geq m'$, we first estimate $p-m'$. The  equality in \eqref{case2bGIFS} together with $\rho_\ell(\bx,\by)\leq t\rho_\ell(\bx,\bz)$ imply that
\begin{equation}\label{case2dGIFS}
\frac{1}{r_{x_{h+1}\cdots x_p}}\cdot\frac{\max\left\{r_{x_{h+1}\cdots x_{m'}}, r_{y_{h+1}\cdots y_{n'}}\right\}}
{\max\left\{r_{x_{p+1}\cdots x_m}, r_{z_{p+1}\cdots z_n}\right\}}
\leq t,
\end{equation}
 so we have
  $$
  r_{x_{h+1}\cdots x_{m'}}\leq  t r_{x_{h+1}\cdots x_{p}} \leq tr_{x_{h+1}\cdots x_{m'}}(r^*)^{p-m'+1},
  $$
   which implies $p-m'+1\leq\frac{-\log t}{\log r^*}$.

   Since $x_{h+2}\cdots x_{m'-1}$ is a word over $\partial \Sigma$,  we have
    $$Y=
\frac{\left(r_{x_{h+2}\cdots x_p}\right)^s}{r'_{x_{h+2}\cdots x_p}}
=\frac{\left(r_{x_{m'}\cdots x_p}\right)^s}{r'_{x_{m'}\cdots x_p}}
\leq\left(\frac{(r^*)^s}{r'_*}\right)^{p-m'+1}\leq \left(\frac{1}{r'_*}\right)^\frac{-\log t}{\log r^*},$$ which together with \eqref{case2cGIFS} implies that
\begin{equation}\label{t3GIFS}
\rho'_\ell(\bx,\by)\leq \frac{t^s(r'_*)^{\log t/\log r^*}}{(r'_\ast)^{3}(r_\ast)^{2s}} \rho'_\ell(\bx,\bz)
=\frac{t^{s+\log(r'_*)/\log r^*}}{(r'_\ast)^{3}(r_\ast)^{2s}} \rho'_\ell(\bx,\bz).
\end{equation}
Therefore,  \eqref{t2GIFS} and \eqref{t3GIFS} confirm \eqref{quasisymmetricGIFS}. The theorem is proved.
\end{proof}


\section{Metrics induced by valuation pairs}\label{SECTION4}

In this section, we use button chains to define metrics on a p.c.f. fractal. The technique of button chains has been
used in \cite{Yuan_2022, RaoZhang_2016, Luo_2004, RaoYang_2026, Tetenov_2003} to study various topological properties of fractal sets.

In this section, we  always assume that $\cG=(G_\Sigma, M_{\cL})$ is a
  p.c.f. GIFS  with sparse post-critical point set  and all the maps $g_e$ are bijective contractions.
 Let
 $$\cL=\cL'\cup \cL'',$$
  where $\cL'$ is  the collection of non-isolated states
 and $\cL''$  the collection of isolated states.
 Let $(P_\ell)_{\ell\in \cL}$ be the post-critical point sets of $\cG$.
Clearly $P_\ell\neq \emptyset$ if and only if $\ell\in \cL'$.

\begin{rema}\rm Let $\ell\in \cL''$. Then $M_\ell$ will not occur in the right hand side of the set equations of the GIFS.
Indeed, $M_\ell$ appears only once in the set equations in the form
$$
M_\ell=\bigcup_{\delta\in \cL'} \bigcup_{e\in \Sigma_{\ell,\delta}} g_e(M_\delta).
$$
\end{rema}

For $\ell\in \cL'$, denote $X_{\ell,0}=P_\ell$.
For $\ell\in \cL$ and  $n\geq 1$, denote
$$ \quad X_{\ell,n}= \bigcup_{\bnu\in \Sigma_\ell^n} P_{\bnu}.$$

\subsection{Button chain}
\begin{defi}\rm Let $A$ be a compact set on a  metric space $(Y,d)$,
and let $a,b\in A$.
We   call the triple $(A,a,b)$ a \emph{button} of $(Y,d)$.

Let $f$ be a map on $A$, we define the image of a button $B=(A,a,b)$ to be
$$f(B)=(f(A),f(a),f(b)).$$
\end{defi}


Now we consider buttons related to  $\cG$.
For $\ell\in \cL$ and $n\geq 1$, we set
$$\mathcal{B}_{\ell,0}=\{(M_\ell,a,b); a, b\in P_\ell \ \text{and} \ a\neq b\},$$
$$
\mathcal{B}_{\ell,n}=\{g_{\bomega}(B); {\bomega}\in \Sigma_{\ell}^n, B\in \mathcal{B}_{t(\bomega),0}\}.$$
We call $(M_\ell,a,b)$ a \emph{basic button}, and we call elements in ${\mathcal B}_{\ell,n}$ buttons of order $n$.

For $n\geq0$, we call a sequence
$$E=\left((A_1,a_1,b_1), (A_2,a_2,b_2), \cdots, (A_k,a_k,b_k)\right)$$
  a \emph{$\mathcal{B}_{\ell,n}$-chain} joining $a_1$ and $b_k$,
if
$$\{(A_j,a_j,b_j)\}_{j=1}^k\subset \mathcal{B}_{\ell,n} \text{ and } b_j=a_{j+1} (\forall 1\leq j \leq k-1).$$

 For $1\leq j\leq k$, we call $A_j$ a \emph{basic cylinder} of $E$, and call $a_1$ and $b_k$ the \emph{origin} and \emph{terminus} of $E$; we call $a_j$ and $b_j$ \emph{dividing points} of $E$.
  If $a$ and $b$ are dividing points of $E$, we use $E([a,b])$ to denote the sub-chain of $E$ joining $a$ and $b$.

For convenience, we also write a chain as $E=\sum_{j=1}^k(A_j,a_j,b_j)$.

\subsection{Valuation pair $(R,\tau)$ and induced metric.}
Let $R:~\Sigma\to (0,1)$ be a map and we call it a  \emph{ratio function}. For $\bomega\in \Sigma^k$, we define
$R(\bomega)=\prod_{j=1}^k R(\omega_j)$.

For $\ell\in \cL'$, let $\tau_\ell$ be a metric on $P_\ell$.
Denote $\tau=(\tau_\ell)_{\ell\in \cL'}$.

We   call the pair $(R,\tau)$ a \emph{valuation pair}.

 Next, we define a weight function $\varphi$  on button chains   as follows:
 For a button, we define
 $$\varphi(M_\ell, a,b)=\tau_\ell(a,b) \ \ \text{ and } \ \ \varphi(g_{\bomega}(B))=R(\bomega)\varphi(B).$$
 For a  chain
$E=\left( B_1, B_2,\dots, B_k\right)$,
we define
$$\varphi(E)=\sum_{j=1}^{k}\varphi(B_j).$$

Let $n\geq 1$. For $x,y\in X_{\ell,n}$, we denote the collection of $\mathcal{B}_{\ell,n}$-chains joining $x$ and $y$ by ${\mathfrak C}_{\ell,n}(x,y)$, and we define
\begin{equation}\label{eq:Dn}
D_{\ell,n}(x,y)=\min_{E\in {\mathfrak C}_{\ell,n}(x,y)}\varphi(E).
\end{equation}

 Clearly,

 \begin{lem}  $D_{\ell,n}$ is  a metric on $X_{\ell,n}$ for $n\geq 1$.
 \end{lem}

 A  chain $E\in {\mathfrak C}_{\ell,n}(x,y)$  is said to be  \emph{minimal}
 if $\varphi(E)=D_{\ell,n}(x,y).$

 \begin{defi}[Consistency condition.]\label{def:consistency}\rm
 A valuation pair $(R,\tau)$ is said to satisfy the \emph{consistency condition}, if
 for all $\ell\in \cL'$, it holds that
\begin{equation}\label{pair1-gifs}
 D_{\ell,1}(a,b)=\tau_\ell(a,b), \ \ \forall a,b\in P_\ell.
\end{equation}
\end{defi}

\begin{rema}\rm Notice that the metrics $D_\ell, \ell\in \cL'$ are defined by
the parameters $(\tau_\ell)_{\ell\in \cL'}$ and $$\{R(e);~t(e)\in \cL'\}.$$
Hence, the consistency condition is irrelevant to the ratios $R(e)$ such that $e$ has origin
in $\cL''$.
\end{rema}

\begin{lem}\label{K0K1}
If a valuation pair $(R,\tau)$ is consistency, then for all $\ell\in \cL$ and all $n\geq 1$, $D_{\ell,n+1}$
is an extension of $D_{\ell,n}$, or equivalently,
$D_{\ell,n+1}|_{X_{\ell,n}}=D_{\ell,n}$.
\end{lem}

\begin{proof} Let $n\geq 1$ and let $B=(M_{\bomega}, a, b)$ be a button of order $n$ in $M_\ell$. Then necessarily we have
$t(\bomega)\in \cL'$. By the consistency condition (corresponding to $t(\bomega)$), we have

(i) {\color{red}Every ${\mathcal B}_{\ell, n+1}$-chain joining $a$ and $b$ in $M_{\bomega}$ has weight no less than $\varphi(B)$;}

(ii) There exists a ${\mathcal B}_{\ell, n+1}$-chain joining $a$ and $b$ in $M_{\bomega}$ that has weight $\varphi(B)$, and we denote this chain by $F_B$.

Pick $a, b\in X_{\ell,n}$.
Let $E=\sum_{j=1}^{k}B_j$ be a  $\mathcal{B}_{\ell,n}$-chain joining $a$ and $b$. Then
$$
E'=[F_{B_1},\dots, F_{B_k}]
$$
is a ${\mathcal B}_{\ell, n+1}$-chain joining $a$ and $b$ and $\varphi(E')=\varphi(E)$.  It follows that
$D_{\ell,n}(a,b) \geq D_{\ell,n+1}(a,b).$

On the other hand, let $E'$
be a  $\mathcal{B}_{\ell,n+1}$-chain  joining $a$ and $b$. Now the $n$-th cylinders of $M_{\ell}$ cut $E'$
into sub-chains, so we can write   $E'$ as
$$E'=[E'_1, E'_2, \ldots, E'_k],$$
where $E'_j$ are  maximal sub-chains of $E'$ such that
$E'_j$ is covered by a single $n$-th cylinder, say, $E'_j\subset M_{\bnu_j}$.
Denote the origin and terminus of $E'_j$ by $a'_j$ and $b'_j$, then $a'_j, b'_j\in X_{\ell,n}$.
Set $B_j=(M_{\bnu_j}, a'_j,b'_j)$. Then
$E=[B_1,\dots, B_k]$
is a $\mathcal{B}_{\ell,n}$-chain joining $a$ and $b$,
 and  $ \varphi(E)\leq  \sum_{j=1}^k \varphi(E_j')=\varphi(E').$
 It follows that $D_{\ell,n}(a,b)\leq D_{\ell,n+1}(a,b).$ The lemma is proved.
\end{proof}

Therefore,  a metric $D_\ell$ on $\bigcup_{n\geq 1}X_{\ell,n}$  can be defined.
 Since $R(e)<1$ for all $e\in \Sigma$, the completion of the metric space $(\bigcup_{n\geq 0}X_{\ell,n}, D_\ell)$ is $(M_\ell, D_\ell)$.
We call $(M_\ell, D_\ell)_{\ell\in \cL}$ the \emph{metrics induced} by $(R,\tau)$.
We denote the GIFS in the new metrics by
$$\cG_{R,\tau}=((g_e)_{e\in\Sigma}, (M_\ell, D_\ell)_{\ell\in \cL}).$$

\subsection{Nearly self-similarity and angle separation condition}

Denote
 $$\tau_{\min}=\min_{\ell\in\cL'}\min_{a\neq b\in P_{\ell}}\tau_\ell(a,b), \quad
R_{\min}=\min_{e\in \Sigma} R(e).$$
Similarly, we define $\tau_{\max}$ and  $R_{\max}$.

\begin{lem}\label{lem:edge}
Let   $\bomega \in \Sigma_\ell^*\setminus\{\varepsilon\}$. If $a, b\in   P_{\bomega}$,
then
$
D_\ell(a,b)\geq \tau_{\min} R_{\min} R(\bomega).
$
\end{lem}

\begin{proof}  Denote $h=|\bomega|$. Denote $\bomega'=\omega_1\dots \omega_{h-1}$.
Let $H$ be a minimal ${\mathcal B}_{\ell,h}$-chain joining $a$ and $b$.
Since $\cG$  is sparse, we assert that  $H$ contains at least one button $B$ of order $h$ in $M_{\bomega'}$.
 For otherwise, both $a$ and $b$ belong to buttons not in $M_{\bomega'}$, and hence both of them belong to $P_{\bomega'}$,
 which is a contradiction. Therefore
 $
 D_\ell(a,b)\geq \varphi(B)\geq \tau_{\min} R_{\min} R(\bomega).
 $
 \end{proof}

\begin{lem}\label{K0K11GIFS}
If a valuation pair $(R,\tau)$ satisfies the consistency condition, then
for any $\ell\in \cL$, $n\geq 1$, $\bomega\in \Sigma_\ell^*$, and $a,b\in X_{\delta,n}$ where $\delta=t(\bomega)$,
it holds that
\begin{equation}\label{sim1}
C_1 R(\bomega)D_\delta(a,b)\leq  D_\ell(g_{\bomega}(a),g_{\bomega}(b))\leq  R(\bomega)D_\delta(a,b),
\end{equation}
where $C_1=({\tau_{\min} R_{\min}})/{\tau_{\max}  }.$

\end{lem}

\begin{proof} Let $E$
be a minimal  $\mathcal{B}_{\delta,n}$-chain joining $a$ and $b$, then $D_\delta(a,b)=\varphi(E)$, so
$$
\begin{array}{rl}
R(\bomega)D_\delta(a,b)=R(\bomega)\varphi(E)
 = \varphi(g_{\bomega}(E))\geq  D_\ell(g_{\bomega}(a),g_{\bomega}(b)),
\end{array}$$
which confirms the second inequality in \eqref{sim1}.

 Next, we prove the first inequality in \eqref{sim1}.  Denote $|\bomega|=m$.
 Firstly, we prepare an easy fact. For any $a',b'\in X_{\delta,n}$, if $E$ is a minimal ${\mathcal B}_{\ell,n+m}$-chain joining $g_{\bomega}(a')$ and $g_{\bomega}(b')$ such that  $E\subseteq g_{\bomega}(M_\delta)$, then $g_{\bomega}^{-1}(E)$
 is a ${\mathcal B}_{\delta,n}$-chain joining $a'$ and $b'$, so
 \begin{equation}\label{pair223GIFS}
 D_\ell(g_{\bomega}(a'), g_{\bomega}(b'))=\varphi(E)=R(\bomega) \varphi(g_{\bomega}^{-1}(E))\geq R(\bomega)D_\delta(a',b').
\end{equation}

 Let $H$ be a minimal $\mathcal{B}_{\ell,n+m}$-chain joining $g_{\bomega}(a)$ and $g_{\bomega}(b)$.
If $H\subset g_{\bomega}(M_\delta)$, then we are done by \eqref{pair223GIFS}. (See Figure \ref{fig:similar}, $H$
is illustrated by the dotted red line.)

\begin{figure}[h]
  \centering
  \includegraphics[width=9 cm]{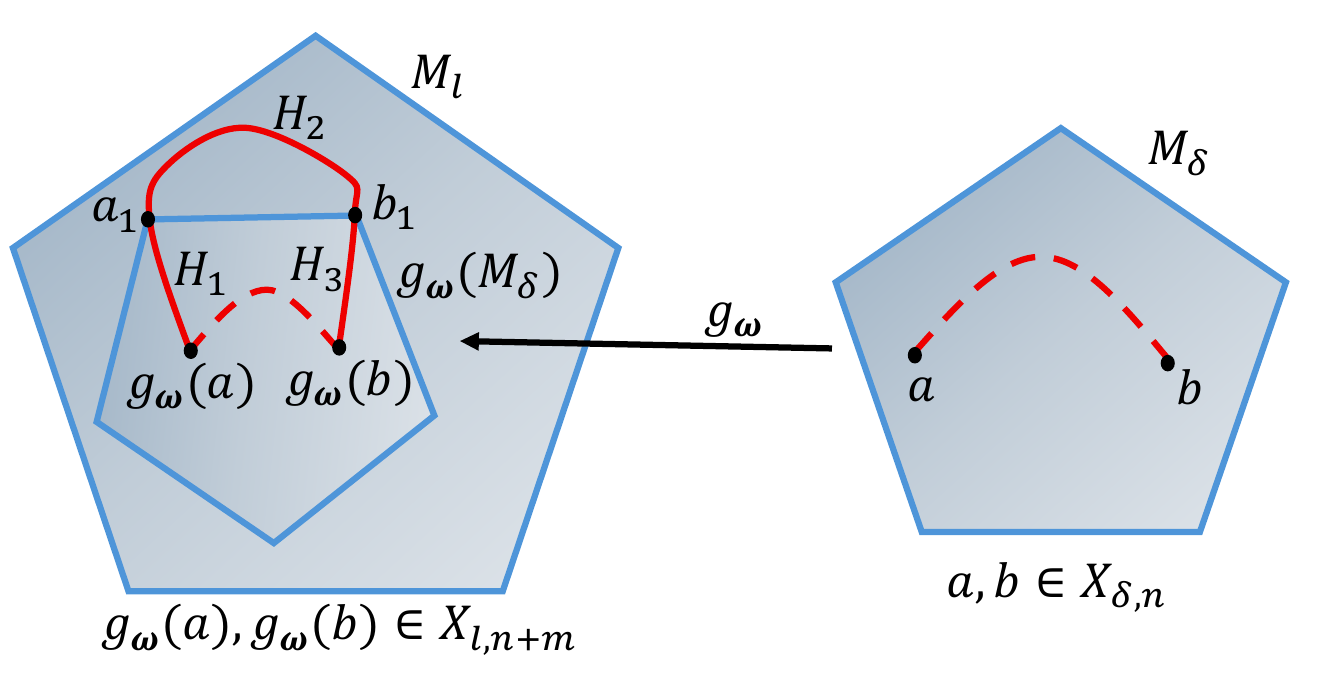}\\
   \caption{A minimal chain joining $g_{\bomega}(a)$ and $g_{\bomega}(b)$.}\label{fig:similar}
\end{figure}

 Suppose  $H$ does not   belong to $g_{\bomega}(M_\delta)$. (See Figure \ref{fig:similar}, $H$
is illustrated by the  solid red line.)
Then $H$ can be decomposed into three sub-chains
$$H=[H_1, H_2, H_3],$$
where $H_1$ is the maximal prefix of $H$ which is contained in $g_{\bomega}(M_\delta)$, and
$H_3$ is the maximal suffix of $H$ which is contained in $g_{\bomega}(M_\delta)$.
(We remark that $H_1$ and\slash or $H_3$ may be the empty set.)
Denote the origin and terminus of $H_2$ by $a_1$ and $b_1$ respectively. Since $a_1$ belongs to the intersection of two different $m$-th cylinders,
we have $g_{\bomega}^{-1}(a_1)\in P_\delta= X_{\delta,0}$.
Similarly $g_{\bomega}^{-1}(b_1)\in P_{\delta}$.

By Lemma \ref{lem:edge}, we have
 \begin{equation}\label{pair22GIFS}
D_\ell(a_1,b_1) \geq  \tau_{\min}R_{\min} R(\bomega) \geq \tau_{\min} R_{\min} R(\bomega)   \cdot \frac{D_\delta(g_{\bomega}^{-1}(a_1),g_{\bomega}^{-1}(b_1))}{\tau_{\max}}.
 \end{equation}
 By  \eqref{pair223GIFS} and \eqref{pair22GIFS}, we have
$$
\begin{array}{rl}
D_\ell(g_{\bomega}(a),g_{\bomega}(b))&=\varphi(H_1)+D_\ell(a_1,b_1)+\varphi(H_3) \\
&\geq R(\bomega) \left(D_\delta(a,g_{\bomega}^{-1}(a_1)) +C_1D_\delta(g_{\bomega}^{-1}(a_1),g_{\bomega}^{-1}(b_1)) + D_\delta(g_{\bomega}^{-1}(b_1),b)\right) \\
&\geq C_1R(\bomega)D_\delta(a,b).
\end{array}
$$
The lemma is proved.
\end{proof}

\begin{proof}[\textbf{Proof of Theorem \ref{thm:R-tau}}]
 Let $(M_\ell, D_\ell)_{\ell\in \cL}$ be the metrics induced by $(R,\tau)$. By Lemma \ref{K0K11GIFS},
\begin{equation}
C_1 R(\bomega)D_\delta(x,y)\leq  D_\ell(g_{\bomega}(x),g_{\bomega}(y))\leq  R(\bomega)D_\delta(x,y)
\end{equation}
 provided that
$x,y\in \bigcup_{n\geq 0}X_{\delta,n}$ and $\bomega\in \Sigma_\ell^*$,  where $\delta=t(\bomega)$. Since $\bigcup_{n\geq 0}X_{\delta,n}$ is dense in $M_\delta$ in the metric $D_\delta$, we obtain that $\cG_{R,\tau}$ is a nearly self-similar GIFS.

 Next, we prove item (ii).
To prove the angle separation condition, we only need to prove that there exists a constant $\theta_0>0$ such that
 for all $\bomega, \bnu\in\Sigma^*_{\ell}$ with $M_{\bomega}\cap M_{\bnu}=\{z\}$, it holds that
\begin{equation}\label{eq:angle}
D_\ell(x,y)\geq \theta_0\max\{D_\ell (x,z),D_\ell (y,z)\},\quad \forall x\in M_{\bomega},
\forall  y\in M_{\bnu}.
\end{equation}
 Since $\bigcup_{n\geq 1} X_{\ell,n}$ is dense in $M_\ell$, we only need to show \eqref{eq:angle}
  for  $x, y\in \bigcup_{n=1}^\infty X_{\ell,n}$.

 If $D_\ell(x,z)\geq 3D_\ell (y,z)$ (or $D_\ell (y,z)\geq 3D_\ell (x,z)$), by the triangle inequality, we have
$$D_\ell (x,y)\geq D_\ell (x,z)-D_\ell (y,z)\geq  2/3\max\{D_\ell (x,z),D_\ell (y,z)\}.$$
So in the following we assume that $1/3\leq D_\ell(x,z)/D_\ell(y,z)\leq 3$.

 Let $E$ be a minimal $\mathcal{B}_{\ell,n}$-chain joining $x$ and $y$ in $M_\ell$.
If $z$  is a dividing point of $E$, the lemma is trivially true. So we assume that $z$  is not a dividing point of
$E$.

 First, we  assume that  $|\bnu\wedge \bomega|=0$. Let $E_{xa}$ be the maximal sub-chain of $E$ initiated from $x$ and  contained in $M_{\omega_1}$, let $E_{by}$ be the maximal
sub-chain initiated from $y$ and contained in $M_{\nu_1}$. Then $a, z\in P_{\omega_1}$ and $b,z\in P_{\nu_1}$.

Set
$
m^*=\max\{|M_\ell|_{D_\ell}; ~\ell\in \cL\}.
$

\begin{figure}[h]
  \centering
  \includegraphics[width=8 cm]{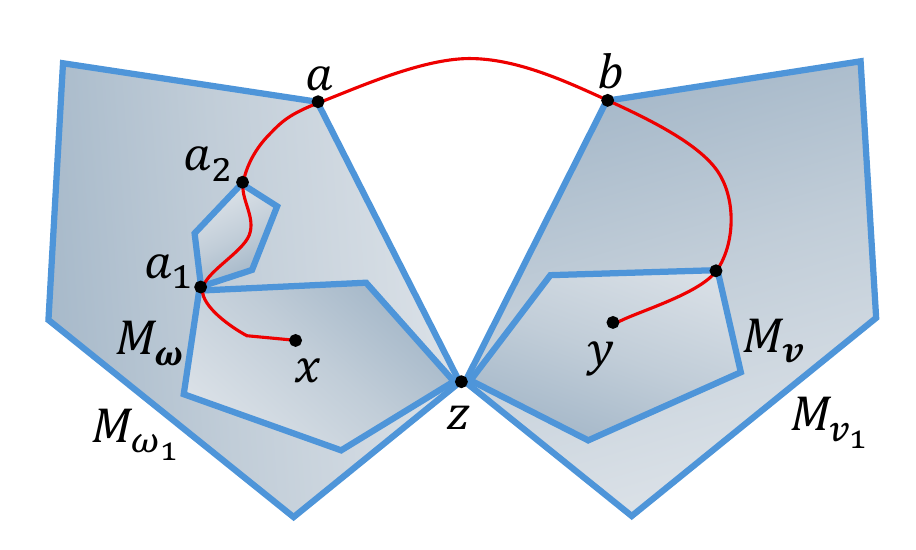}\\
   \caption{A minimal chain joining $x$ and $y$.}
\end{figure}

 \textit{Case 1.}   $|\bomega|=1$ and $|\bnu|=1$.

Since $M_{\bomega}\cap M_{\bnu}=\{z\}$, we conclude that $a\neq b$. Clearly
$D_\ell(a,b)\geq R_{\min}\tau_{\min}$, so
\begin{equation}\label{eq:Dxy}
  D_\ell(x,y) \geq \tau_{\min}  R_{\min}
 \geq   \frac{\tau_{\min}  R^2_{\min}  }{m^*} D_\ell(x,z).
\end{equation}
 Using $1/3\leq D_\ell(x,z)/D_\ell(y,z)\leq 3$, we obtain
\begin{equation}\label{eq:half}
D_\ell(x,y)\geq  \frac{\tau_{\min}  R^2_{\min}  }{3m^*} D_\ell(y,z).
\end{equation}

 \textit{Case 2.}  $|\bomega|>1$ or $|\bnu|>1$.

Without loss of generality, let us assume that $h:=|\bomega|>1$.
Since $a,z\in P_{\omega_1}$, we deduce that $a\not\in M_{\bomega}$ by the sparse property.
Let $a_1\in P_{\bomega}$ such that
$$D_\ell(x,a)=D_\ell(x, a_1)+D_\ell(a_1,a).$$
 Let
$H$ be a minimal $\mathcal{B}_{\ell,h}$-chain joining $a_1$ and $a$, and let
$(M_{\bomega'}, a_1,a_2)$ be the initial  button of $H$  . Denote $\bomega^*=\bomega\wedge \bomega'$ and let  $p=|\bomega^*|$.
Then $a_1\in X_{\ell, p+1}$. So by the sparse property, we have $p+1=h$ since $z\in X_{\ell,1}$ already.
It follows that $|\bomega\wedge \bomega'|=h-1$. Therefore,
$$
\begin{array}{rl}
  D_\ell(x,y) & \geq   D_\ell(a_1,a_2)\\
&\geq \tau_{\min}  R_{\min}  R(\bomega')\quad \ \  \text{(By Lemma \ref{lem:edge}.)} \\
&\geq \tau_{\min}  R^2_{\min}\frac{|M_{\bomega}|_{D_\ell}}{|M_{t(\bomega)}|_{D_{t(\bomega)}}}   \quad \ \
\text{(Since $|M_{\bomega}|_{D_\ell}\leq R(\bomega)|M_{t(\bomega)}|_{D_{t(\bomega)}} $.)}\\
&\geq   \frac{\tau_{\min}  R^2_{\min}  }{m^*} D_\ell(x,z).
\end{array}
$$
So equations \eqref{eq:Dxy} and \eqref{eq:half} still hold.

 Next, we deal with the case that  $\bi=\bomega\wedge \bnu$ and $|\bi|\geq 1$. Denote $(x',y',z')=g_{\bi}^{-1}(x,y,z)$ and set
$$
\theta_1=  \frac{  R^2_{\min}  \tau_{\min}  }{3m^*}.
$$
  Then
$$
\begin{array}{rl}
D_\ell(x,y)& \geq  C_1 R(\bi)  D_{t(\bi)}(x', y')\\
 & \geq  C_1 R(\bi) \theta_1 \max\{D_{t(\bi)}(x',z'), D_{t(\bi)}(y',z')\}\\
 &\geq  C_1   \theta_1 \max\{D_\ell(x,z), D_\ell(y,z)\}.
\end{array}
$$
Set $\theta_0=C_1 \theta_1$, we obtain the lemma.
\end{proof}


\section{Spanning dendrites}\label{SECTION5}

In this section, we introduce a notion of spanning dendrites.
We will use spanning dendrites  to help us   to find consistent  valuation pairs on a p.c.f. GIFS.

\subsection{Dendrite}
A \emph{continuum} is a nonempty, compact, connected metric space. A \emph{subcontinuum} is a continuum which is a subset of a given continuum.

 A \emph{Peano continuum} is a  continuum which is locally connected. See \cite{Nadler_1992}.

A \emph{dendrite} means a  Peano continuum containing no simple closed curve.

Let us recall some notations about dendrite. Let $X$ be a dendrite, let $a$ be a point in $X$.
The \emph{order} of $a$  is defined to be the number of the connected components of the set $X\backslash\{a\}$, see \cite{Tetenov_2017}.
A point with order $1$ is called an \emph{end point} of $X$; the set of all end points of $X$ will be denoted by $\mathcal{E}(X)$.
A point with order no less than $3$ is called a \emph{ramification point} of $X$; the set of all ramification points of $X$ will be denoted by $\mathcal{R}_0(X)$. We denote $\mathcal{R}(X)=\mathcal{E}(X)\cup \mathcal{R}_0(X)$.

 A set $\gamma$ homeomorphic to $[0, 1]$ is called a \emph{Jordan arc}.
Let $a, b$ be the end points of $\gamma$.
When the end points  of $\gamma$ need to be emphasized, we  write $\gamma$ as $\gamma_{ab}$.

\begin{prop}[see \cite{Charatonik_1998}]\label{dendrite proposition}
Every dendrite is uniquely arcwise connected, and every subcontinuum of a dendrite is a dendrite.
\end{prop}


The following proposition is a special case of a result of Kigami \cite{Kigami_2001} on self-similar structures.

\begin{prop}(\cite{Kigami_2001})\label{connected components proposition}
Let $(F=\{f_j\}_{j=1}^N, K)$ be a  bi-Lipschitz  p.c.f.  IFS such that $K$ is connected. Let $P$ be the post-critical point set of $K$.
Then

$(\textrm{\rmnum{1}})$ $K$ is arcwise connected.

$(\textrm{\rmnum{2}})$ The number of connected components of $K\backslash P$ is finite.
\end{prop}

The following lemma is obvious.

{\color{red}\begin{lem}\label{GIFS: connected components proposition}
Let
$((g_e)_{e\in\Sigma}, (M_\ell)_{\ell\in \cL})$ be a   p.c.f. GIFS  such that  $M_\ell$ are connected.
Let $(P_\ell)_{\ell\in\cL}$ be the posted-critical point sets of the GIFS.
Then

$(\textrm{\rmnum{1}})$ $M_\ell$ are arcwise connected.

$(\textrm{\rmnum{2}})$ The number of connected components of $M_\ell\setminus P_\ell$ is finite.
\end{lem}}

\begin{theo}\label{theo:span} Let $M$ be a Peano continuum, let $Q$ be a finite subset of $M$. If  $M\backslash Q$ is connected,
then   there exists a dendrite $\Gamma\subset M$ such that  $\mathcal{E}(\Gamma)=Q$.
\end{theo}

\begin{proof}
Write $Q=\{p_1,p_2,\ldots,p_\ell\}$.

We claim that  for any $2\leq j\leq \ell$, there is a dendrite $\Gamma_j\subset M$ such that $\mathcal E(\Gamma_j)=\{p_1,\dots, p_j\}$.

If $j=2$, the claim holds since $M$ is arcwise connected.
Suppose the claim holds for $j$, we are going to construct a dendrite $\Gamma_{j+1}\subset M$ such that $\mathcal E(\Gamma_{j+1})=\{p_1,\dots, p_{j+1}\}$.

 Take any inner point $x$ of  $\Gamma_j$.
 Let  $\gamma_{p_{j+1} x}\subset (M\backslash Q)\cup\{p_{j+1}\}$ be a Jordan arc joining $p_{j+1}$ and $x$.
 Let $y$ be the first point that $\gamma_{p_{j+1}x}$ hits $\Gamma_j$.
 Then $\Gamma_{j+1}=\Gamma_j\cup\gamma_{p_{j+1}x}$  is the desired dendrite.
The theorem is proved.
\end{proof}

\subsection{Spanning GIFS}

\begin{defi}[spanning GIFS] \emph{Let $(G_{\Sigma}, M_{\cL})$ be a p.c.f. GIFS without isolated state.
Let $\Gamma_\ell$ be a dendrite in $M_\ell$. If   $P_\ell\subset {\mathcal E}(\Gamma_\ell)$ for all $\ell\in \cL$, then
we call the triple
$(G_{\Sigma}, M_{\cL},  \Gamma_{\cL})$
a  \emph{ spanning GIFS}.}
\end{defi}

By Theorem \ref{theo:span}, we have

\begin{lem}\emph{Let $(G_{\Sigma}, M_{\cL})$  be a p.c.f. GIFS without isolated state.
If $M_\ell\setminus P_\ell$
are connected for all $\ell\in \cL$, then there exist
dendrites $(\Gamma_\ell)_{\ell\in \cL}$ such that
$(G_{\Sigma}, M_{\cL},  \Gamma_{\cL})$
is a  spanning GIFS.}
\end{lem}


 \section{{Classification of cylinders of  a spanning system: villages and primary arcs}}\label{SECTION6}

In this section, we always assume that
$(G_\Sigma, M_\cL, \Gamma_\cL)$ is
a   spanning  GIFS, that is,  the system does not possess an isolated state,
$\Gamma_\ell$ is a dendrite in $M_\ell$ and   $P_\ell\subset {\mathcal E}(\Gamma_\ell)$ for all $\ell\in \cL$.

We shall classify the cylinders of $G_\Sigma$ into four classes:
villages, non-village corner cylinders, ess-covering of primary arcs, and the rest.

Let $\Omega$ be a subset of $\Sigma^*$. We set
$$
\partial \Omega=\bigcup_{\ell\in \cL}\left \{\bomega\in \Omega\cap \Sigma_\ell^*;~ M_{\bomega}\cap P_\ell\neq \emptyset\right \}.
$$
Moreover, if $M_{\bomega}\cap  P_\ell=\{p\}$ for some $\ell\in\cL$, then we call $M_{\bomega}$ a \emph{corner cylinder}
  at $p$.

For $\Omega\subset \Sigma_\ell^*$, we define
 $$\overline{\Omega}=\bigcup_{\bomega\in \Omega} M_{\bomega}.$$

 \begin{defi}[basic arc]\rm
 A connected component of $\Gamma_\ell\setminus {\mathcal R}(\Gamma_\ell)$ is called an \emph{open basic arc} and its closure is called a \emph{basic arc} of $\Gamma_\ell$.
 \end{defi}

Clearly the collection of end points of all basic arcs is ${\mathcal R}(\Gamma_\ell)$.

\subsection{Village, primary arc, and ess-covering}

 Let   $\gamma$ be a basic arc  having $h$ as an end point, we say $\gamma$ \emph{originates} in an $n$-th cylinder $M_{\bomega}$ at $h$, if there exists $a\in \gamma\cap P_{\bomega}$ such that
the sub-arc  $\gamma([h,a])$ is contained in $M_{\bomega}$.

\begin{defi}[Village]\label{VillageGIFS} \rm  For $h\in   {\mathcal R}(\Gamma_\ell)$, define
$$
 {\mathcal V}_\ell (h)=\{e\in \Sigma; ~   \text{  there exists a basic arc $\gamma$ which originates in   $M_e$ at $h$}\}
$$
and we call  $V_{\ell}(h):=\overline{{\mathcal V}_\ell (h)}$
 the  village of $h$ (in $M_\ell)$.   If $h\in {\mathcal R}_0(\Gamma_\ell)$, we call $V_{\ell}(h)$ an \emph{inner village}.
 \end{defi}

\begin{exam}
 \rm  Let us use the following IFS to explain the notion of   village. Let
$$\left\{f_{1}(z)=\frac{1}{2}iz+\frac{1}{2}, \ f_{2}(z)=\frac{1}{2}z+\frac{1}{2}+\frac{1}{4}i, \ f_{3}(z)=\frac{1}{2}iz+2\right\}$$
be an IFS.
Let $K$ be the attractor of the IFS $\{f_i\}_{i=1}^{3}$ (see Figure \ref{fig:juxing} (a)). Then two cylinders can intersect at multiple points. It  is easy to see that $C_K=\left \{\frac{1}{2}+\frac{1}{4}i, \frac{1}{2}+\frac{3}{4}i, \frac{3}{2}+\frac{1}{4}i, \frac{3}{2}+\frac{3}{4}i \right \}$ and
$$
P_K=\left \{0, \frac{1}{2}, \frac{3}{2}, 2, i, \frac{1}{2}+i, \frac{3}{2}+i, 2+i \right \}.
$$
Let $\Gamma$ be the union of the red lines in Figure \ref{fig:juxing} (b).
Then $\Gamma$ is a spanning dendrite of $K$.

Let $h_0=1+i/2$ be the center of the rectangle. Then
$${\mathcal R}_0(\Gamma)=\left \{h_0, h_1=f_{11}(h_0),  h_2=f_{13}(h_0), h_3=f_{31}(h_0), h_4=f_{33}(h_0) \right \}.$$
{\color{red}Let $\mathcal{F}$ be the $3$-rd iteration of the IFS $\{f_i\}_{i=1}^{3}$. To illustrate the villages more clearly, we use $\mathcal{F}$ instead of the original IFS $\{f_i\}_{i=1}^{3}$. Then the inner villages are given by
$$V(h_0)=f_{222}(K), V(h_1)=f_{112}(K), V(h_2)=f_{132}(K), V(h_3)=f_{312}(K), V(h_4)=f_{332}(K).$$}

\begin{figure}[h]
  \centering
  \subfigure[$P_K$ contains $8$ points.]{
  \includegraphics[width=7 cm]{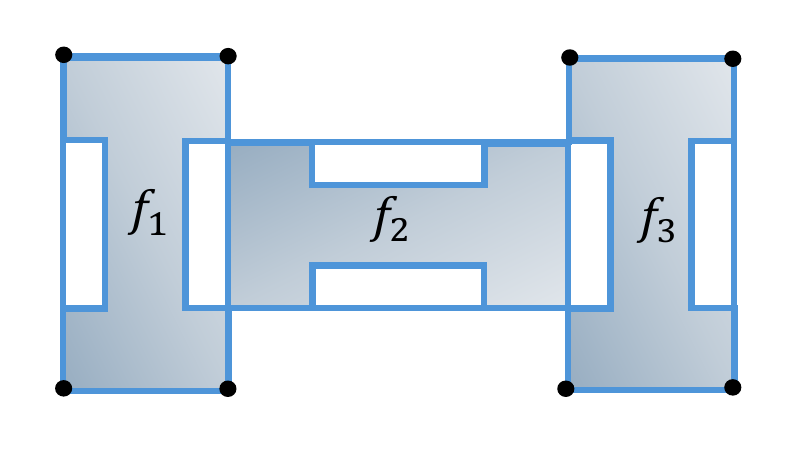}
  }
  \subfigure[spanning dendrite.]{
  \includegraphics[width=8 cm]{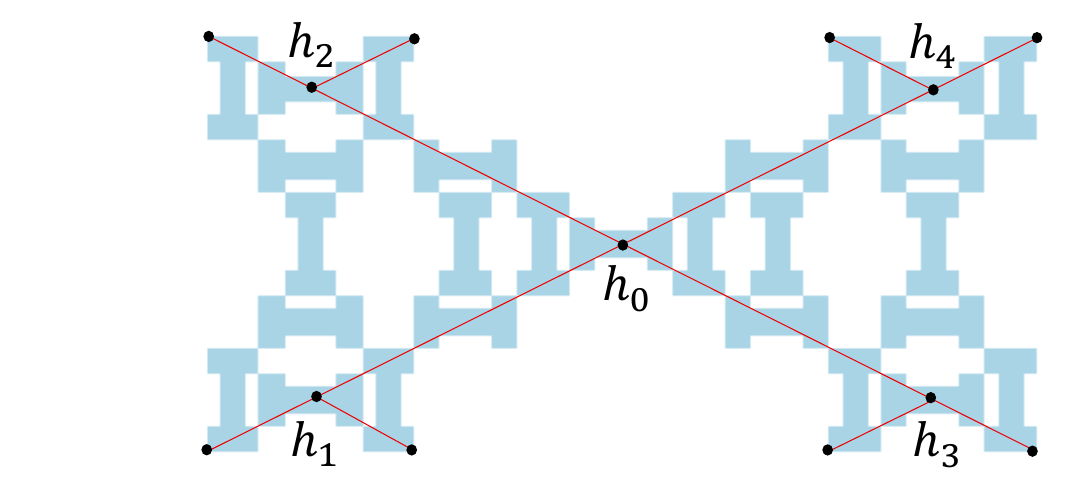}
  }\\
  \caption{The IFS $\{f_i\}_{i=1}^{3}$, where the $1$-st cylinders intersect at multiple points.}
  \label{fig:juxing}
\end{figure}

\end{exam}

We equip $M_\ell$ with the order that $p\prec q$ if
the lowest codings of $p$ is smaller than that of $q$ in lexicographical order.

\begin{defi}[Primary arc] \rm
 Let $\gamma$ be a basic arc of $\Gamma_\ell$ with end points $p$ and $q$.
 If $V_\ell(p)\cap V_\ell(q)\neq \emptyset$, we define the primary arc associated with $\gamma$ to be $\emptyset$;
 otherwise, we define the primary arc associated with $\gamma$
 to be the sub-arc $\gamma_{ab}$ satisfying

(i) $a\in V_{\ell,p}, b\in V_{\ell,q}$ and the open arc $\gamma_{ab}^\circ$ does not intersect
 $V_{\ell,p}\cup  V_{\ell,q}$;

 (ii) $\gamma_{ab}$ is the closest one to $\min\{p,q\}$ among all the sub-arcs satisfying item (i).
\end{defi}

 Let $\gamma$ be an arc  in $M_\ell$, and let $E$ be a subset of $M_\ell$.
   We say a set $E$ \emph{essentially intersects}  $\gamma$ if $E$ contains  a non-degenerate sub-arc of $\gamma$.
    Moreover, we define
$$
\Omega(\gamma)=\left\{e\in \Sigma_\ell; M_e \ \text{essentially intersects } \gamma \right\},
$$
and define the \emph{ess-covering of $\gamma$} to be
$$N(\gamma)=\overline{\Omega(\gamma)}.$$

According to the above definitions, we divide the cylinders $\{M_e;~e\in \Sigma_\ell\}$ into four classes:

$(1)$ \emph{Village cylinder}: if $M_e$ is a subset of a village.

$(2)$ \emph{Ess-covering cylinder}: if $M_e\subset N(\gamma)$ for some primary arc $\gamma\in M_\ell$.

$(3)$ \emph{Grey corner cylinder}: if $M_e$ is a corner cylinder but it is not a village cylinder.

$(4)$ \emph{Dark cylinder}: otherwise.

\begin{figure}[h]
  \centering
  \includegraphics[width=9.5 cm]{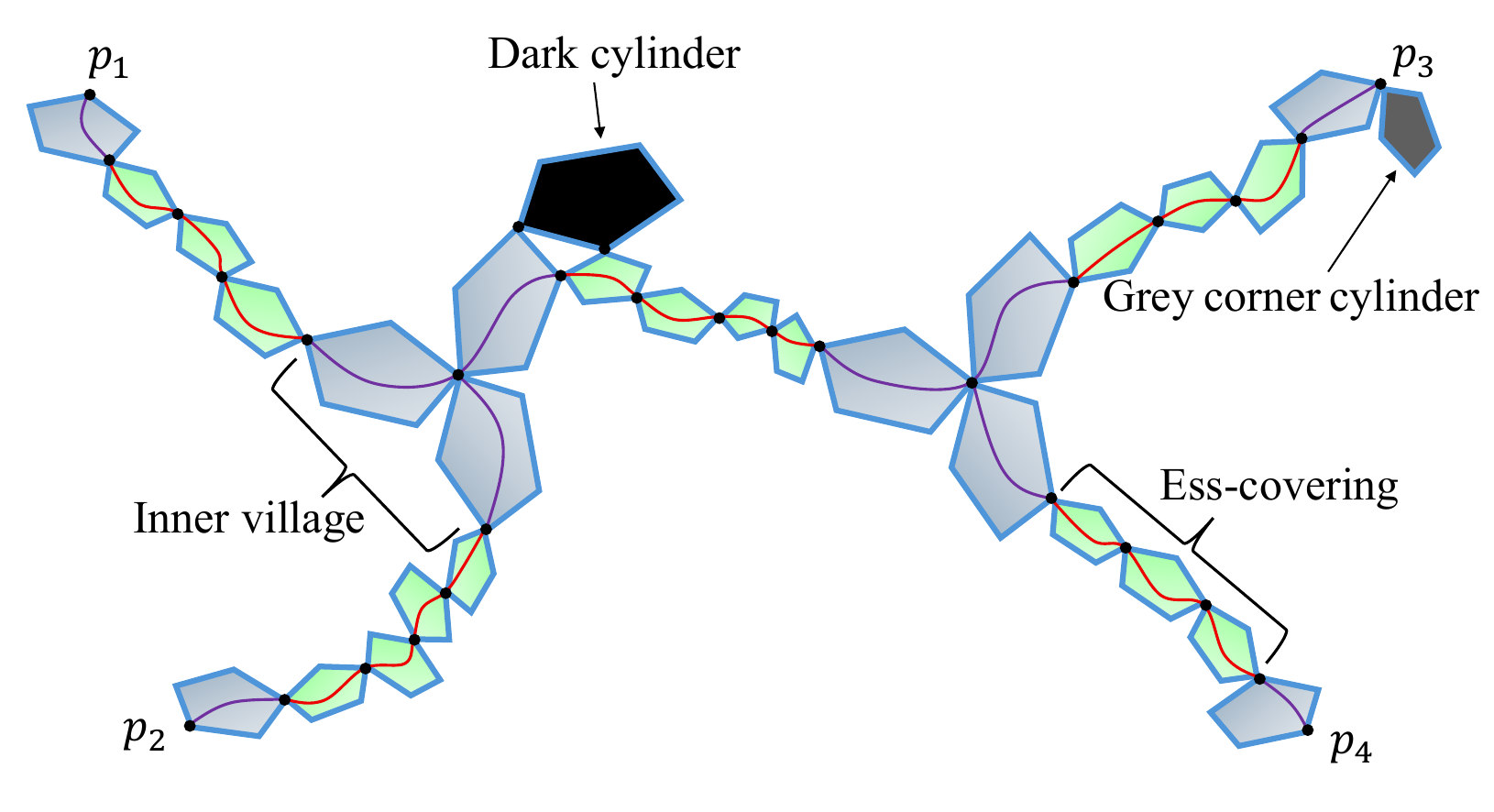}\\
   \caption{.}
   \label{jing_dian_shi_yi_tu}
\end{figure}

\subsection{Separation conditions}

\begin{defi}[First separation condition] \emph{We say  $(G_{\Sigma}, M_\cL, \Gamma_\cL)$ satisfies the \emph{first separation condition}, if for every $\ell\in \cL$,}

$(\textrm{\rmnum{1}})$
All the  villages in $M_\ell$ are disjoint.

$(\textrm{\rmnum{2}})$ For any   $h\in {\mathcal R}(\Gamma_\ell)$ and  any   basic arc  $\gamma$ of $\Gamma_\ell$,
 either $h$ is an  end point of $\gamma$, or $ V_{\ell}(h) \cap  \gamma=\emptyset.$

$(\textrm{\rmnum{3}})$  For every $h\in   {\mathcal R}(\Gamma_\ell)$,
  $h\in M_e\cap M_{e'}$ implies that $\{h\}= M_e\cap M_{e'}$.
\end{defi}

Notice that if $\mathcal G$
  satisfies the first separation condition, then for any basic arc $\gamma$, the associated primary arc is non-degenerated.

Let $\gamma\in M_\ell$ be a primary arc, and let $p\in {\mathcal R}(\Gamma_\ell)$. We say $\gamma$ or $N(\gamma)$ is \emph{related to  $p$}, if $p$ is an end point of the basic arc containing $\gamma$.


 \begin{defi}[The second separation condition]  \emph{Let $\mathcal G=(G_{\Sigma}, M_\cL, \Gamma_\cL)$
  be a spanning GIFS satisfying the first separation condition.
 We say $\mathcal G$ satisfies the \emph{second separation condition}, if for every $\ell\in \cL$,
 the following hold:}

\emph{Let $\gamma$ be a primary arc in $M_\ell$, and let $h\in  {\mathcal R}(\Gamma_\ell)$.}

\emph{(i)  All the ess-coverings of primary arcs   are disjoint.}

\emph{(ii)  A grey corner cylinder does not intersect an ess-covering cylinder.  }

\emph{(iii)  $V_{\ell}(h) \cap N(\gamma)=V_{\ell}(h)\cap \gamma$. }
\end{defi}

\subsection{Evolutions of GIFS}

\begin{theo}\label{thm:evolution} Let $(G_{\Sigma}, M_\cL, \Gamma_\cL)$   be a spanning system. Then there exists an evolution
$G_\Omega$   such that
$(G_{\Omega}, M_\cL, \Gamma_\cL)$ is a   spanning GIFS satisfying the first and the second
 separation conditions.
\end{theo}

\begin{proof} \textit{Step 1.} We set
$$
\Omega_1=\Sigma^{n_1}.
$$
 By Proposition \ref{prop:Kigami-GIFS},
$P_\ell(G_{\Omega_1})\subset P_\ell(G_\Sigma)$,   so $(G_{\Omega_1}, M_\cL,
\Gamma_\cL)$ is a spanning GIFS.
We choose $n_1$ large
so that   the new system $G_{\Omega_1}$ satisfies the first separation condition:

$\bullet$ All villages in $M_\ell$ are disjoint: Let $r>0$ be a positive real number such that the balls in
$$\{B(h,r); h\in {\mathcal R}(\Gamma_\ell)\}$$
 are disjoint.
Choose $n_1$ large so that all cylinders in  $G_{\Omega_1}$ have diameter less than $r$.

$\bullet$ {\color{red} For any   $h\in {\mathcal R}(\Gamma_\ell)$ and  any   basic arc  $\gamma$ of $\Gamma_\ell$,
 either $h$ is an  end point of $\gamma$, or $ V_{\ell}(h) \cap  \gamma=\emptyset$:
We choose $n_1$ large so that the diameter of all $1$-st cylinders in  $G_{\Omega_1}$ is less than
$$
\min\{\text{dist}\left(\gamma, {\mathcal R}(\Gamma_\ell)\setminus  {\mathcal{E}(\gamma)}\right); \text{$\gamma$ is a basic arc of $\Gamma_\ell$}\}.
$$
That way, we can avoid the situation shown in Figure \ref{fig:separation_condition_1}.}

\begin{figure}[h]
  \centering
  \includegraphics[width=7 cm]{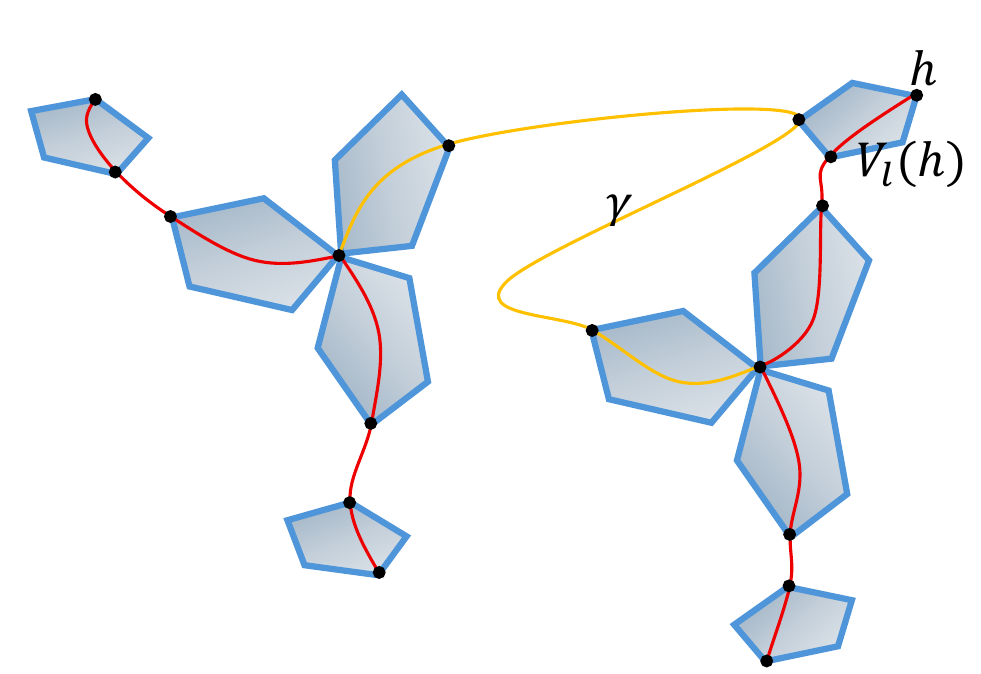}\\
   \caption{$h$ is not an end point of $\gamma$, and $ V_{\ell}(h) \cap  \gamma\neq\emptyset$.}
   \label{fig:separation_condition_1}
\end{figure}

$\bullet$  For every $h\in  {\mathcal R}(\Gamma_\ell)$,
if two cylinders contain $h$, then they  meet only at $h$:
Set
$$
n_0=1+\max_{h\in {\mathcal R}(\Gamma_\ell)} \left \{|\bomega\wedge \bnu|; \bomega, \bnu\in \pi_{\ell}^{-1}(h)
\text{ and }\bomega\neq \bnu \right \},
$$
then  $n_0<\infty$, since every point in $\bigcup_{\ell\in \cL} M_\ell$ has finitely many codings.

Let $n_1>n_0$. Pick $h\in {\mathcal R}(\Gamma_\ell)$. Let $M_{\bomega_1}, \dots, M_{\bomega_k}$ be the collection of cylinders of order $n_1-1$ containing $h$.
Let $M_{\bomega_1i}$ and $M_{\bomega_1j}$ be two sub-cylinders of $M_{\bomega_1}$.
Since $|\bomega_1i\wedge \bomega_1j|=n_1-1\geq n_0$, we conclude that there is one and only one of these sub-cylinders containing $h$.
  Let $M_{\bomega_1e_1}, \dots, M_{\bomega_ke_k}$ be the collection of cylinders of order $n_1$ containing $h$, then these cylinders
meet only at $h$ by the sparse property.

The first separation condition of $G_{\Omega_1}$ is confirmed.

\textit{Step 2.}  We use the villages and primary arcs of $G_{\Omega_1}$ to define $\Omega_2$. Set
  $$
 \Omega'=\{I\in \Omega_1; M_I \text{ is an ess-covering cylinder or  a grey corner cylinder}\}.
 $$
  We keep the cylinders corresponding to $\Omega_1\setminus \Omega'$   unchanged,
  and let the cylinders in $\Omega'$ be iterated $n_2$ times (in the original GIFS), precisely,
  set
  $$
  \Omega_2=(\Omega_1\setminus \Omega')\cup(\Omega'\otimes \Sigma^{n_2}).
  $$

  Notice that all villages in $G_{\Omega_1}$ are unchanged, so $G_{\Omega_1}$ and $G_{\Omega_2}$ share the same villages, and consequently, they share the same  primary arcs. The two systems share the same basic arcs since they share the same spanning
  dendrites. Therefore, $G_{\Omega_2}$ satisfies the first separation condition since $G_{\Omega_1}$ does.

  We choose   $n_2$ large so that the  system $G_{\Omega_2}$ satisfies the second separation condition:

$\bullet$ The ess-coverings of primary arcs are disjoint: Let $\gamma, \gamma'$ be two primary arcs in $M_\ell$.
Then
$$\epsilon_1=d_\ell(\gamma,\gamma')>0.$$

$\bullet$ Ess-coverings of  primary arcs and  grey corner cylinders are disjoint:
Let $p\in P_\ell$ and $\gamma$ be a primary arc in $M_\ell$. Denote $\epsilon_2=d_\ell(\gamma,p)$.

$\bullet$ Ess-covering of a primary arc and a village are disjoint if they are not related:
Let $h\in {\mathcal R}(\Gamma_\ell)$ and $\gamma\subset \Gamma_\ell$ be a primary arc not related to $h$. Then
$$\epsilon_3=d_\ell(\gamma, V_\ell(h))>0,$$
since by the first separation condition, $d_\ell(\gamma',V_\ell(h))>0$, where $\gamma'$ is the basic arc containing $\gamma$.

$\bullet$ {\color{red} Ess-covering of a primary arc and a village intersect at an end point of the primary arc if they are related:
Let $h\in {\mathcal R}(\Gamma_\ell)$ and $\gamma\subset \Gamma_\ell$ be a primary arc related to $h$.
By the definition of primary arc, $\gamma\cap V_\ell(h):=\{a\}$ is an end point of $\gamma$.
Then
$$
\epsilon_4=d_\ell \left(\gamma, \bigcup_{\bomega\in {\mathcal V}_\ell(h)} P_{\bomega}\setminus\{a\}\right)>0.
$$
Moreover, by the definition of $N(\gamma)$,
any member of $V_\ell(h)$ does not belong to $N(\gamma)$, which means that $V_\ell(h)\cap N(\gamma)\subset \bigcup_{\bomega\in {\mathcal V}_\ell(h)} P_{\bomega}$.}

Let $\epsilon$ be the minimum of all the distances occurring above.  We choose $n_2$ large so that all cylinders corresponding to $\Omega'\times \Sigma^{n_2}$
have diameters less than $\epsilon/2$. Then all the above separation conditions hold, and the second separation condition
is confirmed.
\end{proof}

\section{Bright area and its topology, induced chain and corner evolution}\label{SECTION7}

In this section, we introduce three main ingredients.

Let $(G_\Sigma, M_\cL, \Gamma_\cL)$ be a   spanning GIFS satisfying the first and the second   separation condition.

  \subsection{Bright area and its cut points}
 Set
 \begin{equation}\label{eq:Harbin}
{\mathcal H}_{\ell}= \left(\bigcup_{h\in {\mathcal R}(\Gamma_\ell)} V_\ell(h) \right)\cup \left( \bigcup_{\gamma \subset \Gamma_\ell \text{ is a primary arc }} N(\gamma)\right)
\end{equation}
 and we call it the \emph{bright area} of $M_\ell$.
Clearly,   ${\cH}_{\ell}$ is compact  and  path connected, since  ${\cH}_{\ell}$ is a union of path connected sets each of which intersects the dendrite $\Gamma_\ell$  (see Proposition \ref{connected components proposition}).

The theorem below finds the cut points of ${\cH}_{\ell}$.

\begin{theo}\label{lem:cuts-GIFS}
Let $h\in {\mathcal R}(\Gamma_\ell)$ and $\gamma\subset \Gamma_\ell$ be a primary arc related to $h$.
 Denote $a:=V_\ell(h)\cap N(\gamma)$.    Then ${\cH}_{\ell}\setminus\{a\}$
consists of two  connected components, one containing $V_\ell(h)$ and the other containing $N(\gamma)$.
\end{theo}

\begin{proof}
 Suppose on the contrary ${\cH}_{\ell}\setminus\{a\}$ is connected.

 Since $a$ is an inner point of $\Gamma_{\ell}$,  $\Gamma_{\ell}\setminus\{a\}$ has  exactly two connected components.
Let $U_1$ be the component containing   $\gamma^\circ$, and   $U_2$ be
the component containing $V_\ell(h)\setminus\{a\}$.
Set
$$
 \hat U_1=\left (\bigcup_{(\gamma')^\circ\subset U_1 }N(\gamma')\right )\cup \left (\bigcup_{h'\in U_1\cap    {\mathcal R}(\Gamma_{\ell}) } V_\ell(h')\right )
  \setminus\{a\} ,
$$
and
$$
 \hat U_2=\left (\bigcup_{(\gamma')^\circ \subset U_2 }N(\gamma')\right )\cup \left (\bigcup_{h'\in U_2 \cap  {\mathcal R}(\Gamma_{\ell})} V_\ell(h')\right )  \setminus\{a\},
$$
where $(\gamma')^\circ$ denotes the interior of $\gamma'$.
Then   $\hat U_1$ is connected, since all the sets in the union are connected and each intersects $U_1$;
similarly, $\hat U_2$ is also connected.

 We claim that $\hat U_1$ and $\hat U_2$ are a separation partition of   ${\cH}_{\ell}\setminus\{a\}$.

 Suppose on the contrary  $x\in \text{closure}(\hat U_1)=\hat U_1\cup\{a\}$ and $x\in \hat U_2$, then $x\neq a$.
 In $\hat U_1$, $x$ either belongs to a village, which we denote by  $V_\ell(h_1)$, or belongs to
 the  ess-covering of a primary arc, which we
 denote by $N(\gamma_1)$. Similarly, in $\hat U_2$, $x$ either belongs
 to a village $V_\ell(h_2)$, or belongs to a  ess-covering  $N(\gamma_2)$.

 First, $x\in N(\gamma_1)\cap N(\gamma_2)$ is impossible, since all ess-coverings are disjoint
 by the second separation condition.

 Secondly, $x\in V_\ell(h_1)\cap V_\ell(h_2)$ is impossible, since villages are disjoint due to the
 first separation condition.

 Thirdly, if $h_1\neq a$ or $\gamma_2\neq \gamma$, then $x\in V_\ell(h_1)\cap N(\gamma_2)$ is impossible by the
 second separation condition.
 If $h_1=a$ and $\gamma_2=\gamma$, then we have $x=a$ which is impossible.
 This proves our claim as well as the theorem.
 \end{proof}

  We set
 \begin{equation}\label{eq:K-ell}
 K_\ell=\{a;~\text{$a$ is an end point of a primary arc in $M_\ell$}\}.
 \end{equation}
 Then every point in $K_\ell$ is a cut point of ${\mathcal H}_\ell$.

 \subsection{Induced chain }

  Let $\gamma$ be an arc in $\Gamma_\ell$ with end points $a$ and $b$ such that $a,b\in X_{\ell,1}$.
   Then $\gamma$ is cut into sub-arcs  by the cylinders contained in $N(\gamma)$, and each sub-arc   defines a button
in $N(\gamma)$ in the following way: if a sub-arc is $\gamma_{c_1c_2}$ where $c_1,c_2\in M_e$,  then we
define a button $(M_e, c_1, c_2)$. Hence, there is a first order chain in $N(\gamma)$ joining $a$ and $b$.
Since $\gamma$ is a Jordan arc, this chain is unique and we denote it by   ${\mathbf E}_\gamma$.
See Figure \ref{fig:chain-arc}.

\begin{figure}[h]
  \centering
  \includegraphics[width=5 cm]{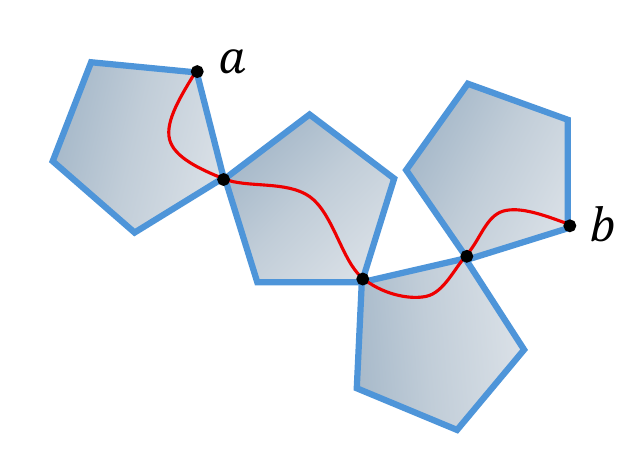}\\
   \caption{Chains associated to primary arcs.}\label{fig:chain-arc}
\end{figure}

\subsection{Corner evolution}
Now we define the corner evolution of a spanning GIFS.
 Corner evolution is introduced  by  Tyson and Wu \cite{TW_2006}, and it plays a crucial role in our discussion.


We define
$$
\Delta_1=(\Sigma\setminus \partial \Sigma)\cup(\partial \Sigma\otimes \Sigma),
$$
and we call $\cG^{(1)}=G_{\Delta_1}$ the \emph{first corner evolution} of $G_\Sigma$.

\begin{figure}[h]
  \centering
  \subfigure[]{
  \includegraphics[width=6 cm]{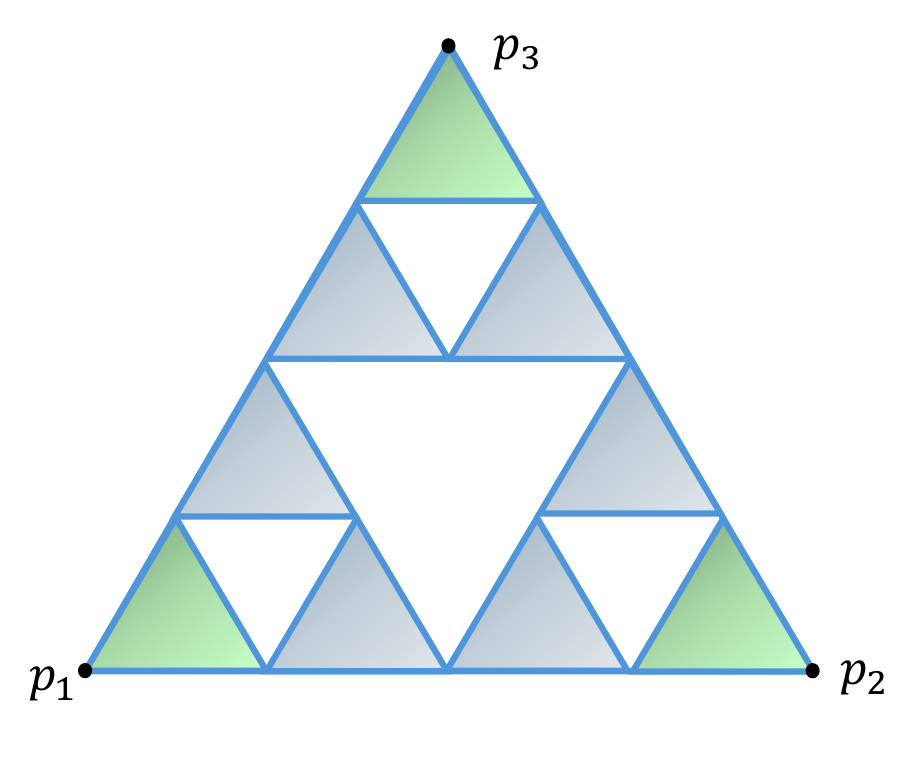}
  }
  \subfigure[]{
  \includegraphics[width=6 cm]{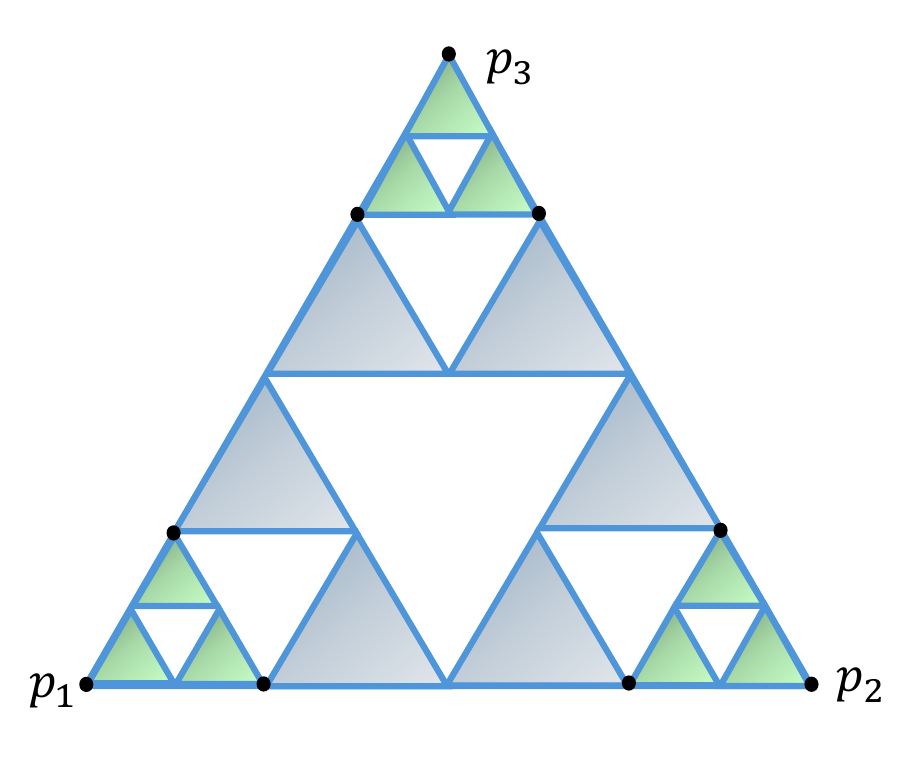}
  }\\
  \caption{The first and second corner evolutions of the Sierpi$\acute{\text{n}}$ski carpet: (a) the first corner evolution; (b) the second corner evolution.}
  \label{fig:corner evolution}
\end{figure}

For $m\geq 2$, inductively, we define
$$
\Delta_m=\left(\Delta_{m-1}\setminus \partial \Delta_{m-1}\right)\cup(\partial \Delta_{m-1}\otimes \Sigma),
$$
and we call $\cG^{(m)}=G_{\Delta_m}$ the \emph{$m$-th corner evolution} of $G_\Sigma$.
{\color{red}Figure \ref{fig:corner evolution} demonstrates the first and second corner evolutions of the Sierpi$\acute{\text{n}}$ski carpet.}


\section{$\tau$-functions on $P_{\cL}$}\label{SECTION8}

From Section \ref{SECTION8} to \ref{SECTION12}, we construct valuation pairs on the $m$-th corner evolution of $(G_\Sigma, M_{\cL}, \Gamma_{\cL})$.
 Given an integer $m$ big enough, we will construct a metric on $P_\ell$ related to $m$.

  \begin{defi}[Admissible triple]\emph{We call $(\ell, p, p')$ an \emph{admissible triple} if it belongs to
$$ \bigcup_{\ell\in \cL} \{\ell\}\times \left (P_\ell^2\setminus \diag P_\ell^2\right ).$$}
\end{defi}

In the following, we will call $(\ell, p,p')$ a \emph{triple} instead of an admissible triple for simplicity.

Recall $\varphi$ is the weight function defined by $R$ and $\tau$ in Section \ref{SECTION4}.
In what follows, we will use  $\tau(\ell, p,p')$ instead of  $\tau_\ell(p,p')$, to denote the distance between $p$ and $p'$
in the metric $\tau_\ell$.

Let  $(M_e,a,b)$ be a button of order $1$. We define
\begin{equation}\label{eq:S0}
S_0(M_e,a,b)=\tau(t(e), g_e^{-1}(a), g_e^{-1}(b)).
\end{equation}
Let
$E=\sum_{j=1}^k (M_{e_j},a_j,b_j)$ be a ${\mathcal B}_{\ell,1}$-chain  in $M_\ell$.
We define
$$S_0(E)=\sum_{j=1}^k S_0(M_{e_j},a_j,b_j).$$

Let  $\eta_{\ell,p,p'}$ be the arc in $\Gamma_\ell$ joining $p$ and $p'$.
Let
$$k(\ell, p,p')$$ be the   number of  primary arcs in $\eta_{\ell,p,p'}$.
Let $h$ be a ramification point of $\Gamma_\ell$ located in $\eta_{\ell, p,p'}$.
Suppose there are two primary arcs in $\eta_{\ell,p,p'}$ intersecting $V_\ell(h)$, and we denote
the intersections to be $a$ and $b$, respectively.
We define
$$
S_{\ell,p,p'}(V_\ell(h))=\left \{
\begin{array}{ll}
 S_0(M_e, a,b), & \text{if $a,b\in M_e$}; \\
 S_0((M_e,a, h), (M_{e'}, h,b)),& \text{if $a\in M_e$ and $b\in M_{e'}$}.
\end{array}
\right .
$$

Denote $N=\#\Sigma$.
Regarding $\tau$   as unknowns,
we require that $\tau$ satisfy the following equations:
\begin{equation}\label{tauB-GIFS}
\tau(\ell,p,p')  =2+\frac{k(\ell,p,p')}{mN}+\frac{\sum_{h\in \eta_{\ell,p,p'}} S_{\ell,p,p'}(V_\ell(h))}{2mN}.
\end{equation}
 Let $(\ell,p,p')$ run over all  triples,
 then \eqref{tauB-GIFS} gives us a linear equation system
with unknowns $\tau(\ell,p,p')$.

 \begin{rema}\rm The coefficients in \eqref{tauB-GIFS} come from the following fact:
In Section \ref{SECTION11}, we will define a ratio function $R$ satisfying

 (i) if $M_e$ belongs to  an inner village, then  $R(e)=1/(2mN)$;

 (ii) For every primary arc $\gamma$, $\varphi({\mathbf E}_{\gamma})=1/(mN)$.
\end{rema}

\begin{lem}\label{lem:linear-GIFS} If    $m$ is  an integer satisfying $m\geq  N^2$,
then  there is a unique metric $\tau$ satisfying  \eqref{tauB-GIFS}  and
 $$2< \tau_{\min}\leq \tau_{\max}< 3 .$$
\end{lem}
\begin{proof}
Recall that $\Omega_V:=\bigcup_{h\in\mathcal R_0(\Gamma_\ell)}{\mathcal V}_\ell (h)=\{e\in \Sigma;~ \text{ $M_e$ belongs to an inner village}\}.$
  Move all the terms containing unknowns on the right hand side of \eqref{tauB-GIFS}
  to the left hand side to obtain a standard form of linear equation.  Let $A$ be the coefficient matrix of the above linear equation system.

 Fix a triple $(\ell,p,p')$ and consider the equation given by \eqref{tauB-GIFS}.   First,  the coefficient of $\tau(\ell,p,p')$ is not greater than $1$,
   since the terms containing unknowns in the right hand side of \eqref{tauB-GIFS} are a non-negative linear combination of $1$ and all $\tau(\ell,p,p')$.
   On the other hand, the coefficient of $\tau(\ell,p,p')$   is no less than
$$
1-\frac{2\#\Omega_V}{2mN}> 1-\frac{1}{m}
$$
since $\#\Omega_{V}< N$. So every diagonal entry  of $A$ falls into
the interval $(1-1/m, 1]$.

By the same reason,  the non-diagonal entries of $A$ fall into the interval $(-1/m,0]$.

 Let  $v_0=(\#P_{\ell})^2-\#P_{\ell}$, which is the size of $A$.
Then $v_0<N^2$ and
 $$
 \det A\geq \left (1-\frac{1}{m}\right )^{v_0}-\frac{v_0!}{m^{v_0}}>0.
 $$
 So \eqref{tauB-GIFS} has a unique solution by Cramer's Rule.

 First, all  $\tau(\ell,p,p')>2$ since the right hand side of \eqref{tauB-GIFS} is a non-negative linear combination.
 Secondly, since $k(\ell,p,p')<N$,  we have
 $$\tau_{\max} \leq  2+\frac{1}{m}+  \frac{1}{2mN} \sum_{h\in \eta_{\ell,p,p'}} S_{\ell,p,p'}(V_\ell(h))\leq 2+\frac{1}{m}+  \frac{1}{2m} \tau_{\max},$$
 which implies $\tau_{\max} <3$.
\end{proof}

\begin{exam}\emph{Figure \ref{fig:Sier1}(b) illustrates an evolution of the Sierpi$\acute{\text{n}}$ski gasket IFS.
 We use this IFS to illustrate our idea.
}

\emph{ First, we set $R(\text{red})=1/5$, $R(\text{mag})=1/10$, and
$R(\text{dark})=3/4.$
We set big ratios for dark cylinders to expel edges of dark cylinders  from minimal paths.
Now we pretend that the minimal weights from $p_j$ to $b_j$, $j=1,2,3$, all equal $1$.  Hence, to fulfill the consistency condition, we have
$$
\left \{
\begin{array}{l}
\tau(p_1,p_2)=1+\varphi\left(\overline{[b_1,h_1]}\right)+\varphi\left(\overline{[h_1,b_2]}\right)+1=2+\frac{\tau(p_1,p_3)}{10}+\frac{\tau(p_1,p_2)}{10}\\
\tau(p_1,p_3)=1+\varphi\left(\overline{[b_1,h_1]}\right)+\varphi\left(\overline{[h_1,b_3]}\right)+1=2+\frac{2\tau(p_1,p_3)}{10}\\
\tau(p_2,p_3)=1+\varphi\left(\overline{[b_2,h_1]}\right)+\varphi\left(\overline{[h_1,b_3]}\right)+1=2+\frac{ \tau(p_1,p_2)}{10}+\frac{ \tau(p_1,p_3)}{10}.
\end{array}
\right .
$$
Solving the equations, we get  $\tau(p_1,p_2)=\tau(p_1,p_3)=\tau(p_2,p_3)=2.5$.}

\emph{For green cylinders in $\overline{[p_1, b_1]}$ (or in $\overline{[p_3, b_3]}$), we remove the weight
of $\overline{[p_1,a_1]}$ from $1$, and divide $4$ to obtain the weight of each small green segment, and divide by
$\tau(p_1,p_3)$ to get the ratio, so
$$
R(\text{left.green})=\frac{1-\tau(p_1,p_2)/5}{4\tau(p_1,p_3)}=\frac{1}{20}.
$$
Similarly, for green cylinders on the right, we set
$$
R(\text{right.green})=\frac{1-\tau(p_2,p_3)/5}{12\tau(p_1,p_3)}=\frac{1}{60}.
$$
}

\begin{figure}
  \centering
  \subfigure[]{
  \includegraphics[width=6. cm]{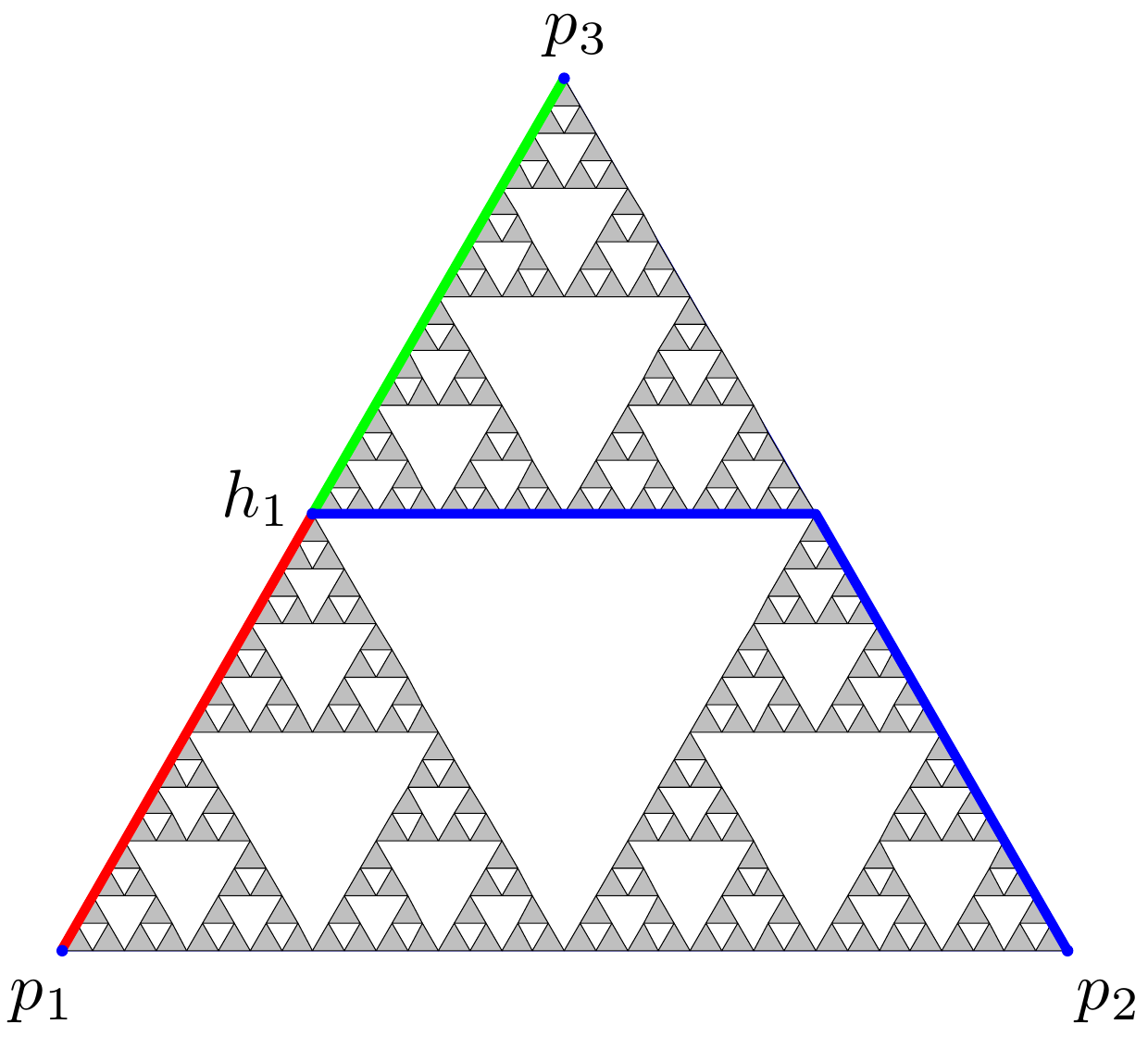}
  }
  \quad
  \subfigure[]{
  \includegraphics[width=6. cm]{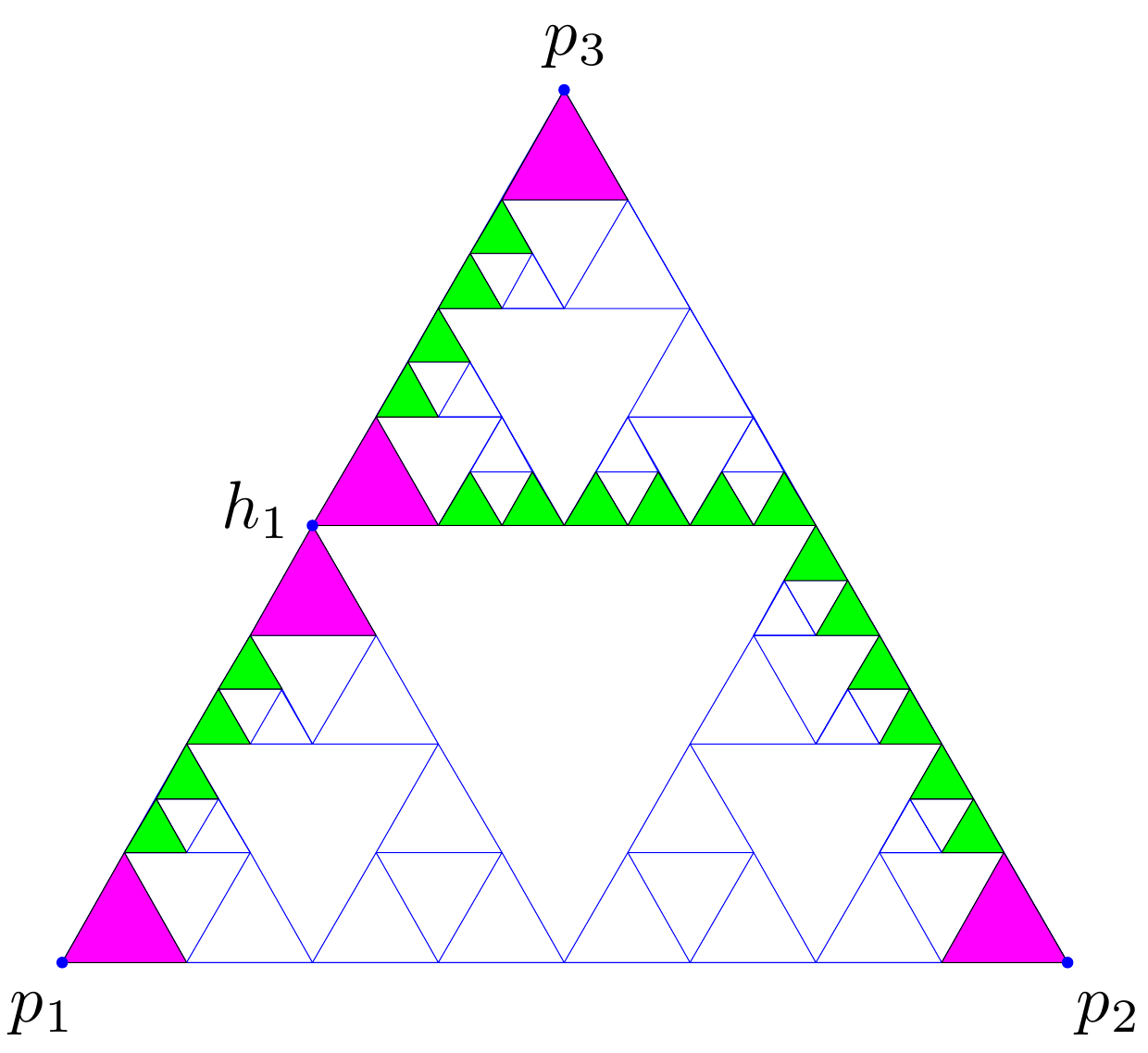}
  }\\
  \caption{(a)  A spanning dendrite of the Sierpi$\acute{\text{n}}$ski gasket. (b) Villages, bridges, bright cylinders and dark cylinders.}
  \label{fig:Sier1}
\end{figure}

\emph{
One can show that for $p_1,p_3$, the path along the line segment $\overline{[p_1,p_3]}$ has minimal weight,
and the weight is $2.5$.  For $p_1, p_2$, the path with minimal weight is   $\overline{[p_1,h_1]}+\overline{[h_1,u]}
+\overline{[u, p_2]}$, and the weight is $2.5$. For $p_2, p_3$, the path with minimal weight is   $\overline{[p_3,h_1]}+\overline{[h_1,u]}
+\overline{[u, p_2]}$, and the weight is $2.5$.  Hence the pair $(R,\tau)$ satisfies the consistency condition, and thus induces a metric on $K$ by Lemma \ref{K0K1}. $\Box$
}
\end{exam}

\subsection{Modified primary arc}\label{SECTION14}
Let
 $(G_\Sigma, M_\cL, \Gamma_\cL)$ be a   spanning GIFS satisfying the first and the second   separation condition.
 Let $\tau$ be  a metric function  on $P_{\cL}$.

Let $\gamma$ be a primary arc in $M_\ell$ with end points $a$ and $b$, and let  $N(\gamma)=\overline{\Omega(\gamma)}$ be the
ess-covering of $\gamma$.
Let us consider the class of  chains $H$ satisfying:

$(i)$ $H$ joins $a$ and $b$;

$(ii)$ The basic cylinders of $H$ are chosen from $\{M_e;~e\in \Omega(\gamma)\}$.

Let $\hat H$ be a chain such that $S_0(\hat H)$ attains the minimum in the above class.
We call $\hat H$  a \emph{$\tau$-optimal} chain w.r.t. $N(\gamma)$.

\begin{defi}\label{def:modified-arc}\rm  Let $\gamma$ be a primary arc.
 Let $\gamma'$ be a Jordan arc such that the induced chain of $\gamma'$ is $\tau$-optimal in $N(\gamma)$.
  We call $\gamma'$ a \emph{modified primary arc} w.r.t. $\gamma$.
\end{defi}

%
%

\section{Action arc and its evolutions}\label{SECTION9}
Let $(G_\Sigma, M_{\cL}, \Gamma_{\cL})$ be a spanning GIFS
  satisfying the first and second
separation conditions.
Let  $\eta_{\ell,p,p'}$ be the arc in $\Gamma_\ell$ joining $p$ and $p'$.

\subsection{Action arc}
We list the  primary arcs contained in $\eta_{\ell,p,p'}$ in the order
of their occurrence as
$$
\gamma_1,\dots, \gamma_k.
$$
Let $a_j$ and $b_j$ be the end points of $\gamma_j$.
For convenience, let us denote $b_0=p$ and $a_{k+1}=p'$.
Then for $j=0,\dots, k$, $b_j$ and  $a_{j+1}$ belong to the same village, which we denote by $V_\ell(h_j)$.

\begin{figure}[h]
  \centering
  \subfigure[]{
  \includegraphics[width=3.1 cm]{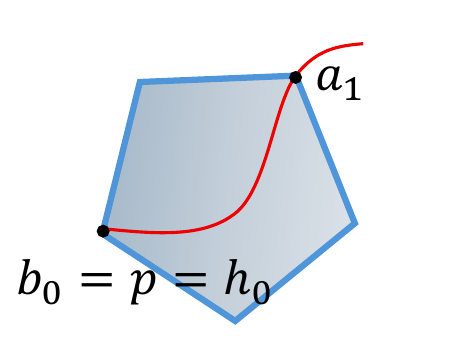}
  }
  \subfigure[]{
  \includegraphics[width=3.1 cm]{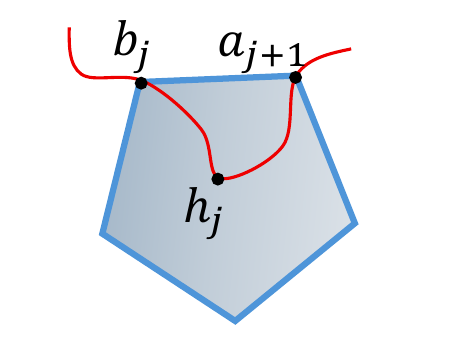}
  }
  \subfigure[]{
  \includegraphics[width=4 cm]{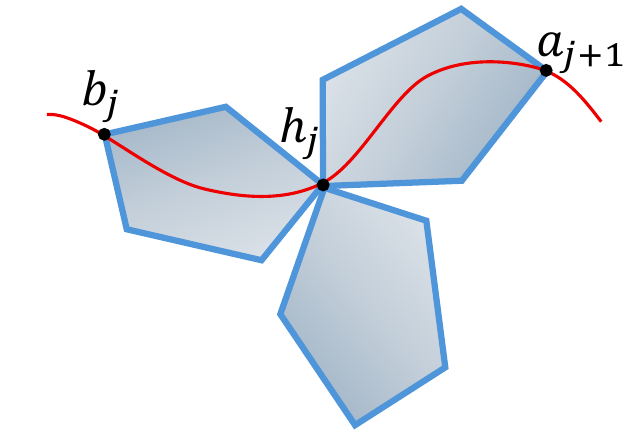}
  }
  \subfigure[]{
  \includegraphics[width=3.1 cm]{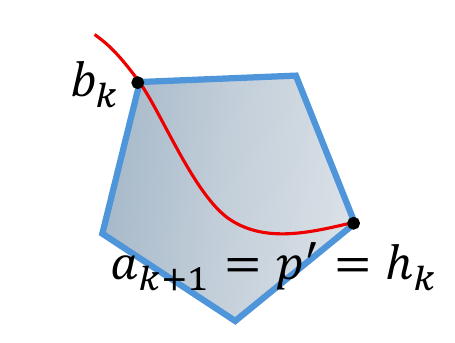}
  }\\
  \caption{The possible patterns of $b_{j}$ and $a_{j+1}$.}\label{fig:chain-village}
\end{figure}

Clearly $h_j\not \in \{b_j, a_{j+1}\}$ for $j=1,\dots, k-1$.
If $b_j$ and $a_{j+1}$ belong to two different cylinders, which we denoted by $M_e$ and $M_{e'}$,
then $h_j$ is a vertex of $M_e$ as well as $M_{e'}$.
Denote $\delta=t(e)$,   then $q=g_e^{-1}(b_j)$ and $q'=g_e^{-1}(h_j)$ belong to $P_\delta$. Therefore,
$\theta'_j=g_e(\eta_{\delta,q,q'})$ is a Jordan arc in $M_e$ joining $b_j$ and $h_j$.
 Similarly, there is a Jordan arc $\theta''_j$
in $M_{e'}$ joining $h_j$ and $a_{j+1}$. Denote
$$
\theta_j=\theta'_j+\theta''_j.
$$
Let $\gamma_j'$ be the modified primary arc of $\gamma_j$ for $j=1,\dots, k$.
Set
$$
{\color{red}\kappa_{\ell,p,p'}=\theta_{0}+\sum_{j=1}^k(\theta_j+\gamma'_j)},
$$
which is a Jordan arc joining $p$ and $p'$, and we call it the \emph{action arc} w.r.t. the triple  $(\ell,p,p')$.

\begin{defi}[Antipodal point]\rm  Let $p\in P_\ell$.
 Let $M_e$ be the village at $p$ in $G_\Sigma$, and let $\gamma$ be the primary arc  in $M_\ell$ joining $M_e$.
  We denote
  $$\{a_{p}\}=M_e\cap \gamma$$
and  call $a_{p}$ the \emph{antipodal point} of $p$.
\end{defi}

\begin{figure}[h]
  \centering
  \includegraphics[width=5 cm]{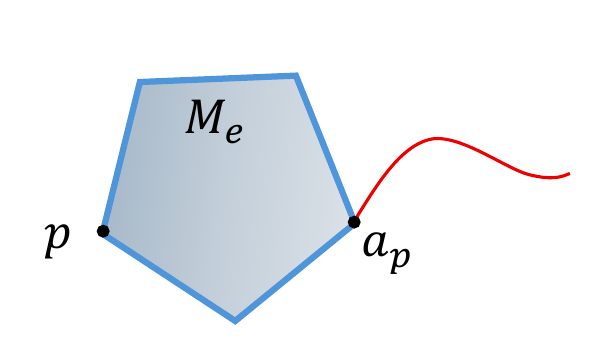}\\
   \caption{Antipodal point.}
\end{figure}


%
%
%
%
%
%
%

 \subsection{Arm evolution at $p\in P_\ell$}
Let $p\in P_\ell$ and   let  $a_p$ be the antipodal point of $p$.
Let $M_e$ be the $G_{\Sigma}$-cylinder containing both $p$ and $a_p$. We define
$$
A_{\ell,p}^{(1)}=g_e\left(\kappa_{t(e), g^{-1}_e(p), g_e^{-1}(a_p)} \right ),
$$
and call it the \emph{ first arm at $p$}.

Recall that $\cG^{(m)}$ is the $m$-th corner evolution of $G_\Sigma$.

Recursively, we define $A_{\ell,p}^{(m)}$ as follows.
Let $M_{\bomega}$ be  the ${\mathcal G}^{(m-1)}$-cylinder at $p$  essentially intersecting $A_{\ell,p}^{(m-1)}$.
Then $|\bomega|=m$.
  Let $u\in P_{\bomega}$ be the vertex of $M_{\bomega}$ other than $p$ located on  $A_{\ell,p}^{(m-1)}$.
 We define
$$
A_{\ell,p}^{(m)}= g_{\bomega}\left(\kappa_{t(\bomega), g_{\bomega}^{-1}(p), g_{\bomega}^{-1}(u)}\right)+ A_{\ell,p}^{(m-1)}\left([u, a_p]\right),
$$
and call it the \emph{$m$-th  arm at $p$}.

\paragraph{Notations.}
Let $M_{\bomega}$ be the corner cylinder at $p$ essentially intersecting $A_{\ell,p}^{(m)}$.
 Let $a'\in P_{\bomega}$ be the point other than $p$ located on $A_{\ell,p}^{(m)}$.
We denote by $\tilde T_{\ell,p}^{(m)}$ the chain induced by $A_{\ell,p}^{(m)}$ w.r.t. ${\mathcal G}^{(m)}$.
We denote
$$B_{\ell,p}^{(m)}=(M_{\bomega}, p, a')$$
which is the initial button of $\tilde T_{\ell,p}^{(m)}$, and
  denote by $T_{\ell,p}^{(m)}$ the chain obtained by deleting $B_{\ell,p}^{(m)}$ from $\tilde T_{\ell,p}^{(m)}$.

\begin{figure}[h]
  \centering
  \includegraphics[width=6.5 cm]{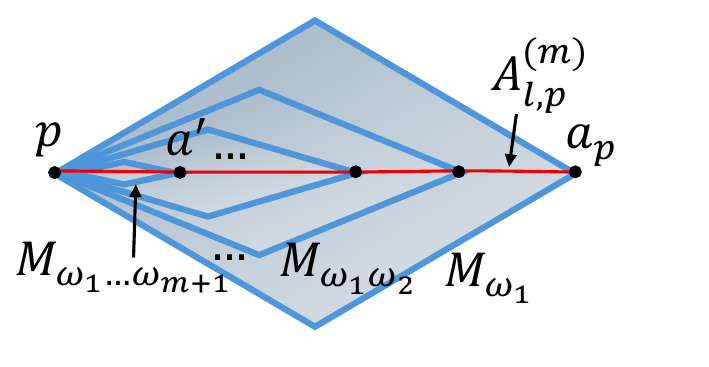}\\
   \caption{$B_{\ell,p}^{(m)}$ and $T_{\ell,p}^{(m)}$.}
\end{figure}

\subsection{Primary chains}
 Let $p, p'\in P_\ell$. We define
\begin{equation}
\displaystyle
\kappa^{(m)}_{\ell,p,p'}=A_{\ell,p}^{(m)}+\kappa_{\ell, p, p'}\left([a_p, a_{p'}]\right)+A_{\ell,p'}^{(m)},
\end{equation}
and call it  the  $m$-th evolution of $\kappa_{\ell,p,p'}$.

\begin{defi}[Primary chain] \rm
We denote by ${\mathbf H}_{\ell,p,p'}$ the chain induced by $\kappa_{\ell, p,p'}$ and call it the \emph{primary chain} joining $p$ and $p'$ in $G_\Sigma$.
We denote by ${\mathbf H}_{\ell,p,p'}^{(m)}$ the chain induced by $\kappa_{\ell, p,p'}^{(m)}$ and call it the \emph{$m$-th primary chain} joining $p$ and $p'$ in $G_{\Delta_m}$.
\end{defi}


\subsection{Bright areas}
We set
\begin{equation}
Q_{\ell,m}=\bigcup_{(\ell,p,p')} N\left(\kappa_{\ell, p,p'}^{(m)}\right)
\end{equation}
and call it the \emph{bright area} of $M_\ell$ w.r.t. ${\mathcal G}^{(m)}$. Moreover, a cylinder of $G_{\Delta_m}$
 is called a \emph{bright cylinder} if it belongs to the bright area $Q_{\ell,m}$. Set $Q_{\ell,0}=\bigcup_{(\ell,p,p')} N\left(\kappa_{\ell, p,p'}\right)$.

\begin{rema}\rm  (i) Indeed $Q_{\ell,m}$  is the union of all basic cylinders of all primary chains ${\mathbf H}^{(m)}_{\ell,p,p'}$.

(ii) It is possible that some members of inner villages are not contained in the bright area.
\end{rema}

The following lemma is obvious.

\begin{lem}\label{eq:QQQ}  $(Q_{\ell,m})_{m\geq 0}$ is decreasing, and
$Q_{\ell,0}\subset {\mathcal H}$,
where ${\mathcal H}$ is defined in \eqref{eq:Harbin}.
\end{lem}

 \begin{lem}\label{lem:S-minimal}
The primary chain  ${\mathbf H}_{\ell,p,p'}$  has $S_0$-minimal weight  among all
first order  chains of $\cG$
  joining $p$ and $p'$
consisting of  buttons in $Q_{\ell,0}$.
\end{lem}

\begin{proof}   By Theorem \ref{lem:cuts-GIFS}, the end points of primary arc of $G_\Sigma$ are cut points of $Q_{\ell,0}$.
These points cut $Q_{\ell,0}$ into several parts, each part is either contained in a village, or contained in the ess-covering
of a modified primary arc.  So we only need to show that
the corresponding sub-chains are all minimal.

\begin{figure}[h]
  \centering
  \includegraphics[width=9 cm]{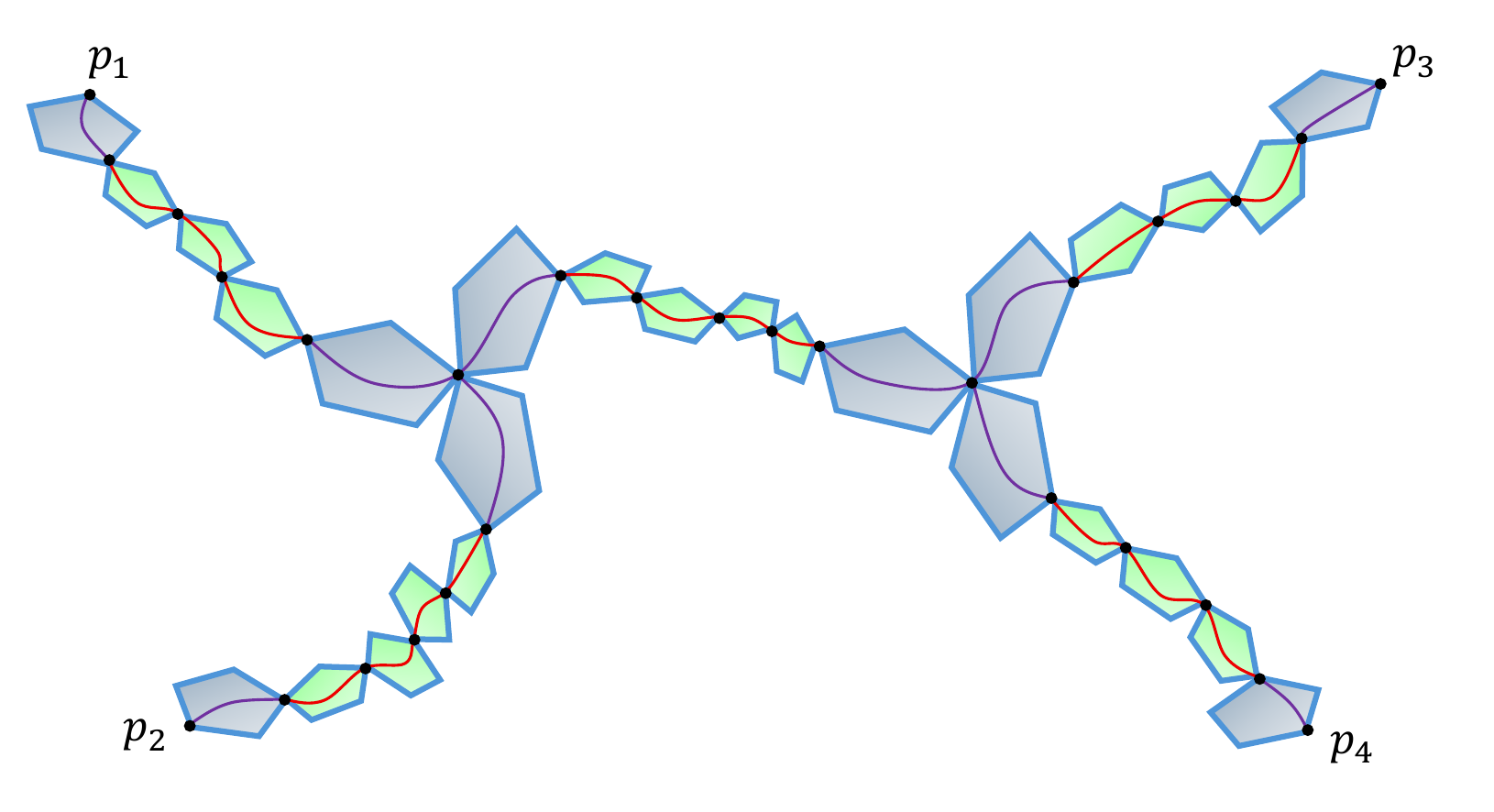}\\
   \caption{$Q_{\ell,0}$.}
\end{figure}

We claim that if $E$ is the maximal sub-chain of ${\mathbf H}_{\ell,p,p'}$ contained in a village $V_\ell(h)$, then $E$ is $S_0$-minimal.

 If $E=(M_e, a,b)$ consists of a single button, then it is minimal. For if $E'$ is another chain joining $a$ and $b$,
then $E'$ has at least two buttons. Since $\tau$ is between $2$ and $3$, we have
$$
S_0(E')\geq {4} \geq S_0(E).
$$
Now suppose $E=(M_e, a, h)+(M_{e'}, h,b)$. Since all cylinders in $V_\ell(h)$ meet only at $h$ (see item (iii)
of the first separation condition), if $E'$ is another chain joining $a$ and $b$, then $E'$ has at least $3$ buttons.
By the same reason as above, we obtain $S_0(E')\geq 6\geq S_0(E)$. Our claim  is proved.

We assert that if $E$ is the maximal sub-chain of ${\mathbf H}_{\ell,p,p'}$ contained in an ess-covering $N(\gamma)$ where $\gamma$ is
a modified primary arc,  then $E$ is minimal. By definition, this sub-chain is induced by
 $\gamma'$, where $\gamma'$ is the modified primary arc. By our choice of $\gamma'$, the  induced chain  is $S_0$-minimal
 in $N(\gamma)$, which proves our assertion.  The lemma is proved.
\end{proof}

\begin{lem}\label{mini-bri}
The primary chain  ${\mathbf H}_{\ell,p,p'}^{(m)}$  has $S_0$-minimal weight  among all
first order  chains of $\cG^{(m)}$
  joining $p$ and $p'$
consisting of buttons in $Q_{\ell,m}$.
\end{lem}

\begin{proof} We prove this lemma by induction on $m$.
For $m=0$, the lemma holds by Lemma \ref{lem:S-minimal}.

By Theorem \ref{lem:cuts-GIFS}, the end points of primary arc of $G_\Sigma$ are cut points of $Q_{\ell,m}$.
These points cut $Q_{\ell,m}\cap \overline{{\mathbf H}_{\ell,p,p'}^{(m)}}$ into several parts, and we only need to show that
the corresponding sub-chains are all minimal.

The antipodal points $a_p$ and $a_{p'}$ of $p$ and $p'$,  respectively, cut $Q_{\ell,0}$ into three parts,
 consequently, they also cut $ {\mathbf H}_{\ell,p,p'}^{(m)}$ into three parts: two corner parts and one central part.
  By Lemma \ref{lem:S-minimal},  $ {\mathbf H}_{\ell,p,p'}^{(m)}([a_p, a_{p'}])$  is minimal in $Q_{\ell,0}$, so it is also
 minimal in $Q_{\ell,m}$ since $Q_{\ell,m}\subset Q_{\ell,0}$.

Secondly,  we consider the corner part. Without loss of generality, we only   show
that ${\mathbf H}_{\ell,p,p'}^{(m)}([p,a_p])$ is minimal.

Let $M_{\bomega}$ be the corner cylinder at $p$ essentially intersecting $\kappa_{\ell,p,p'}^{(m)}$.
 Write $\bomega=\omega_1\dots \omega_{m+1}$ and denote $\bomega'=\omega_1\dots \omega_{m}$.

Let $b$ be the vertex of $P_{\bomega'}$ other than $p$. Then $b$ is a cut point of $Q_{\ell,m}$.

 Firstly, we have ${\mathbf H}_{\ell,p,p'}^{(m)}([b,a_p])={\mathbf H}_{\ell,p,p'}^{(m-1)}([b,a_p])$, and hence it is minimal in $Q_{\ell, m}$ since it is minimal
in $Q_{\ell,m-1}$ by the induction hypothesis.

Denote $q'=g_{\bomega'}^{-1}(b)$,  $q=g_{\bomega'}^{-1}(p)$,  $\delta=t(\bomega')$, then
$$g_{\bomega'}^{-1}\left({\mathbf H}_{\ell,p,p'}^{(m)}([p,b])\right)={\mathbf H}_{\delta,q,q'}.$$
By Lemma \ref{lem:S-minimal}, ${\mathbf H}_{\delta,q,q'}$ is minimal in $Q_{\delta,0}$.
Since
$$S_0\left({\mathbf H}_{\ell,p,p'}^{(m)}([p,b])\right)=S_0(g_{\bomega'}({\mathbf H}_{\delta,q,q'}))=S_0({\mathbf H}_{\delta,q,q'})$$ and
$$\overline{\left({\mathbf H}_{\ell,p,p'}^{(m)}([p,b])\right )}\subset g_{\bomega'}(Q_{\delta,0}),$$
we infer that ${\mathbf H}_{\ell,p,p'}^{(m)}([p,b])$ is $S_0$-minimal in $\overline{\left({\mathbf H}_{\ell,p,p'}^{(m)}([p,b])\right )}$, and also minimal in $Q_{\ell,m}$.
The lemma is proved.
\end{proof}


\section{\text{Constructing valuation pair for $\cG^{(m)}$}}\label{SECTION10}
 Let
 ${\mathcal G}=(G_{\Sigma}, M_\cL, \Gamma_{\cL})$ be a  spanning GIFS satisfying the first and the second   separation conditions.
 Let  $r:\Sigma\to (0,1)$ be a ratio function and we will call it a \emph{reference ratio function}.
 Let $r_{\max}$ be the maximum of $r$.

 Let $m$ be an integer satisfying $m\geq N^2$.
Recall that  ${\mathcal G}^{(m)}=G_{\Delta_m}$ is the $m$-th corner evolution
of ${\mathcal G}$. Let $\tau$ be the metric on $(P_\ell)_{\ell\in \cL}$ related to $m$  defined in Section \ref{SECTION8}.

In the following, we  construct a  ratio function $R: \Delta_m\to(0,1)$  related to $r$ and $\tau$.

\begin{defi}[Dark cylinder]\rm
 A cylinder of ${\mathcal G}^{(m)}$ is called a \emph{dark cylinder} if it is neither corner nor bright.
 \end{defi}

\subsection{Determine $R$}
 First, we  determine $R$ for corner, village and dark cylinders.
  Let $\bomega\in\Delta_m$. Denote
  $$s=-\log m/\log r_{\max}.$$
   We set
\begin{equation}\label{eq:ratio-1}
R(\bomega)=\left\{
\begin{array}{ll}
 r_{\bomega}^s,& \text{if $M_{\bomega}$ is a corner cylinder,}\\
 {1}/{(2mN)}, &\text{if $M_{\bomega}$ belongs to a bright inner village,}\\
 {2N}/{m}, &\text{if $M_{\bomega}$ is a dark cylinder.}
\end{array}
\right.
\end{equation}

Next, we determine the ratio of bright cylinder at corner area w.r.t. $G_\Sigma$.
   Recall that
$$ {\mathbf H}_{\ell,p,p'}^{(m)}([p,a_p])={B}_{\ell, p}^{(m)}+  T_{\ell,p}^{(m)}$$
is the sub-chain of the ${\mathbf H}_{\ell,p,p'}^{(m)}$ located in the bright $G_{\Sigma}$-corner cylinder at $p$.
Be aware that we have set the ratios of  corner cylinders,  so $\varphi\left({B}_{\ell,p}^{(m)}\right)$ is known.

   We intend to make
   $$\varphi\left({B}_{\ell, p}^{(m)}+ T_{\ell,p}^{(m)}  \right)=1,$$
 so for cylinders $M_{\bomega}$ in $T_{\ell,p}^{(m)}$, we set
\begin{equation}\label{eq:ratio-arm}
R(\bomega)=\frac{1-\varphi\left({B}_{\ell, p}^{(m)}\right)}{S_0\left ( T_{\ell,p}^{(m)} \right )}.
\end{equation}

Finally, we determine the ratio of bright cylinders in the central area.
Let $\gamma$ be a primary arc,  let ${\mathbf E}_{\gamma}$ be the  chain induced by $\gamma$.
We intend to make
$$
\varphi({\mathbf E}_{\gamma})=\frac{1}{mN},
$$
so for $M_{\bomega}\subset N(\gamma)$, we set
\begin{equation}\label{eq:ratio-N(gamma)}
R(\bomega)=\frac{1}{mNS_0\left( {\mathbf E}_{\gamma}\right)}.
\end{equation}

Till now, the function $R$ is completely determined. Summarizing the above assignments, we have

\begin{lem}\label{lem:R-value} Let $R:\Delta_m\to(0,1)$ be the ratio function defined above. Then

$(i)$ If $M_{\bomega}$ is a bright non-corner cylinder, then $R(\bomega)\leq 1/(2m)$.

$(ii)$  If $M_{\bomega}$ is a corner cylinder, then $R(\bomega)=r_{\bomega}^s$;

$(iii)$ {\color{red}$R(\bomega)=R_{\max}$ if  $M_{\bomega}$ is dark, and $R_{\bomega}\leq 1/m$ otherwise.}
\end{lem}

\begin{proof} (i) Let $M_{\bomega}$ be a bright cylinder.
If $M_{\bomega}$ belongs to the corner area of $p$ w.r.t. $G_\Sigma$, then $R(\bomega)$ is defined by  \eqref{eq:ratio-arm}.
Since the number of buttons in $ T_{\ell,p}^{(m)}$   is  no less than $m$, we have
 $S_0\left ( T_{\ell,p}^{(m)} \right )\geq 2m$ and hence
$$R(\bomega)\leq \frac{1}{2m}.$$

If $M_{\bomega}$ belongs to a village, then $R(\bomega)=1/(2mN)$.
If $M_{\bomega}$ belongs to an ess-covering, then $R(\bomega)\leq 1/(2mN)$ by \eqref{eq:ratio-N(gamma)}.
Item (i) is proven.

(ii) is trivial.

(iii) If  $M_{\bomega}$ is a corner cylinder, then
$r_{\bomega}^s\leq 1/m\leq 2N/m$. The lemma is proved.
\end{proof}

\section{Consistency}\label{SECTION11}

In this section, we confirm that the valuation pair $R,\tau$ defined in the previous sections is consistency.

 \begin{lem}\label{minbri}
The primary chain  ${\mathbf H}_{\ell,p,p'}^{(m)}$  has minimal weight  among all
first order  chains of $\cG^{(m)}$
  joining $p$ and $p'$
consisting of bright buttons.
\end{lem}

\begin{proof}  Let $M_{\bomega}$ be the corner cylinder at $p$ essentially intersecting $\kappa_{\ell,p,p'}^{(m)}$.
Let $b\neq p$ be the other vertex of $M_{\bomega}$ on $\kappa_{\ell,p,p'}^{(m)}$.
Similarly, for $p'$, we define a point $b'$.
Recall that
$$K_\ell=\{a;~\text{$a$ is an end point of a primary arc in $M_\ell$}\}.$$
Then $K_\ell\cup\{b,b'\}$ cut $Q_{\ell,0}$ into several connected parts,
which we denote by $Q_1,\dots, Q_k$. The set $K_\ell\cup\{b,b'\}$
 also cut ${\mathbf H}_{\ell,p,p'}^{(m)}$ into several sub-chains.
By Lemma \ref{mini-bri}, these sub-chains are $S_0$-minimal in $Q_{\ell,m}$, hence they are also
$\varphi$-minimal, since in each $Q_j$, the cylinders have uniform ratio.
  The lemma is proved.
\end{proof}

\begin{lem}\label{lem:pedal-m-GIFS} Let   $p\in P_\ell$, and let  $M_{\bomega}$ be a  corner cylinders of $\mathcal G^{(m)}$ at $p$. If $M_{\bomega}$ is not bright, then
\begin{equation}\label{lem:pedal-p-GIFS}
M_{\bomega}\cap Q_{\ell,m}=\{p\}.
\end{equation}
 \end{lem}

\begin{proof}
 We prove the lemma by induction on $m$.

  Let $m=0$. Let $I\in \Sigma$ be the letter such that $p\in M_I$ and $M_I$ is a bright button in $Q_{m,0}$.  By
   the first separation condition,
  $M_{\bomega}$ does not intersects the village other than $M_I$, and $M_{\bomega}\cap M_I=\{p\}$.
   By item (ii) of the second separation condition, $M_{\bomega}$ does not intersect  ess-coverings of primary arcs.
   So the lemma holds in this case.

Let $m\geq 1$.
 Clearly, $|\bomega|=m+1$. Write $\bomega=\bomega' e$ and denote $\delta=t(\bomega')$.

 If  $M_{\bomega'}$ is not bright in ${\mathcal G}^{(m-1)}$, then  $M_{\bomega'}\cap Q_{\ell,m-1}=\{p\}$ by induction hypothesis.
 Since $Q_{\ell,m}\subset Q_{\ell,m-1}$ and $M_{\bomega}\subset M_{\bomega'}$, we obtain \eqref{lem:pedal-p-GIFS}.

 If $M_{\bomega'}$ is  bright in ${\mathcal G}^{(m-1)}$, then $M_{e}$ is a non-bright corner cylinder of ${\mathcal G}$ at $q=g_{\bomega_1}^{-1}(p)$, so  $M_{e}\cap Q_{\delta,0}=\{q\}$ by induction hypothesis.
 On the other hand,  let $B_{\ell,p}^{(m-1)}=(M_{\bomega'}, p, a)$ be the corner button of ${\mathbf H}_{\ell,p,p'}^{(m-1)}$ at $p$, then
 $$
 Q_{\ell,m} \cap M_{\bomega'}=g_{\bomega'}({\mathbf H}_{\delta,p,a})\subset g_{\bomega_1}(Q_{\delta,0}).
 $$
 It follows that
 $$  M_{\bomega}\cap Q_{\ell,m}= M_{\bomega}\cap \left(Q_{\ell,m} \cap M_{\bomega'}\right)  \subset g_{\bomega'}\left(M_{e}\cap Q_{\delta,0}\right)=\{p\}. $$
 The lemma is proved.
\end{proof}

\begin{theo}\label{minresm} Let  $(R, \tau)$  be the valuation pair    defined as above.  Then
for each triple $(\ell,p,p')$, the primary chain $\mathbf H^{(m)}_{\ell,p,p'}$ is  minimal.
Moreover, the valuation pair  $(R,\tau)$ satisfies the consistency condition.
\end{theo}

\begin{proof} Notice that  for a button ${\mathbf B}$ in $G_{\Delta_m}$, we have
$\varphi({\mathbf B})\geq 4N/m$
if ${\mathbf B}$ is dark,
and
$\varphi({\mathbf B})<3/m$
 otherwise. (See Lemma \ref{lem:R-value}.)

Let   $E$ be a minimal chain    joining $p$ and $p'$.

 Let $B_0=(M_{\bomega}, p, u)$ be the first button of $E$.
Since $\bomega\in \partial \Delta_m$, by the corner evolution process, $\bomega$ can be written as
 $\bomega=\omega_0\omega_1\dots \omega_m$, where $\omega_0\in \partial \Sigma$ and
 $$\omega_0\omega_1\dots \omega_k\in \partial \Delta_k, \quad \text{ for }1\leq k\leq m.$$
 By the p.c.f. set sparse property, we have
 \begin{equation}\label{eq:nested}
 P_{\omega_0\omega_1\dots \omega_{k-1}}\cap P_{\omega_0\omega_1\dots \omega_k}=\{p\},\quad  \forall 1\leq k\leq m.
 \end{equation}

 For a button $(M_{\bomega},x,y)$ where $\bomega\in \Sigma^*$, we call $|\bomega|$ the \emph{$\Sigma$-rank} of the button.

 Let $a$ be the point on $E$ such that $E([p,a])$ is the maximal sub-chain contained in $M_{\omega_0}$.
 For $k=2,\dots, m+1$, Let $E_k$ be a maximal sub-chain of $E([p,a])$ consisting of button of $\Sigma$-rank $(k)$.
 Then    the previous button of
  $E_k$ has $\Sigma$-rank $(k+1)$, and the   button  next to $E_k$ has $\Sigma$-rank $(k-1)$.
  By \eqref{eq:nested}, $E_k$ is not empty.

 {\color{red}First, we argue that $M_{\omega_0}$ is a bright cylinder in $Q_{\ell,0}$, the bright area of $G_\Sigma$. For otherwise, all the chains $E_k$ consist of dark buttons, so
 $$
 \varphi(E)\geq m\cdot   \frac{4N}{m} >4>\tau(p,p'),
 $$
 which violates the minimality of $E$.}

  Next, we decompose ${\mathbf H}^{(m)}_{\ell,p,p'}([p, a_{p}])=T_{\ell, p}^{(m)}+B_{\ell,p}^{(m)}$
  into sub-chains according to ranks of buttons.
 Denote   $T_{\ell,p}^{(0)}=\emptyset$ for simplicity. Then
 $$
{\mathbf H}^{(m)}_{\ell,p,p'}([z_k, z_{k-1}]) :=T_{\ell,p}^{(k)}\setminus T_{\ell,p}^{(k-1)}, \quad k=1,\dots, m,
 $$
 is the maximal sub-chain of  ${\mathbf H}^{(m)}_{\ell,p,p'}([p, a_{p}])$ consisting of buttons of $\Sigma$-rank $(k+1)$,
 and
 $$B_{\ell, p}^{(m)}={\mathbf H}^{(m)}_{\ell,p,p'}([p, z_{m}])$$
  is the maximal sub-chain of  ${\mathbf H}^{(m)}_{\ell,p,p'}([p, a_{p}])$ consisting of buttons of $\Sigma$-rank $(m+1)$.

\begin{figure}[h]
  \centering
  \includegraphics[width=8 cm]{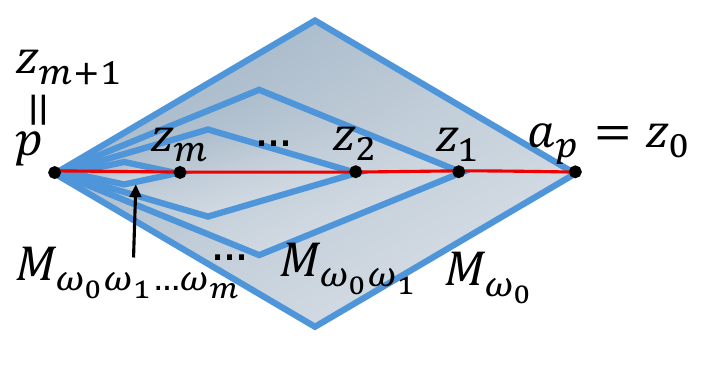}\\
   \caption{The points $z_0=a_p, z_1,\dots, z_{m+1}=p$.}
\end{figure}

We claim that
 \begin{equation}\label{eq:Um}
\varphi(E_k)\geq   \varphi\left( {\mathbf H}^{(m)}_{\ell,p,p'}([z_{k-1}, z_{k-2}])\right), \quad \text{ for } 2\leq k\leq m.
\end{equation}

If   $E_k$ contains a dark button,  \eqref{eq:Um}  holds since
$$
\varphi(E_k)\ge   \frac{4N}{m}>     \frac{3}{m}\cdot N\geq        \varphi\left( {\mathbf H}^{(m)}_{\ell,p,p'}([z_{k-1}, z_{k-2}])\right).
$$

If $E_k$ does not contain dark buttons, then  $E_k$ locates in $ \overline{{\mathbf H}^{(m)}_{\ell,p,p'}([z_{k-1}, z_{k-2}])}$.
Denote  $E_k$  as  $E([c_k,d_k])$.
Since $c_k$ joins a button of $\Sigma$-rank $(k+1)$, so $c_k$  is a vertex of a corner cylinder of $G_{\Delta_{k-1}}$.
Applying Lemma  \ref{lem:pedal-m-GIFS} to $Q_{\ell,k}$, we conclude that
the corner cylinder must be bright, then  we have $c_k=z_{k-1}$ since $z_{k-1}$ is a cut point of $Q_{\ell,k}$. On the other hand, $d_k$ joins a button of $\Sigma$-rank $(k-1)$,
so $d_k$ is a vertex of a cylinder of $\Sigma$-rank $(k-1)$, which forces $d_k=z_{k-2}$.
So  \eqref{eq:Um} holds by the minimality of ${\mathbf H}^{(m)}_{\ell,p,p'}$ in $Q_{\ell,m}$.
Our claim is proved.

By the same reason as above, we have
$$
\varphi(E_m)\geq \varphi\left({\mathbf H}^{(m)}_{\ell,p,p'}([p, z_{m-1}])\right),
$$
which together with \eqref{eq:Um} imply that
\begin{equation}\label{eq:arm-p}
\varphi( E([p,a]) )=\varphi\left({\mathbf H}^{(m)}_{\ell,p,p'}([p, a_p])\right).
\end{equation}

 Similarly, let $M_{e'}$ be the $\Sigma$-cylinder essentially intersecting $\kappa_{\ell, p,p'}$.
  Let $b$ be the point such that $E([p',b])$ is the maximal sub-chain contained in $M_{e'}$.
  By the same argument as above, we have
  \begin{equation}\label{eq:arm-p'}
\varphi( E([p',b]) )\geq \varphi\left({\mathbf H}^{(m)}_{\ell,p,p'}([p', a_{p'}])\right).
\end{equation}

Now we consider the sub-chain  $E([a,b])$.
  If $E([a,b])$ contains a dark button, then
 \begin{equation}\label{eq:center-1}
 \varphi(E([a,b]))\geq \frac{4N}{m}\geq \varphi\left({\mathbf H}_{\ell,p,p'}([a_{p},  a_{p'}])\right)=\varphi\left({\mathbf H}^{(m)}_{\ell,p,p'}([a_{p},  a_{p'}])\right);
 \end{equation}
 if $E([a,b])$ does not contains a dark button, since $a_p$ and $a_{p'}$
 are bi-part cut points of $Q_{\ell,m}$, $E([a,b])$ has a sub-chain joining $a_p$ and $a_{p'}$
 in $Q_{\ell,m}$, so \eqref{eq:center-1} still holds.

 Hence, by \eqref{eq:arm-p}, \eqref{eq:arm-p'} and \eqref{eq:center-1}, we have $\varphi(E)\geq \varphi({\mathbf H}^{(m)}_{\ell,p,p'}),$  which confirms  that  ${\mathbf H}^{(m)}_{\ell,p,p'}$ is minimal.

Finally, we calculate $\varphi\left({\mathbf H}^{(m)}_{\ell,p,p'}\right)$. Recall that
\begin{equation}\label{tauB}
\tau(\ell,p,p')  =2+\frac{k(\ell,p,p')}{mN}+\frac{\sum_{h\in \eta_{\ell,p,p'}} S_{\ell,p,p'}\left(V_\ell(h)\right)}{2mN}.
\end{equation}

By \eqref{eq:ratio-arm}, we have
$$\varphi\left({\mathbf H}^{(m)}_{\ell,p,p'}([p,a_p])\right)=\varphi\left({\mathbf H}^{(m)}_{\ell,p,p'}([p,a_p])\right)=1.$$
By \eqref{eq:ratio-N(gamma)}, the buttons belonging to ess-coverings contribute $\frac{k(\ell,p,p')}{mN}$ to $\varphi\left({\mathbf H}^{(m)}_{\ell,p,p'}\right)$.
The buttons belonging to inner villages contributes $\left(\sum_{h\in \eta_{\ell,p,p'}} S_{\ell,p,p'}\left(V_\ell(h)\right)\right)/(2mN)$.
Therefore, by  \eqref{tauB}, we have $\varphi\left(  {\mathbf H}^{(m)}_{\ell,p,p'}\right)=\tau(\ell,p,p')$, which confirms the consistency condition.
The theorem is proved.
\end{proof}

\section{Proof of the main theorem}\label{SECTION12}

\begin{theo}\label{thm:KP-GIFS}
Let
 ${\mathcal G}=(G_{\Sigma}, M_{\cL}, d_{\cL})$  be a nearly self-similar p.c.f. GIFS satisfying the angle separation condition.
 If  $M_\ell\setminus P_\ell$ are connected for all $\ell\in \cL$, then $\dim_C M_\ell=1$.
\end{theo}

\begin{proof} Let $\cL=\cL'\cup \cL''$
where  $\cL'$ is the collection of non-isolated states, and $\cL''$ is the collection of isolated states.
 Let $\Sigma'=\Sigma|_{\cL'}$ be the restriction
of the graph $\Sigma$ on $\cL'$. Then
$$
\Sigma=\Sigma'\cup \Sigma'',
$$
where $\Sigma''=\bigcup_{\ell\in \cL''} \Sigma_{\ell}$.

\textbf{Step 1.} Remove the isolated states.

Clearly $\cG'=(G_{\Sigma'}, (M_\ell)_{\ell\in \cL'})$
is a sub-GIFS of $\cG$.  Since $M_\ell\setminus P_\ell$ are connected, we can choose dendrites $\Gamma_\ell$ such that
$\mathcal E(\Gamma_\ell)=P_\ell$,  then
 $$\cG'=(G_{\Sigma'}, M_{\cL'}, \Gamma_{\cL'})$$
  is a spanning GIFS.

\textbf{Step 2.} Evolutions.

 By Theorem \ref{thm:evolution}, there is an evolution $G_\Omega$ of $G_{\Sigma'}$, such that
 $$\cG''=(G_\Omega, M_{\cL'}, \Gamma_{\cL'})$$
 is a spanning system satisfying the first and the second separation condition.

Denote $N=\#\Omega$.  Let $m\geq N^{2}$ be an integer.
Let
\begin{equation}
 {\mathcal G} ^{(m)}=(G_{\Delta_m}, M_{\cL'})
\end{equation}
  be the $m$-th corner evolution of ${\mathcal G}''$.

  For $\ell\in \cL'$, let $\tau_\ell$ be the metric on $P_{\ell}$ associated with $m$ defined in Section \ref{SECTION8}.

Let  $r:~\Sigma\to(0,1)$ be  the ratio function of $\cG$ (since $\cG$ is nearly self-similar).
  Let $R:\Delta_m\to(0,1)$ be the ratio function defined in Section \ref{SECTION10} related to  $\tau$ and $r$.
Then $(R, \tau)$ is a  valuation pair on ${\mathcal G}^{(m)}$.

By Theorem \ref{minresm},  $(R,\tau)$ satisfies the consistency condition. Let $(M_\ell, D_\ell)_{\ell\in\cL'}$
be the metrics induced by  $(R,\tau)$. We denote the GIFS in the new metric by
\begin{equation}
 {\mathcal G}_{R,\tau}^{(m)}=(G_{\Delta_m}, (M_\ell, D_\ell)_{\ell\in \cL'}).
\end{equation}
By Theorem \ref{thm:R-tau} ,  we have that

(i) ${\mathcal G}_{R,\tau}^{(m)}$ is a nearly self-similar GIFS if we set $R(I)$ to be the contraction ratio of $g_I$
 for all $I\in \Delta_m$;

(ii)  ${\mathcal G}_{R,\tau}^{(m)}$ satisfies the angle separation
condition;

(iii)   For $I\in  C(\Delta_m)$,  $R(I)=r_I^{s}$.

\textbf{Step 3.} Put back the isolated part.

Now we consider the system
$$
(G_{\Delta_m}\cup \Sigma'', (M_\ell)_{\ell\in \cL}),
$$
by adding the isolated states and related edges to $G_{\Delta_m}$.  Denote
$$
H_1=\left(G_{\Delta_m}\cup \Sigma'', (M_\ell)_{\ell\in \cL}, (d_\ell)_{\ell\in \cL}\right).
$$
Then $H_1$,  as an evolution of the original GIFS,  satisfies the angle separation condition
and is nearly self-similar with the ratio function $r$.

Let $R':G_{\Delta_m}\cup \Sigma''\to (0,1)$ be a ratio function defined by $R'(e)=R(e)$ if $e\in \Delta_m$, and
$R'(e)=r(e)$ if $e\in \Sigma''$. Let $\tau$ be the metrics defined on $P_\ell, \ell\in \cL'$ as above.
Then $(R',\tau)$ satisfies the consistency condition, since the ratios of $\Sigma''$ are irrelevant to this condition.
Let $(D_\ell')_{\ell \in \cL}$ be the metrics induced by $(R',\tau)$. Then clearly $D'_\ell=D_\ell$ for $\ell\in \cL'$.
Therefore, we can denote $(D_\ell')_{\ell \in \cL}$ by $(D_\ell)_{\ell \in \cL}$ by setting $D_\ell=D'_\ell$ for $\ell\in \cL'$.

By Theorem  \ref{thm:R-tau},
$$
H_2=\left(G_{\Delta_m}\cup \Sigma'', (M_\ell)_{\ell\in \cL}, (D_\ell)_{\ell\in \cL}\right)
$$
 satisfies the angle separation condition and is nearly self-similar  with the ratio function $R'$.

 Since $\partial(G_{\Delta_m}\cup \Sigma'')=\partial(G_{\Delta_m})\subset \partial\Delta_m$,
and $R(I)=r_{I}^s$ for $I\in \partial\Delta_m$,
By Theorem \ref{main-1-GIFS},
 $ (M_\ell, d_\ell)$ and $(M_\ell, D_\ell)$ are quasisymmetrically equivalent, and hence
  $$\dim_C M_\ell \leq \dim_H (M_\ell, D_\ell).$$

\textbf{Step 4.} Dimension estimates.

Denote $h=\sum_{\ell\in \cL} \#\mathcal{P}_\ell$. Then
$$N_m:=\#\Delta_m\leq \#\Delta_{m-1}+h(N-1).$$
It follows that $N_m\leq N+mh(N-1)$.
Since $R_{\max}=2N/m$ (see Lemma \ref{lem:R-value}), for $\ell\in \cL'$, we obtain
$$ \dim_H (M_\ell, D_\ell)  \leq\frac{\log {N_m}}{\log \frac{m}{2N}} \leq \frac{\log (m(h+1)N)}{\log \frac{m}{2N}}\rightarrow 1, \quad \text{ when }m\rightarrow\infty.$$
which implies that $\dim_C M_\ell=1$.

For $\ell\in \cL''$, notice that $M_\ell$ can be written as
$$M_{\ell}=\bigcup_{j=1}^k g_{e_{j}}(M_{\ell_j})$$
with $\ell_j\in \cL'$. Since $M_\ell$ is nearly self-similar in the metrics $(D_\ell)_{\ell\in \cL}$,
we obtain that $g_{e_{j}}$'s are bi-Lipschitz. Therefore,
$$\dim_H M_\ell \leq \max_{\delta\in \cL'} \dim_H M_\delta,
$$
and tends to $1$ as $m\to \infty$. Again we have $\dim_C M_\ell=1$.
The theorem is proved.
\end{proof}

\section{GIFS case}\label{SECTION13}

Let
$((g_e)_{e\in\Sigma}, (M_\ell)_{\ell\in \cL})$ be a   p.c.f. GIFS  such that  $M_\ell$ are connected but $M_\ell\setminus P_\ell$
may not be connected.  We shall use $[\ell, j]$ to denote a pair in $\cL\times \N$.

 Let $\tilde T_{[\ell,1]},\dots, \tilde T_{[\ell,{m_{\ell}]}}$ be the   connected components of $M_\ell\backslash P_\ell$, $\ell\in \cL=\{1,\dots, L\}$. Denote
\begin{equation}\label{eq:Lambda}
\Lambda_\ell=\left\{[\ell,j]; ~ 1\leq j\leq m_{\ell}\right\}, \quad  \Lambda=\bigcup_{\ell\in \cL}\Lambda_\ell.
\end{equation}

\begin{theo}\label{thm:transfer-GIFS}
Let
$((g_e)_{e\in\Sigma}, (M_\ell)_{\ell\in \cL})$ be a   p.c.f. GIFS  such that  $M_\ell$ are connected.
Let $(P_\ell)_{\ell\in\cL}$ be the posted-critical point sets of the GIFS.  Let $\{\tilde T_{\lambda}; \lambda\in \Lambda_\ell\}$ be the  collection of connected components of $M_\ell\backslash P_\ell$, $\ell\in \cL=\{1,\dots, L\}$.
Then there exists a   GIFS
 \begin{equation}\label{eq:transfer-gifs}
{\mathcal H}=((h_u)_{u\in \Omega}, (T_{\lambda})_{\lambda\in \Lambda})
\end{equation}
satisfying the following conditions:

$(i)$   $\{h_\omega; \omega \in \Omega\}=\{g_e;~e\in \Sigma\}$;

$(ii)$ The attractors $T_{\lambda}$ are the closures of $\tilde T_{\lambda}$;

$(iii)$ $T_\lambda \setminus P_{\lambda}$ are connected.
\end{theo}

\begin{proof} First, we construct a GIFS satisfying (i) and (ii).

 We claim that if $e\in \Sigma_{\ell,\delta}$, then we have
$$g_e(M_\delta \backslash P_\delta)\subset (M_\ell \setminus P_\ell).$$
For otherwise,
there exists $x\in M_\delta \backslash P_\delta $ and $p\in P_\ell$ such that  $g_e(x)=p$, then
$x=g_e^{-1}(p)\in P_\delta$, which is a contradiction. Our claim is proved.

  Take  $e\in \Sigma_{\ell,\delta}$.  Take $\lambda'\in \Lambda_{\delta}$ and $\lambda\in \Lambda_\ell$.
First, by the claim above, we have
$$
g_e(\tilde T_{\lambda'})\subset M_\ell\setminus P_\ell=\bigcup_{\lambda\in \Lambda_\ell}\tilde T_{\lambda}.
$$
Since  $g_e(\tilde T_{\lambda'})$
 is connected and
$\tilde T_{\lambda}$ is a component of $M_\ell\setminus P_\ell$, we conclude that
  either  $g_e(\tilde T_{\lambda'})\subset   \tilde T_{\lambda}$, or they are disjoint.
Using this fact, we construct a graph  $ (\Lambda, \Omega)$, where $\Lambda$ is the vertex set and $\Omega$ is the edge set,
 as following:
  \begin{equation}\label{eq:GIFS2}
 \Omega=\{([\ell,j], g, [\delta,i])\in \Lambda\times {\mathcal F}_\Sigma \times \Lambda;~ (\ell, g, \delta)\in \Sigma \text{ and }g(\tilde T_{[\delta,i]})\subset \tilde T_{[\ell,j]}\},
 \end{equation}
 where ${\mathcal F}_{\Sigma}=\{g_e; \ e\in \Sigma\}$.
Namely,  if $e=(\ell, g, \delta)\in \Sigma$ and  $g_e(\tilde T_{[\delta,i]})\subset \tilde T_{[\ell,j]}$, then we define an edge
 $u=([\ell,j], g_e, [\delta,i])$ with origin $[\ell,j]$ and  terminus $[t(e),i]$.

Setting $G:\Omega\to \{g_e;~e\in \Sigma\}$ by  $G(u)=h_u=g$ for an edge $u=(\lambda, g, \lambda')$,
we obtain an GIFS which we denote by
$${\mathcal H}=(G_\Omega, T_\Lambda).$$

 Let $T_{\lambda}$ be the closure of $\tilde T_{\lambda}$, $\lambda\in \Lambda$.
Next, we show that $(T_{\lambda})_{\lambda\in \Lambda}$ is the attractor of  $G_{\Omega}$.

   Set
$$
U_\lambda= \bigcup_{\lambda'\in \Lambda}\bigcup_{u\in \Omega_{\lambda, \lambda'}}h_u(\tilde T_{\lambda'}).
$$
Then each $\tilde T_{\lambda},$ $\lambda\in \Lambda$, belongs to one and only one $U_\lambda$.
Clearly  $\text{cl}(U_\lambda)\subset T_\lambda$ since $U_\lambda\subset \tilde T_\lambda$.

Since
$$
M_\ell=\bigcup_{\delta\in \cL}\bigcup_{e\in \Sigma_{\ell,\delta}}g_e(M_\delta)=\bigcup_{\delta\in \cL}\bigcup_{e\in \Sigma_{\ell,\delta}} \bigcup_{\lambda\in \Lambda_\delta}g_e(\tilde T_\lambda)\cup \bigcup_{\delta\in \cL}\bigcup_{e\in \Sigma_{\ell,\delta}} g_e(P_\delta)
=  \bigcup_{\lambda\in \Lambda_\ell} U_\lambda\cup   \bigcup_{\delta\in \cL}\bigcup_{e\in \Sigma_{\ell,\delta}} g_e(P_\delta),
$$
we obtain that
\begin{equation}\label{eq:cl-1}
\bigcup_{\lambda\in \Lambda_\ell} \text{cl}(U_\lambda)=\text{cl}\left(\bigcup_{\lambda\in \Lambda_\ell} U_\lambda \right )=M_\ell.
\end{equation}
That $\tilde T_\lambda$, $\lambda\in \Lambda_\ell$  are connected components of $M_\ell\setminus P_\ell$ implies that
for any $\lambda_0\in \Lambda_\ell$, it holds that
\begin{equation}\label{eq:cl-2}
\left ( \bigcup_{\lambda\in \Lambda_\ell\setminus \{\lambda_0\}} \text{cl}(U_\lambda)\right )\cap \tilde T_{\lambda_0}
\subset \left ( \bigcup_{\lambda\in \Lambda_\ell\setminus \{\lambda_0\}} T_\lambda \right )\cap \tilde T_{\lambda_0}
 =\emptyset.
\end{equation}
This together with \eqref{eq:cl-1} imply that
$\tilde T_{\lambda_0} \subset \text{closure}(U_{\lambda_0}),$ so we have $T_{\lambda_0}=\text{closure}(U_{\lambda_0})$.
Then
$$
T_{\lambda} = \text{closure}~ (U_{\lambda})
 =\bigcup_{\lambda'\in \Lambda}\bigcup_{u\in \Omega_{\lambda, \lambda'}}   h_u(T_{\lambda'}),
$$
 which proves item (ii). Item (i) is obviously true.

Thirdly, we prove that the GIFS defined by \eqref{eq:GIFS2} is p.c.f..
Since $((g_e)_{e\in\Sigma}, (M_\ell)_{\ell\in \cL})$ is a p.c.f. GIFS, there exists a constant $N_0$ such that each point
 in $M_\ell$ has at most $N_0$ codings.

  We assert that the number of codings of any point $x\in T_\lambda$ is no larger than $N_0\cdot (\#\Lambda)$.
  First, for any $k\geq 1$ and $x\in M_\ell$, there are  at most $N_0$ number of  $I\in \Sigma_\ell^k$ such that $x\in M_I$.
  Notice that
  $$
M_\ell=\bigcup_{\delta\in \cL}\bigcup_{\be\in \Sigma^k_{\ell,\delta}}g_{\be}(M_\delta)=\bigcup_{\delta\in \cL}\bigcup_{\be\in \Sigma^k_{\ell,\delta}} \bigcup_{\lambda\in \Lambda_\delta}g_{\be}( T_\lambda).
$$
  The right hand sides of the above equation provides exactly all the $k$-cylinders of ${\mathcal H}$ in $M_\ell$.
   Therefore, if $x\in h_{\bomega}(T_\lambda)=g_{\be}(T_\lambda)$,
    then there are at most $N_0$ choices of $\be$ and $\#\Lambda$ choices of $\lambda$.
 The assertion is proved.

 Pick $\ell\in \cL$. We claim that for each $\lambda\in \Lambda_\ell$,
 \begin{equation}\label{eq:PK-PG2}
 C_{\lambda}(G_\Omega)\subset C_\ell(G_\Sigma)\cup \bigcup_{e\in\Sigma}g_e(P_{t(e)}(G_\Sigma)) \quad \text{ and } \quad  P_{\lambda}(G_\Omega)\subset P_\ell(G_\Sigma).
 \end{equation}
 Let $x\in C_{\lambda}(G_\Omega)$. Then $x\in h_{u_1}(T_{\lambda_1})\cap h_{u_2}(T_{\lambda_2})$ where $u_1, u_2\in\Omega_\lambda$.
 Write
 $h_{u_1}=g_{e_1}$ and $h_{u_2}=g_{e_2}.$

 If $e_1\neq e_2$, then $x\in M_{e_1}\cap M_{e_2}\subset C_\ell(G_\Sigma)$.

 If $e_1=e_2$, then $g^{-1}_{e_1}(x)\in T_{\lambda_1}\cap T_{\lambda_2}$, where $T_{\lambda_1}$ and $T_{\lambda_2}$ are members of the partition of
 $M_{\ell_0}$ where $\ell_0$ is the terminus of $e_1$. It follows that
 $$T_{\lambda_1}\cap T_{\lambda_2}\subset P_{\ell_0}(G_\Sigma),$$
  which implies $x\in g_e(P(G_\Sigma))$. The first assertion of \eqref{eq:PK-PG2} is proved.

 By $h_u\in \{g_e;~e\in \Sigma\}$ and the first assertion, we deduce that $P_{\ell_j}(G_\Omega)\subset P(G_\Sigma)$.
 Our claim is proved.

 Let $p\in P_\lambda(G_\Omega)$, then by definition, there exists $\bu\in \Omega^*_{\lambda_0,\lambda}$
 such that $h_{\bu}(p)\in  C_{\lambda_0}(G_\Omega)$.
 Write $h_{\bu}=g_{\be}$ and let $\lambda_0\in \Lambda_{\ell_0}$, then
 by the first assertion we just proved,  we have
  $$g_{\be}(p):=q\in C_{\ell_0}(G_\Sigma)\cup \bigcup_{e'\in\Sigma}g_e(P_{t(e')}(G_\Sigma)).$$

  If $q\in C_{\ell_0}(G_\Sigma)$, then $p=g_{\be}^{-1}(q)$ and it belongs to $P_\ell(G_\Sigma)$.

  If $q$ belongs to the second term of the above union,  then we have
    $g_{\be}(p)=g_{e'}(p')$   for some $e'\in \Sigma$ and $p'\in P_{t(e')}(G_\Sigma)$.
  If $e_1\neq e'$, then $q\in C_{\ell_0}(G_\Sigma)$ and we come back to the first scenario, otherwise,
  we have $p=g^{-1}_{e_2\dots e_k}(p')$ and still $p$ belongs to  $P_\ell(G_\Sigma)$.
 Our claim is proved.

 Therefore, for all $\lambda\in \Lambda$, we have
 $$
 \#{\mathcal P}_{\lambda}\leq N_0(\# \Lambda)(\# P(G_\Omega))<\infty,
 $$
 which means $G_\Omega$ is a p.c.f. GIFS.

 (iii)
By \eqref{eq:PK-PG2},    for every  $\lambda\in \Lambda_\ell$,
 $P_{\lambda} \subset P_{\ell}$, consequently, $T_{\lambda}\setminus P_{\lambda}=\tilde T_{\lambda}$ and it is connected.
\end{proof}

\begin{theo}\label{thm:GIFS-angle-GIFS} If $((g_e)_{e\in\Sigma}, (M_\ell)_{\ell\in \cL})$ is a nearly self-similar p.c.f. GIFS such that $M_\ell$ are connected,
then the GIFS \eqref{eq:transfer-gifs} satisfies the angle separation condition provided $((g_e)_{e\in\Sigma}, (M_\ell)_{\ell\in \cL})$ does.
\end{theo}

\begin{proof}
 Let $C_0$ be a deformation bound for the nearly self-similar p.c.f. GIFS $((g_e)_{e\in \Sigma},(M_\ell)_{\ell\in \cL})$.
Let $n_0$ be an integer such that the GIFS $G_{\Sigma^{n_0}}$ has sparse post-critical point set, where $G_{\Sigma^{n_0}}$ is
 the $n_0$-th iteration  of $G_\Sigma$.
Let
$$
\epsilon_0=\min_{\ell\in \cL} \min\{\text{dist}_\ell(T_{\bomega}, T_{\btau});~\bomega, \btau\in \Omega^*_{(l)}, \ |\bomega|=|\btau|=n_0 \text{ and } T_{\bomega}\cap T_{\btau}=\emptyset\}.
$$
Let $\theta_0'$ be the constant in the angle separation condition of $((g_e)_{e\in\Sigma}, (M_\ell)_{\ell\in \cL})$.

 Pick $\ell\in \cL$ and $\lambda\in \Lambda_\ell$. We need to show that there is a constant $\theta_0$     such that
 for any   $\bomega, \btau\in \Omega^*_{\lambda}$ with
$T_{\bomega} \cap T_{\btau'}=\{z\}$, it holds that
\begin{equation}\label{eq:angle-GIFS2}
d_\ell(x,y)\geq \theta_0\max\{d_\ell(x,z), d_\ell(y,z)\} \quad \text{ for all $x\in T_{\bomega}$, $y\in T_{\btau}$.}
\end{equation}
Without loss of generality, we may assume that any sub-cylinder of  $T_{\bomega}$ can not contain both $x$ and $z$, otherwise
we  replace $T_{\bomega}$ by this sub-cylinder. Similarly, we assume that $T_{\btau}$ is the smallest cylinder
containing both $y$ and $z$. Denote $n=|\bomega|$ and $m=|\btau|$.

Let $(\omega_k)_{k\geq 1}$ be a coding of $x$ in $G_\Omega$ initialized by $\bomega$.
  For $k\geq 1$, let $(\alpha_k)_{k\geq 1}\in \Sigma^\infty$  be the simplification of $(\omega_k)_{k\geq 1}$;
  consequently,
$h_{\omega_k}=g_{\alpha_k}$.
 Similar, let $(\tau_k)_{k\geq 1}$ be a coding of $y$ in $G_\Omega$ initialled by $\btau$, and    let $(\beta_k)_{k\geq 1}$
 its simplification.

 Let $\bi=(\alpha_1\dots \alpha_n)\wedge (\beta_1\dots \beta_m)$. Denote $q=|\bi|$.
 Then (we set $g_{\varepsilon}=id$ when $\alpha_1\neq \beta_1$)
$$T_{\omega_{q+1}\dots \omega_n}=h_{\omega_1\dots\omega_q}^{-1}(T_{\bomega})=g_{\bi}^{-1}(T_{\bomega});$$
similarly, $T_{\tau_{q+1}\dots \tau_m}=g_{\bi}^{-1}(T_{\btau'})$. Therefore,
$$T_{\omega_{q+1}\dots \omega_n}\cap T_{\tau_{q+1}\dots \tau_m}=\{g_{\bi}^{-1}(z)\}.$$
  Denote $(x',y',z')=g_\bi^{-1}(x,y,z)$,
$\balpha'=\alpha_{q+1}\dots \alpha_n$ and $\bbeta'=\beta_{q+1}\dots \beta_m$.
Then $$x'\in T_{\omega_{q+1}\dots\omega_n} \subset M_{\balpha'} \text{ and }
 y'\in T_{\tau_{q+1}\dots \tau_m} \subset M_{\bbeta'}.$$
  Let $\ell'$ be the origin state of $\bi$.

\textit{Case 1.} $\max\{|\balpha'|, |\bbeta'|\}\geq n_0$.

Suppose that $|\balpha'|\geq n_0$  without loss of generality.
Then
$M_{\balpha'}\cap M_{\bbeta'}=\{z\}$, since    $\alpha_{q+1}\neq \beta_{q+1}$ and
$(G_{\Sigma^{n_0}}, K)$ is sparse.
Since $G_\Sigma$ satisfies the angle separation condition,we have
$$
d_{\ell'}(x',y')\geq \theta'_0\max\{d_{\ell'}(x',z'), d_{t(\bi)}(y',z')\}.
$$
So by the nearly self-similarity of  $G_\Sigma$, we have
$$
d_\ell(x,y)\geq C_0^{-2}\theta'_0\max\{d_\ell(x,z), d_\ell(y,z)\}.
$$

 \textit{Case 2.}  $\max\{|\balpha'|, |\bbeta'|\}< n_0$.

 By the maximality of $n$, we have $z\not\in T_{\omega_{1}\dots  \omega_{n+1}}$, so
 $z'\not\in T_{\omega_{q+1} \dots\omega_{n+1}}$.
 Similarly, $z'\not\in T_{\beta_{q+1} \dots \beta_{m+1}}$.
 It follows that  $ T_{\omega_{q+1} \dots \omega_{n+1}} $ and $T_{\beta_{q+1} \dots \beta_{m+1}}$ are disjoint, and hence
 the distance between them is no less than $\epsilon_0$.
 Therefore,
  $$
  d_{\ell'}(x',y')\geq \epsilon_0 \geq \frac{\epsilon_0}{\max_{i\in \cL}|M_i|} \max\{d_{\ell'}(x',z'), d_{\ell'}(y',z')\}
  $$
  and it follows that
  $$
  d_\ell(x,y)\geq  \frac{\epsilon_0}{C_0^2 \max_{i\in \cL}|M_i|}\max\{d_\ell(x,z), d_\ell(y,z)\}.
  $$
  The theorem is proved.
\end{proof}

\begin{proof}[\textbf{Proof of Theorem \ref{main-4}.}] Let
${\mathcal G}=((g_e)_{e\in\Sigma}, (M_\ell, d_\ell)_{\ell\in \cL})$ be a nearly self-similar p.c.f. GIFS  such that  $M_\ell$ are connected,
and the GIFS satisfies the angle separation condition.
Let $(P_\ell)_{\ell\in\cL}$ be the post-critical sets of the GIFS.

 Let $\cL=\cL'\cup \cL''$
where  $\cL'$ is the collection of non-isolated states, and $\cL''$ is the collection of isolated states.
 Let $\Sigma'=\Sigma|_{\cL'}$ be the restriction
of the graph $\Sigma$ on $\cL'$. Then
$\Sigma=\Sigma'\cup \Sigma''$
where $\Sigma''=\bigcup_{\ell\in \cL''} \Sigma_{\ell}$. Let us consider the GIFS
$$
\cG'=(G_{\Sigma'}, (M_\ell, d_\ell)_{\ell\in \cL'}).
$$
 Let $\tilde T_1,\dots, \tilde T_{m_\ell}$ be the   connected components of $M_\ell\backslash P_\ell$, $\ell\in \cL'$.
 Let $T_1,\dots, T_{m_\ell}$ be their closures respectively.
  By Theorem \ref{thm:transfer-GIFS} and Theorem \ref{thm:GIFS-angle-GIFS},
 there exists a spanning GIFS
 \begin{equation}\label{eq:transfer-gifs14}
H=((h_u)_{u\in \Omega}, (T_{\lambda}, \Gamma_{\lambda})_{\lambda\in\Lambda})
\end{equation}
 such that  $H$ is nearly self-similar,  satisfies the angle separation condition, and $T_\lambda\setminus P_\lambda$'s are connected.

For $\ell\in \cL''$,  in the set equation of $M_\ell$, we replace each $M_\delta$ by $\bigcup_{\lambda_\in \Lambda_\delta}T_\lambda$, then $M_\ell$ is a union of copies of $T_\lambda$'s. To record this information, we
define edges from $\ell\in \cL'$ to $\Lambda$ as
 $$
 \Omega_1=\bigcup_{\ell\in \cL''} \{(\ell, f, \lambda); (e,f,\delta)\in \Sigma_{\ell,\delta} \text{ and } \lambda\in \Lambda_\delta\}.
 $$

 Finally, for $\ell\in \cL'$, we have
 $$
 M_\ell=\bigcup_{\lambda\in \Lambda_\ell} T_\lambda.
 $$
 To record these information, we define edges from $\ell\in \cL'$ to $\lambda\in \Lambda$ as
 $$
 \Omega_2=\bigcup_{\ell\in \cL'} \{(\ell, id, \lambda);  \lambda\in \Lambda_\ell\}.
 $$

 Now, set
 $$
 H^*=\left ((h_u)_{u\in \Omega\cup\Omega_1\cup \Omega_2}, (T_\lambda)_{\lambda\in \Lambda}\cup (M_{\ell''})_{\ell''\in \cL''}\cup (M_{\ell'})_{\ell'\in \cL'} \right ).
 $$
 Clearly, $H^*$ is nearly self-similar and satisfies the angle separation condition.
 Since $\cL$ is the collection of isolated states in $H^*$, and $\Lambda$ is the collection not-isolated states,
 we conclude that for any $\delta\in \cL\cup \Lambda$, the attractor corresponding to $\delta$ set minus the
post-critical points corresponding to $\delta$ is connected.
 Therefore, by Theorem  \ref{thm:KP-GIFS},
$$\dim_C T_{\lambda}=1, \quad  \text{ for all } \lambda\in\Lambda,$$
and
$$\dim_C M_\ell=1, \quad \text{ for all }\ell\in \cL.$$
The theorem is proven.
\end{proof}

\begin{rema}\rm  If we insist that the GIFS be contractive, we may replace the set equations $M_\ell=\bigcup_{\lambda\in \Lambda_\ell} T_\lambda$
by
$$
g_I(M_\ell)=\bigcup_{\lambda\in \Lambda_\ell} g_I(T_\lambda),
$$
where $I\in \Sigma^*$ is a word terminating at the state $\ell$.
\end{rema}

%
%
%

\end{document}